\documentclass{article}
\usepackage{times,pifont,graphicx,color,amsmath,amsthm,amsfonts,amssymb,mathtools,url,natbib}
\usepackage{caption}
\theoremstyle{theorem}
\newtheorem{theorem}{Theorem}
\newtheorem{lemma}[theorem]{Lemma}
\newtheorem{proposition}[theorem]{Proposition}
\newtheorem{corollary}[theorem]{Corollary}
\newtheorem{problem}[theorem]{Problem}
\theoremstyle{definition}

\def\ed{\stackrel{d}{=}}
\def\led{\stackrel{d}{\le}}
\def\lecx{\stackrel{cx}{\le}}
\def\bitem{$\bullet\,\,$}
\def\ootimes{\otimes}
\def\convd{\stackrel{d}{\rightarrow}}

\def\Pbul{P}
\def\Rbul{R}

\def\Pdecbul{{P^\downarrow}}
\def\Pst {P^*}
\def\np {\#_{\rm part}}
\def\Larea {A_{\mbox{\footnotesize L\'evy}}}
\def\arccot{{\rm arccot}}
\def\Qst {Q^*}
\def\Qdecbul{{Q^\downarrow}}

\def\E{ {\mathbb E }}

\def\tilX{ {\widetilde{X}}}

\def\dbul{ \bullet \bullet }
\def\bul{ \bullet  }

\def\tilY{ {\widetilde{Y}}}
\def\tilZ{ {\widetilde{Z}}}
\def\maxG{  G^{\rm max}}

\def\tilY{ {\widetilde{Y}}}

\def\P{ {\mathbb P }}

\def\eps{ \varepsilon }

\def\hf{ \mbox{$\frac{1}{2}$} }

\def\reals{ {\mathbb R }}
\def\TZS{ S }
\def\thovpi{ \frac{\theta}{\pi} }
\def\Treal{ T}
\def\complex{ {\mathbb C }}
\def\nats{ {\mathbb N }}
\def\alth{ { \alpha,\theta } }

\def\thoval { { \frac{\theta}{\alpha} } }

\def\giv {\,|\,}
\def\Pdec{ P^{\downarrow}}
\def\Rdec{ R^{\downarrow}}
\def\Adec{ A^{\downarrow}}

\def\ndec{ n^{\downarrow}}

\def\Nex{ N^{\rm ex} }
\def\Nexn{ N^{\rm ex}_{\bullet:n }}
\def\pex{ p^{\rm ex} }

\newcommand{\alpi}{{\alpha \pi}}

\newcommand{\nnn}{ {\bf n}}
\newcommand{\Qdec}{Q^{\downarrow}}
\newcommand*\co[1]{%
   \hbox{%
     \vbox{%
       \hrule height 0.5pt 
       \kern0.5ex
       \hbox{%
         \kern-0.1em
         \ensuremath{#1}%
         \kern-0.1em
       }%
     }%
   }%
} 

\newcommand{\FF} {\mathcal F }
\newcommand{\RR} {\mathcal R }
\newcommand{\PP} {\mathcal P }
\newcommand{\GG} {\mathcal G }

\renewcommand{\SS}{ \mathcal{S}}

\newcommand{\BB}{ {B^{\rm br}} }
\newcommand{\zz}{ {p}}
\newcommand{\Abr} {A^{\rm br}_1 }

\begin{document}

\title{Random weighted averages, partition structures and generalized arcsine laws}
\author{Jim Pitman, Dept. Statistics, Evans Hall, U.C. Berkeley, CA 94720}

\maketitle


\begin{abstract}
This article offers a simplified approach to the distribution theory of {\em randomly weighted averages} or {\em $\Pbul$-means} $M_P(X):= \sum_{j} X_j  P_j$, for
a sequence of i.i.d.random variables $X, X_1, X_2, \ldots$, and independent random weights 
$P:= (P_j)$ with $P_j \ge 0$ and $\sum_{j} P_j = 1$.
The collection of distributions of $M_P(X)$, indexed by distributions of $X$, 
is shown to encode Kingman's partition structure derived from $P$.
For instance, if $X_\zz$ has Bernoulli$(\zz)$ distribution on $\{0,1\}$, the $n$th moment of $M_P(X_\zz)$ 
is a polynomial function of $\zz$ which equals the probability generating function of the number $K_n$ of distinct values in a sample of size $n$ from $P$:
$\E (M_P(X_\zz))^n = \E \zz^{K_n}$. 
This elementary identity illustrates a general moment formula for $P$-means in terms of the partition structure associated
with random samples from $P$, first developed by \citet{MR1425401} 
and \citet{MR1618739}  
in terms of random permutations.
As shown by \citet{MR1691650}, 
if the partition probabilities factorize in a way characteristic of the generalized Ewens sampling
formula with two parameters $(\alpha,\theta)$, found by
\cite{MR1337249}, 
then the moment formula yields the Cauchy-Stieltjes transform of an $(\alpha,\theta)$ mean.
The analysis of these random means includes the characterization of $(0,\theta)$-means, known as Dirichlet means, 
 due to \citet{von1941distribution}, 
\citet{watson1956joint}, 
and \citet*{MR1041402}, 
and generalizations of L\'evy's arcsine law for the
time spent positive by a Brownian motion, due to  
\citet{darling49a},  
\citet{MR0094863}, 
and \citet*{MR1022918}. 
\end{abstract}

\tableofcontents

\section{Introduction}

Consider the {\em randomly weighted average} or {\em $\Pbul$-mean} of a sequence of random variables $(X_1, X_2, \ldots)$
\begin{equation}
\label{tilx}
\tilX:= \sum_{j} X_j  P_j
\end{equation}
where $P:= (P_1, P_2, \ldots)$ is a {\em random discrete distribution} 
meaning that the $P_j$ are random variables with $P_j \ge 0 $ and $\sum_j P_j = 1$ almost surely,
where $(X_1, X_2, \ldots)$ and $\Pbul$ are independent, 
and it is assumed that the series converges to a well defined limit almost surely.
This article is concerned with characterizations of the exact distribution of $\tilX$ under various assumptions on the random discrete distribution $P$ and the
sequence $(X_1, X_2, \ldots)$. Interest is focused on the case when the $X_i$ are i.i.d. copies of some basic random variable $X$. Then $\tilX$ is a well defined random
variable, called the {\em $P$-mean of $X$}, whatever the distribution of $X$ with a finite mean, and whatever the random discrete distribution $P$ independent of  the
sequence of copies of $X$.
These characterizations of the distribution of $P$-means are mostly known in some form. But the literature of random $P$-means is scattered, and the conceptual
foundations of the theory have not been as well laid as they might have been.
There has been recent interest in refined development of  the distribution theory of $P$-means
 in various settings, especially for the model of distributions of $P$ indexed by two-parameters $(\alpha,\theta)$, whose size-biased presentation is known 
as GEM$(\alpha,\theta)$ after Griffiths, Engen and McCloskey, and  whose associated partition probabilities were derived by 
\cite{MR1337249}. 
See e.g.
\citet{MR1936323}, 
\citet{MR1983542}, 
\citet{MR2060305}, 
\citet{MR2398765}, 
\citet{MR2668906,MR2676940},  
\citet{MR2529667}. 
See also
\citet{MR2564485}, 
\citet{petrov2009two}, 
\citet{MR3694590},  
\citet{MR3160565} 
for other recent applications of two-parameter model and closely related random discrete distributions, in which settings the theory of $(\alpha,\theta)$-means may be of further interest.
So it may be timely to review the foundations of the theory of random $P$-means, with special attention to $P$ governed by the $(\alpha,\theta)$ model, and references to 
the historical literature and contemporary developments.
The article is intended to be accessible even to readers unfamiliar with the theory of partition structures, and to 
provide motivation for further study of that theory and its applications to $P$-means.

The article is organized as follows.  Section \ref{sec:overview} offers an overview of 
the distribution theory of $P$-means, with pointers to the literature and following sections for details.
Section \ref{sec:basic} develops the foundations of a general distribution theory for $P$-means, essentially from scratch. 
Section \ref{sec:models} develops this theory further for some of the standard models of random discrete distributions.
The aim is to explain, as simply as possible, some of the most remarkable known results involving $P$-means, and to clarify relations between these results and
the theory of partition structures, introduced by 
\citet{MR0368264}, 
then further developed in \citet{MR1337249},  
and surveyed in \citet[Chapters 2,3,4]{MR2245368}.  
The general treatment of $P$-means in Section \ref{sec:basic} makes many connections to those sources, and 
motivates the study of partition structures as a tool for the analysis of $P$-means.

\section{Overview}
\label{sec:overview}

\subsection{Scope}
This article focuses attention on two particular instances of the general random average construction  $\tilX:= \sum_j X_j P_j$.
\begin{itemize}
\item [(i)] The $X_j $ are assumed to be independent and identically distributed (i.i.d.) copies of some basic random variable $X$, with the $X_j$ independent
of $P$. Then $\tilX$ is called the {\em $\Pbul$-mean of $X$}, typically denoted $M_\Pbul(X)$ or $\tilX_\Pbul$.
\item [(ii)] The case $\tilX:= X_1 P_1 + X_2 \co{P}_1$, with only two non-zero weights $P_1$ and $\co{P}_1:= 1 - P_1$. 
It is assumed that $P_1$ is independent of $(X_1,X_2)$. But $X_1$ and $X_2$ might be independent and not identically 
distributed, or they might have some more general joint distribution.
\end{itemize}
Of course, more general random weighting schemes are possible, and have been studied to some extent.
For instance, \citet{MR716487} treat the distribution of randomly weighted sums $\sum_i W_i X_i$ for random non-negative weights $W_i$ not
subject to any constraint on their sum, and $(X_i)$ a sequence of i.i.d. random variables independent of the weight sequence.
But the theory of the two basic kinds of random averages indicated above is already very rich. 
This theory was developed in the first instance for real valued random variables $X_j$.
But the theory extends easily to vector-valued random elements $X_i$, including random measures, as discussed in the next subsection.

Here, for a given distribution of $P$, the collection of distributions of $M_P(X)$, indexed by distributions of $X$, 
is regarded as an encoding of Kingman's partition structure derived from $P$
(Corollary \ref{crl:means:kingman}).
That is, the collection
of distributions of $\Pi_n$, the random partition of $n$ indices generated by a random sample of size $n$ from $P$.
For instance,
if $X_p$ has Bernoulli$(p)$ distribution on $\{0,1\}$, the $n$th moment of the $P$ mean of $X_p$ 
is a polynomial in $p$ of degree $n$, which is also the probability generating function of the number $K_n$ of distinct values in a sample of size $n$ from $P$:
$\E (M_P(X_p))^n = \E p^{K_n}$ (Proposition \ref{crlpgf}).
This elementary identity illustrates a general moment formula for $P$-means,
involving the exchangeable partition probability function (EPPF), which describes the distributions of $\Pi_n$
(Corollary \ref{mainthm}).
An equivalent moment formula, in terms of a random permutation whose cycles are the blocks of $\Pi_n$,
was found by \citet{MR1425401} 
for the $(0,\theta)$ model,
 and extended to general partition structures by
\citet{MR1618739}.  
As shown in Section 
\ref{sec:twomeans}, following \citet{MR1691650}, 
this moment formula leads quickly to characterizations of the distribution of $P$-means
when the EPPF factorizes in a way 
characteristic  of the 
two-parameter family of 
GEM$(\alpha,\theta)$  models defined by a stick-breaking scheme generating $P$ from suitable independent beta factors. 
Then the moment formula yields the Cauchy-Stieltjes transform of an $(\alpha,\theta)$ mean $\tilX_{\alth}$ derived from an i.i.d. sequence of copies of $X$.
The analysis of these random $(\alpha,\theta)$ means $\tilX_{\alth}$ includes the
includes the characterization of $(0,\theta)$-means, commonly known as {\em Dirichlet means}, 
 due to \citet{von1941distribution}, 
\citet{watson1956joint}, 
and \citet*{MR1041402}, 
as well as generalizations of L\'evy's arcsine law for the
time spent positive by a Brownian motion, due to  
\citet{MR0094863}, 
and \citet*{MR1022918}. 

\subsection{Random measures}
To illustrate the idea of extending $P$-means from random variables to random measures, suppose that the $X_j$ are random point masses 
$$
X_j(\bullet) := \delta_{Y_j}(\bullet) = 1(Y_j \in \bullet)
$$ 
for a sequence of i.i.d. copies $Y_j$ of a random element $Y$ with values in an abstract measurable space $(S,\SS)$,
with $\bullet$ ranging over $\SS$.
Then 
\begin{equation}
\label{sprinkle}
P(\bullet) := M_P(1(Y \in \bullet)):= \sum_j 1 ( Y_j \in \bullet) P_j
\end{equation}
is a measure-valued random $P$-mean.
This is a discrete random probability measure on $(S, \SS)$ which places an atom of mass $P_j$ at location $Y_j$ for each $j$. 
Informally, $P(\bullet)$ is a reincarnation of $P = (P_j)$ as a random discrete distribution on $(S,\SS)$ instead of the positive integers, 
obtained  by randomly sprinkling the atoms $P_j$ over $S$ according to the
distribution of $Y$. In particular, if the distribution of $Y$ is continuous, 
on the event of probability one  that there are no ties between any two $Y$-values,
the list of magnitudes of atoms of $P(\bullet)$ in non-increasing order is 
identical to the corresponding reordering $\Pdec$ of the sequence $P:= (P_j, j = 1,2, \ldots)$. 
The original random discrete distribution $P$ on positive integers, and the derived random discrete distribution $P(\bullet)$
on $(S,\SS)$, are then so similar, that using the same symbol $P$ for both of them seems justified.
The integral of a suitable real-valued $\SS$-measurable function
$g$ with respect to $P(\bullet)$ is just the $P$-mean of the real-valued random variable $g(Y)$:
\begin{equation}
\label{intSg}
\int_{S} g(s) P(ds)  = M_P( g(Y) ):= \sum_{j} g(Y_j) P_j  .
\end{equation}
Hence the analysis of random probability measures $P(\bullet)$ of the form \eqref{sprinkle} on an abstract space $(S,\SS)$ reduces
to an analysis of distributions of $P$-means $M_P(X)$ for real-valued $X = g(Y)$.
For $P$ a listing of the normalized jumps of a standard gamma process $(\gamma(r), 0 \le r \le \theta)$, that is a {\em subordinator}, or  {\em increasing process with
stationary independent increments}, with 
\begin{equation}
\label{gamdens}
\P(\gamma(r) \in dx)/dx  = \frac{1 } {\Gamma(r)}  x^{r-1} e^{-x} 1(x>0 ),
\end{equation}
formula \eqref{sprinkle} is Ferguson's \citeyearpar{MR0350949} 
construction of a {\em Dirichlet random probability measure $P(\bullet)$ on $(S,\SS)$ governed by the measure $\theta \, \P(Y \in \bullet)$
with total mass $\theta$.}
For $r,s>0$ let $\beta_{r,s}$ 
denote a random variable with the beta$(r,s)$ distribution  on $[0,1]$
\begin{equation}
\label{betadef}
\P( \beta_{r,s} \in du)/du = \frac{ \Gamma(r+s) }{\Gamma(r) \Gamma(s) } \, u^{r-1} (1-u)^{s-1} 1 (0 < u < 1 ) .
\end{equation}
Such a beta $(r,s)$ variable is conveniently constructed from the standard gamma process $(\gamma(r), r \ge 0 )$ 
by the {\em beta-gamma algebra} 
\begin{equation}
\label{betafromgam}
\beta_{r,s} :=  
\frac{ \gamma(r) }{\gamma(r+s)} = \frac{ \gamma(r) }{ \gamma(r) + \gamma'(s) }
\end{equation}
where $\gamma'(s) := \gamma(r+s) - \gamma(r) \ed \gamma(s)$ is a copy of $\gamma(s)$ that is independent of $\gamma(r)$, and 
\begin{equation}
\label{betagamindpt}
\beta_{r,s} \mbox{ and } \gamma(r+s) \mbox{ are independent.}
\end{equation}
As a consequence, for $g(s) = 1(s \in B)$ in \eqref{intSg}, so $g(X)$ has the Bernoulli$(p)$ distribution on $\{0,1\}$ for $p = \P(Y \in B)$,
the simplest Dirichlet mean \eqref{intSg} for an indicator variable has a beta distribution:
\begin{equation}
\label{dirbeta}
P(B) \ed \beta_{p \theta , q \theta } \mbox{ for } p := \P(Y \in B), \qquad q:= 1 - p .
\end{equation}
See Section \ref{sec:dirichlet} for further disussion.

Replacing the gamma process by a more general 
subordinator makes $P(\bullet)$ a {\em homogeneous normalized random measure with independent increments (HRMI)}
as studied by
\citet{MR1983542}, 
\citet{james2009posterior}. 
from the perspective of Bayesian inference for $P(\bullet)$ given a random sample of size $n$ from $P(\bullet)$.
Basic properties of $P$-means derived from normalized subordinators are developed here in Section \ref{sec:subord}.

\subsection{Splitting off the first term}
\label{sec:splitoff}

It is a key observation that the $P$-mean of an i.i.d. sequence can sometimes be expressed as a $(P_1,\co{P}_1)$-mean
by the {\em splitting off the first term}.
That is the decomposition
\begin{align}
\label{splitser}
\tilX_\Pbul &:= \sum_{j=1}^\infty  X_j  P_j \\
 &= X_1 P_1 + \tilX_\Rbul  \, \co{P}_1 \mbox{ where } \tilX_\Rbul:=  \sum_{j=1}^\infty X_{j+1} R_j 
\end{align}
with $R_j:= P_{j+1}/\co{P}_1$ the {\em residual probability sequence} defined on the event $\co{P}_1 >0 $ by first conditioning $\Pbul$ 
on $\{2,3, \ldots\}$ and then shifting back to $\{1,2, \ldots\}$. In general, the residual sequence $\Rbul$ may be dependent on $P_1$.
Then $\tilX_\Rbul$ and $P_1$ will typically not be independent, and analysis of $\tilX_\Pbul$ will be difficult.
However, 
\begin{equation}
\label{indptsplit}
\mbox{ if $P_1$ and $(R_1,R_2, \ldots)$ are independent, }
\end{equation}
then $P_1$, $X_1$  and  $\tilX_\Rbul$ are mutually independent.
So 
\begin{equation}
\label{splitser1}
\tilX_\Pbul = X_1 P_1 + \tilX_\Rbul  \co{P}_1  .
\end{equation}
The right side is the $(P_1, \co{P}_1)$-mean of $X_1$ and $\tilX_\Rbul$, with $P_1$ independent of $X_1$ and $\tilX_\Rbul$, which are independent 
but typically not identically distributed.

This basic decomposition of a $P$-mean by splitting off the first term leads naturally to discussion of $P$-means for random discrete distributions
defined by a recursive splitting of this kind, called {\em residual allocation models} or {\em stick-breaking schemes}, discussed further in Section \ref{sec:residual}.

\subsection{L\'evy's arcsine laws}
\label{sec:levyarc}

An inspirational example of splitting off the first term is provided by the work of 
\citet{MR0000919} 
on the distributions of the time $A_t$ spent positive up to time $t$, and the time $G_t$ of the last zero before time $t$,
for a standard Brownian motion $B$:
$$
A_t:= \int_0^t 1 ( B_u > 0 ) du \qquad \mbox{and} \qquad   G_t:= \max \{ 0 \le u \le t : B_u  = 0 \},
$$
See e.g.  \citet[Theorem 13.16]{MR1876169} 
for background.
To place this example in the framework of $P$-means:
\begin{itemize}
\item Let $P_1 := 1 - G_1$ be the length of the {\em meander interval}
$(G_1, 1)$. 
\item Let $X_1:= 1(B_1 >0)$ be the indicator of the event $(B_1 >0)$ with Bernoulli $(\hf)$ distribution.
\item
Let $(P_j,X_j)$ for $j \ge 2$ be an exhaustive  listing of the lengths $P_j$ of excursion intervals of $B$ away from 
$0$ on $(0, G_1)$,
with $X_j$  the indicator of the event that $B_t >0$ for $t$ in the excursion interval of length $P_j$. 
\end{itemize}
If the lengths $P_j$ for $j \ge 2$ are put in a suitable order, for instance by ranking, 
then $(X_j, j \ge 1)$ will be a sequence of i.i.d. 
copies of a Bernoulli $(\hf)$ variable $X_{\hf}$, with $(X_j, j \ge 1)$ independent of the excursion lengths $(P_j, j \ge 1)$.
Then by construction,
$$
A_1 = M_P(X_{\hf})
$$
is the $P$-mean of a Bernoulli $(\hf)$ indicator $X_{\hf}$, representing the sign of a generic excursion.
This is so for any listing $P$ of excursion lengths of $B$ on $[0,1]$ that is independent of their signs. But 
if $P_1:= 1-G_1$  puts the meander length first as above,  then
the residual sequence $(R_1, R_2, \ldots)$ is identified with the sequence of relative lengths of excursions away from zero of 
$B$ on $[0,G_1]$. But that is also the list of excursion lengths of the rescaled process $\BB := (B(u G_1)/\sqrt{G_1}, 0 \le u \le 1)$, 
with corresponding positivity indicators $(X_2, X_3, \ldots)$. 
L\'evy showed that $\BB$ is a {\em standard Brownian bridge}, equivalent in distribution to $(B_u, 0 \le u \le 1 \giv B_1 = 0)$,
and that a last exit decomposition of the path of $B$ at time $G_1$ makes the length $P_1$ of the meander interval independent of
$\BB$, hence also independent of the residual sequence $(R_1, R_2, \ldots)$ and the positivity indicators $(X_2, X_3, \ldots)$,
which are encoded in the path of $\BB$. 
Let $\Abr$ denote the total time spent positive  by this Brownian bridge $\BB$. So $\Abr \ed (A_1 | B_1 = 0)$, while also
$\Abr = \sum_{j=1}^\infty R_j X_{j+1}$ by the previous construction.
Then the last exit decomposition provides a splitting of $A_1=M_P(X)$ of the general form \eqref{splitser1}. In this instance,
\begin{equation}
\label{splitserlev}
A_1  = X_{1} P_1 + \Abr \, \co{P}_1 
\end{equation}
where on the right side 
\begin{itemize}
\item $X_{1}, P_1 $ and $\Abr$ are independent, with 
\item $X_1 = 1 (B_1 >0) \ed X_{\hf}$ a Bernoulli$(\hf)$ indicator,
\item $P_1$ the meander length, 
\item $\Abr$ the total time spent positive by $\BB$, and
\item $\co{P}_1 := 1 - P_1 = G_1$ the last exit time.
\end{itemize}

L\'evy showed the meander interval has length $P_1 \ed \beta_{\hf,\hf}$, known as the {\em arcsine law}, because
\begin{equation}
\P( \beta_{\hf,\hf} \le u ) = \frac{2}{\pi} \arcsin \sqrt{u} \qquad ( 0 \le u \le 1),
\end{equation}
while the bridge occupation time
has the uniform $[0,1]$ distribution $\Abr \ed \beta_{1,1}$. 
L\'evy then deduced from \eqref{splitserlev} that 
the unconditioned occupation time $A_1$ has the same arcsine distribution as $P_1$ and $G_1 = \co{P}_1$: 
\begin{equation}
\label{levyarcsine}
A_1 \ed P_1 \ed G_1 \ed \beta_{\hf,\hf} .
\end{equation}

\subsection{Generalized arcsine laws}

L\'evy's arcsine laws \eqref{levyarcsine} for the Brownian occupation time $A_1$, the time $G_1$ of the last zero in $[0,1]$,
and the meander length $P_1:= 1 - G_1$, and his associated uniform law for the Brownian bridge occupation times $\Abr$, have been generalized
in several different ways. One of the most far-reaching of these generalizations gives corresponding results when the
basic Brownian motion $B$ is replaced by process with exchangeable increments. Discrete time versions of these results were first
developed by \cite{MR0058893}. 
\citet[\S XII.8 Theorem 2]{MR0270403} 
gave a refined treatment, with the following formulation for a random walk $S_n:= X_1 + \cdots + X_n$ with exchangeable increments
$(X_i)$, started at $S_0:= 0$: the random number of times $\sum_{i=1}^n 1 ( S_i > 0 )$ that the walk is strictly positive up to time $n$ has the same
distribution as the random index $\min \{ 0 \le k \le n : S_k = M_n \}$ at which the walk first attains its maximum value $M_n:= \max_{0 \le k \le n } S_k$.
In the Brownian scaling limit, Sparre Andersen's identity implies the equality in distribution $A_1 \ed \maxG_1$, the last time 
in $[0,1]$ that Brownian motion attains its maximum on $[0,1]$.
That the distribution of $\maxG_1$ is arcsine was shown also by L\'evy, who then argued  that $\maxG_1 \ed G_1$, the time of the last zero of $B$ on $[0,1]$,
by virtue of his famous identity in distribution of reflecting processes
\begin{equation}
\label{levyid}
(M_t - B_t, t \ge 0 ) \ed ( |B_t|, t \ge 0 )
\end{equation}
where $M_t:= \max_{0 \le s \le t }B_s$ is the running maximum process derived from the path of $B$.

Many other generalizations of the arcsine law have been developed, typically starting from one of the many ways this distribution
arises from  Brownian motion, or from one of its many characterizations by identities in distribution or moment evaluations.
See for instance 
\citet[Theorem 15.21]{MR1876169} 
for the result that L\'evy's arcsine law \eqref{levyarcsine} extends to the occupation time $A_1$ of $(0,\infty)$ up to time $1$
for any symmetric L\'evy process $X$ with $\P(X_t = 0) = 0$ instead of $B$, with $G_1$ replaced by $\maxG_1$, the last time in $[0,1]$ that $X$ attains 
its maximum on $[0,1]$, and $P_1$ replaced by $1- \maxG_1$.
See also 
\citet{MR1410128,MR1429264,MR1714672,MR1645445},
\citet{MR1184917} 
and \citet[Chapter 8]{MR2454984} 
regarding the distribution of occupation times of Brownian motion with drift and other processes derived from Brownian motion.
See
\citet{MR1260572}, 
\citet{MR1385407}, 
\citet{MR1443955} 
for more general results on L\'evy processes,
and \cite{MR1417982}  
and \citet{MR1341555}, 
for an extension of the uniform distribution of $\Abr$ for Brownian motion to more general bridges with exchangeable increments,
and \citet{MR2266997} 
for an extension to conditioned diffusions.
\cite{watanabe1995generalized} 
gave generalized arc-sine laws for occupation times of half lines of one-dimensional diffusion processes and random walks, which 
were further developed in \cite{MR2130959} 
and \citet{MR2226630}. 
Yet another generalization of the arcsine law was proposed by \citet{lijoi2012two}.  

The focus here is on generalized arcsine laws involving the distributions of $P$-means for some random discrete distribution $P$.
The framing of L\'evy's description of the laws of the Brownian occupation times $A_1$ and $\Abr$,
 as $P$-means of a Bernoulli$(\hf)$ variable, for distributions of $P$ determined by the lengths of excursions of a Brownian motion or Brownian bridge, 
inspired the work  of
\citet*{MR1022918} and  
\citet{MR1168191}. 
These articles showed how L\'evy's analysis could be extended by consideration of the path of $(B_t, 0 \le t \le T)$ for a 
random time $T$ independent of $B$ with the standard exponential distribution of $\gamma(1)$.
For then $G_T/T \ed G_1$ by Brownian scaling, while the last exit decomposition 
at time $G_T$ breaks the path of $B$ on $[0,T]$ into two independent random fragments 
of random lengths $G_T$ and $T-G_T$ respectively.
Thus 
$$
G_1 \ed \frac{ G_T } { T }  = \frac{ G_T } { G_T + (T- G_T)}  \ed \frac{ \gamma(\hf) }{\gamma(\hf) + \gamma'(\hf) } \ed \beta_{\hf,\hf} .
$$
This realizes the instance $r = s = \hf$ of the beta-gamma algebra \eqref{betafromgam} in the path of Brownian motion stopped at the independent gamma$(1)$
distributed random time $T$. A similar subordination construction 
was exploited  earlier by \citet{MR588409}  
in their study of fluctuation theory for {L}\'evy processes by splitting at the time $\maxG_T$ of the last maximum before an independent exponential time $T$.
See \citet{MR1406564} 
and  \citet{MR3155252} 
for more recent accounts of this theory.
This involves the lengths of excursions of the L\'evy process below its running maximum process $M$. 
L\'evy recognized that for a Brownian motion $B$ his famous identity in law of processes $M-B \ed |B|$,  as in \eqref{levyid},
implied that the structure of excursions of $B$ below $M$ is identical to the structure of excursions of $|B|$ away from $0$. This leads 
from the decomposition of $M-B$ at the time $\maxG_T$ of the last zero of $M-B$ on $[0,T]$ to the corresponding decomposition for $|B|$,  discussed earlier.
The same method of subordination was exploited further
in \citet*[Proposition 21]{MR1434129}, 
in a deeper study of random discrete distributions derived from stable subordinators.

The above analysis of the $P$-mean $M_P(X)$, for an indicator variable $X = X_{\hf}$, and $P$ the list of lengths of excursions of a Brownian motion or Brownian bridge,
was generalized by \citet*{MR1022918} 
to allow any discrete distribution of $X$ with a finite number of values. That corresponds to a linear combination of occupation times of various sectors in the plane by 
Walsh's Brownian motion on a finite number of rays, whose radial part is $|B|$, and whose angular part is made by assigning each excursion of $|B|$ to the $i$th
ray with some probability $p_i$, independently  for different excursions.
The analysis up to an independent exponential time $T$ relies only on the scaling properties of $|B|$, the Poisson character of excursions of $|B|$,  
and beta-gamma algebra, all of which extend
straightforwardly to the case when $|B|$ is replaced by a Bessel process or Bessel bridge of dimension $2 - 2 \alpha$, for $0 < \alpha < 1$. 
Then $P$ becomes a list of excursion lengths of the Bessel process or bridge over $[0,1]$, while $G_T$ and $T-G_T$ become independent
gamma$(\alpha)$ and gamma$(1-\alpha)$ variables with sum $T$ that is gamma$(1)$. So the distribution of the final meander length in the stable $(\alpha)$ case is given by
\begin{equation}
\label{alphameander}
P_1 \ed \frac{ T- G_T }{ T} \ed \beta_{1-\alpha, \alpha}
\end{equation}
by another application of the beta-gamma algebra \eqref{betafromgam}.
The excursion lengths $P$ in this case are a list of lengths of intervals of the relative complement in  $[0,1]$ of the range of a stable subordinator of index $\alpha$, with
conditioning of this range to contain $1$ in the bridge case. 
In particular, for $0 < p < 1$, the $P$-mean of a Bernoulli$(p)$ indicator $X_p$ represents the
occupation time  of the positive half line for a skew Brownian motion or Bessel process, each excursion of which is positive with probability $p$ and negative
with probability $1-p$. 
The distribution of such a $P$-mean, say $M_{\alpha,0}(X_p)$, associated with a stable subordinator of index $\alpha \in (0,1)$ and a selection probability parameter $p \in (0,1)$, 
was found independently by \citet{darling49a}  
and \citet*{MR0094863}. 
Darling indicated the representation 
$$
M_{\alpha,0}(X_p) \ed T_\alpha(p)/T_\alpha(1)
$$ 
where  $(T_\alpha(s), s \ge 0 )$ is the stable subordinator with 
\begin{equation}
\label{stabal}
\E \exp(-\lambda T_\alpha(s) ) = \exp( -s \lambda^\alpha) \qquad (\lambda \ge 0 ) .
\end{equation}
Darling also presented a formula for the
cumulative distribution function of $M_{\alpha,0}(X_p)$, corresponding to the probability density
\begin{equation}
\label{darling-lamperti}
\frac{ \P( M_{\alpha,0} (X_p) \in du ) }{du } = \frac{ p q \sin (\alpha \pi ) u^{\alpha - 1} \co{u} ^{\alpha - 1}   }
{ \pi [ q^2 u ^{2 \alpha } + 2 p q u^\alpha \co{u}^\alpha \cos (\alpha \pi) + p^2 \co{u}^{2 \alpha }  ]}  
\qquad (0 < u < 1 ) 
\end{equation}
where $q:= 1-p$ and $\co{u}:= 1 - u$.
Later, \citet{zolotarev1957mellin} 
derived the corresponding formula for the density of the ratio of two independent stable$(\alpha)$ variables $T_\alpha(p)/(T_\alpha(1)-T_\alpha(p))$ 
by Mellin transform inversion. This makes a surprising connection between the stable$(\alpha)$ subordinator and the Cauchy distribution, discussed further in Section \ref{sec:transforms}.
\citet*{MR0094863} 
showed that the density of $M_{\alpha,0} (X_p)$ displayed in
\eqref{darling-lamperti} is the density of the limiting distribution of occupation times of
a recurrent Markov chain, under assumptions implying that the return time of some state is in the domain of
attraction of the stable law of index $\alpha$, and between visits to this state the chain enters some given subset of its state space with probability $p$.
Lamperti's approach was to first derive the 
the {\em Stieltjes transform}
\begin{equation}
\label{lampertistieltjes}
\E ( 1 + \lambda M_{\alpha,0}(X_p) )^{-1} =  \sum_{n=0}^\infty \E ( M_{\alpha,0}(X_p) )^n  \lambda ^n = 
\frac {q + p ( 1 + \lambda)^{\alpha- 1}} {q  + p ( 1 + \lambda)^{\alpha}}  
\end{equation}
where $q:= 1-p$.
The associated beta$(1-\alpha,\alpha)$ distribution of $P_1$ appearing in \eqref{alphameander} is also known as a generalized arcsine law. In Lamperti's setting of
a chain returning to a recurrent state, the results of 
\citet{MR0116376}, 
presented also in \citet[\S XIV.3]{MR0270403}, 
imply that Lamperti's  limit law for occupation times holds jointly with convergence in distribution of the fraction of time since last visit to 
the recurrent state to the meander length $P_1$ as in \eqref{alphameander}, along with the generalization  to this case of the distributional identity
\eqref{splitserlev}, which was exploited by \citet*{MR1022918}. 
Due to the results of Sparre Andersen 
mentioned earlier, this beta$(1-\alpha,\alpha)$ distribution also arises from random walks and L\'evy processes as both a limit distribution of scaled occupation times, and as the 
exact distribution of the occupation time of the positive half line for a limiting stable L\'evy process $X_t$ with $\P(X_t >0 ) = 1-\alpha$ for all $t$. 
But in the
context of the $(\alpha,0)$ model for $P$, this beta$(1-\alpha,\alpha)$ distribution appears either as the distribution of the length of the meander interval $P_1$, 
as in \eqref{alphameander}, or as the distribution of a size-biased pick $P_1^*$ from $P$.
See also \citet{MR1168191} 
and \cite[\S 4]{MR1478738} 
for closely related results, and
\citet*{MR2676940}  
for an authoritative recent account of further developments of Lamperti's work.

\subsection{Fisher's model for species sampling}
A parallel but independent development of closely related ideas, from the 1940's to the 1990's,
was initiated by \citet{fisher1943solo}. 
See \citet{MR1481784} for a review. 
Fisher introduced a theoretical model for species sampling, which amounts to random sampling from
the random discrete distribution $(P_1, \ldots, P_m)$ with the {\em symmetric Dirichlet distribution with $m$ parameters equal to $\theta/m$} on the $m$-simplex of 
$(P_1, \ldots, P_m)$ with $P-i \ge 0$ and $\sum_{i=1}^m P_i = 1$. 
See Section \ref{sec:dirichlet} for a quick review of basic properties of Dirichlet distributions.
Fisher showed that many features of sampling from this symmetric Dirichlet model for $P$ have simple limit distributions as $m \to \infty$ with $\theta$ fixed.
Ignoring the order of the $P_i$, the limit model may be constructed directly by supposing that the $P_i$ are the
normalized jumps of a standard gamma process on the interval $[0,\theta]$. 
That model for a random discrete distribution, called here the $(0,\theta)$ model,  was considered by \citet{mccloskey} as an instance of the 
more general model, discussed in Section \ref{sec:subord} in which the $P_i$ are the  normalized jumps of a subordinator on a fixed time interval $[0,\theta]$, 
which for a stable $(\alpha)$ subordinator corresponds to the $(\alpha,0)$ model involved in the L\'evy-Lamperti description of occupation times.
McCloskey showed that if the atoms of $P$ in the $(0,\theta)$ model are presented in the size-biased order $P^*$ of their appearance in a process of random 
sampling, then $P^*$ admits a simple stick-breaking representation by a recursive splitting like  \eqref{splitser}
with i.i.d. factors $P^*_j/(1 - P^*_1 - \cdots - P^*_{j-1}) \ed \beta_{1,\theta}$.
\citet{MR0411097} 
interpreted this GEM$(0,\theta)$ model as the limit in distribution of size-biased frequencies in Fisher's limit model.
This presentation of $(0,\theta)$ model was developed in various ways by
\citet{MR617593}, 
\citet{MR1309433}, 
and \citet{MR1387889}. 
In this model for $P = P^*$ in size-biased random order, the basic splitting \eqref{splitser1} holds with a residual sequence $R$ that is identical in law to the original sequence $P$, 
hence also $\tilX_R \ed \tilX_P$. Then \eqref{splitser1} 
becomes a characterization of the law of $\tilX_P$ by a stochastic equation which typically has a unique solution,
as discussed in
\cite{feigin1989linear}, 
\cite{MR1669737}, 
\citet{MR2177313}. 
See also \cite{MR3730518} 
for a recent review of species sampling models.

\citet{MR0350949} 
and
\citet{MR0368264} 
further developed McCloskey's model of $P$ derived from the normalized jumps of subordinator,
working instead with the ranked rearrangement $\Pdec$ of $P$ with $\Pdec_1 \ge \Pdec_2 \ge \cdots \ge 0$.
However, it is easily seen that the distribution of the $P$-mean of a sequence of i.i.d. copies of $X$ 
is unaffected by any reordering of
terms of $P$, provided the reordering is made independently of the copies of $X$. So for any random discrete distribution $P$, and any distribution of $X$, there is the equality in distribution
\begin{equation}
\label{ranksb}
M_P(X) \ed M_\Pdec (X) \ed M_{\Pst} (X)
\end{equation}
where $\Pst$ can be any random rearrangement of terms of $P$.
This {\em invariance in distribution of $P$-means under re-ordering of the atoms of $P$}
is fundamental to understanding the general theory of $P$-means.
In the analysis of $M_P(X)$ by splitting off the first term, the distribution of $M_P(X)$
is the same, no matter how the terms of $P$ may be ordered. 
But the ease of 
analysis depends on the joint distribution of $P_1$ and $(P_2, P_3, \ldots)$, 
which in turn depends critically on the ordering of terms of $P$.
Detailed study of problems  of this kind  by 
\cite{MR1387889} 
explained why
the size-biased random permutation of terms $\Pst$, first introduced by McCloskey in the setting of species sampling,
is typically more tractable than the ranked ordering used by Ferguson and Kingman.
The notation $\Pst$ will be used consistently below to indicate a size-biased ordering of terms in a random 
discrete distribution.

\subsection{The two-parameter family}
\label{sec:twoparamintro}
The articles of 
\citet{MR1156448} 
and 
\citet*{MR1434129}. 
introduced a family of random discrete distributions indexed by two-parameters $(\alpha,\theta)$, which includes
the various examples recalled above in a unified way. 
Various terminology is used for different encodings of this family of random discrete distributions and associated random partitions.
\begin{itemize}
\item The distribution of the size-biased random permutation $\Pst$ is known as GEM$(\alpha,\theta)$, after Griffiths, Engen and McCloskey,
who were among the first to study the simple stick-breaking description of this model recalled later in \eqref{gembreaks}.
\item The distribution of the corresponding ranked arrangement $\Pdec$ is known as the {\em two-parameter Poisson-Dirichlet distribution}
\citep*{MR1434129}, 
\citep{MR2663265}.  
\item The corresponding random discrete probability measure on an abstract space $(S,\SS)$, constructed as in \eqref{sprinkle} by assigning the GEM or
Poisson-Dirichlet atoms i.i.d. locations in $S$, has become known as a {\em Pitman-Yor process}. 
\citep{MR1952729}.  
\item The corresponding partition structure is governed by the sampling formula of 
\cite{MR1337249} 
which is a two parameter generalization of the Ewens sampling formula,
recently reviewed  by \cite{MR3458585}. 
\item 
The  $P$-means associated with the $(0,\theta)$ model are commonly called {\em Dirichlet means}
\citep{MR2476736}, 
\citep{MR2668906}. 
\end{itemize}
The {\em $(\alpha,\theta)$ model} refers here to this model of a random discrete distribution $P$,
whose size-biased presentation is GEM$(\alpha,\theta)$. For such a $P$ the associated $P$-mean will be called simply an {\em $(\alpha,\theta)$-mean},
with similar terminology for other attributes of the $(\alpha,\theta)$ model, such as its partition structure.

Following further work by numerous authors including
\citet*{MR1041402}, 
\citet{MR1425401} 
and \citet{MR1618739},  
a definitive formula characterizing the distribution of an $(\alpha, \theta)$ mean $\tilX_{\alth}$,
 for an arbitary distribution of a bounded or non-negative random variable $X$,
 was found by \citet{MR1691650}: 
for all $(\alpha,\theta)$ for which the model is well defined, except if $\alpha = 0$ or $\theta = 0$, the distribution of $\tilX_{\alth}$ is uniquely
determined by the generalized Cauchy-Stieltjes transform
\begin{equation}
\label{composmomalth0}
\E ( 1 + \lambda \tilX_\alth  )^{-\theta } = \left(  \E ( 1 + \lambda X )^\alpha \right)^{- \thoval}  \qquad (\alpha \ne 0, \theta \ne 0, \lambda \ge 0).
\end{equation}
Companion formulas for the $(\alpha,0)$ case with $\theta = 0$, $0 < \alpha < 1$, trace back to Lamperti for $X = X_p$ a Bernoulli$(p)$ variable, as in 
\eqref{lampertistieltjes},
while the $(0,\theta)$ case with $\alpha = 0, \theta >0$ is the case of Dirichlet means 
due to \citet{von1941distribution}, 
and \citet{watson1956joint} 
in the classical setting of mathematical statistics, involving ratios of quadratic forms of normal variables, and developed by
\citet*{MR1041402} 
and others in Ferguson's Bayesian non-parametric setting.
These formulas are all obtained as limit cases of the generic two-parameter formula \eqref{composmomalth0}, 
naturally involving exponentials and logarithms due to the basic approximations of these functions by large 
or small powers as the case may be e.g.  $e^x = \lim_{n \to \infty} (1 + x/n)^n$ and $\log x = \lim_{\alpha \downarrow 0 } (x^\alpha - 1 )/\alpha$
for $x >0$.
For $\theta = \alpha \in (0,1)$ the transform 
\eqref{composmomalth0}
was obtained earlier by
\citet{MR1022918} 
in their description of the distribution of occupation times  derived from a Brownian or Bessel bridge, by a straightforward argument
from  the perspective of Markovian excursion theory. But Tsilevich's extension of this formula to general $(\alpha,\theta)$ is not obvious from that
perspective. Rather, the simplest approach to Tsilevich's formula involves analysis of partition structure associated with $(\alpha,\theta)$ model,
as discussed in Section \ref{sec:twomeans}.

Further development  of the theory of $(\alpha,\theta)$ means was made by
\citet*{MR1879060}.  
See also the articles by James, Lijoi and coauthors, listed in the introduction, for the most refined analysis of $(\alpha,\theta)$-means by inversion of the
Cauchy-Stieltjes transform.

\section{Transforms}
\label{sec:transforms}
Typical arguments for identifying the distribution of a $P$-mean involve encoding the distribution by some kind of transform.
This section reviews some probabilistic techniques for handling such transforms, by study of some key examples related to ratios of independent stable variables.
See \citet{MR2016344} 
for further exercises with these techniques, and
\citet*{MR2676940}  
for many deeper results in this vein.

\subsection{The Talacko-Zolotarev distribution}

The following proposition was discovered  independently in different contexts by 
\citet{MR0083847}  
and \citet[Theorem 3]{zolotarev1957mellin}. 

\begin{proposition} 
\label{prop:talzol}
{\em [Talacko-Zolotarev distribution].}
Let $C$ denote a standard Cauchy variable with probability density $\P( C \in dc ) = \pi^{-1}(1+c^2) ^{-1} dc$ for $c \in \reals$, and 
\begin{equation}
\label{calpha}
C_\alpha:= - \cos (\alpi )  + \sin (\alpi )  \, C \qquad ( 0 \le \alpha \le 1 ).
\end{equation}
Let $\TZS_\alpha$ be a random variable with the conditional distribution of $\log C_\alpha$ given the event $(C_\alpha >0)$, with $P( C_\alpha > 0 ) = \alpha$:
\begin{equation}
\label{salcal}
\TZS_{\alpha} \ed (\log C_\alpha \giv  C_\alpha > 0)  \qquad ( 0 < \alpha \le 1),
\end{equation}
with $\TZS_1 = 0$ and the distribution of $\TZS_0$ defined as the limit distribution of $\TZS_\alpha$ as $\alpha \downarrow 0$.
For each fixed $\alpha$ with $0 \le \alpha < 1$, the distribution of $\TZS_\alpha$ is characterized by each of the following three descriptions,
to be evaluated for $\alpha = 0$ by continuity in $\alpha$, as detailed later in \eqref{levychar}:
\begin{itemize}
\item [(i)]
by the symmetric probability density 
\begin{equation}
\label{faldef}
\frac{ \P( \TZS_\alpha \in ds) }{ds}  = f_\alpha(s):=  \frac{ \sin \alpi }{ ( 2 \pi \alpha ) ( \cos \alpha \pi  + \cosh s )}  \qquad (s \in \reals);
\end{equation}
\item [(ii)]
by the characteristic function
\begin{equation}
\label{cfsal}
\E \exp( i \lambda \TZS_\alpha ) = \phi_\alpha(\lambda) := \frac{ \sinh \alpha \pi \lambda }{\alpha \, \sinh \pi \lambda }   \qquad (\lambda \in \reals);
\end{equation}
\item [(iii)]
by the moment generating function 
\begin{equation}
\label{mgfsal}
\E \exp(  r \TZS_\alpha ) = \E ( C_ \alpha ^r \giv C_\alpha >0 ) = \phi_\alpha( - i r ) = \frac{ \sin \alpi r }{\alpha \sin   \pi r }   \qquad (|r| < 1) .
\end{equation}
\end{itemize}
\end{proposition}
\begin{proof}
The linear change of variable \eqref{calpha} from the standard Cauchy density of $C$ makes
\begin{equation}
\label{caldensr}
\P( C_\alpha \in dx ) = \frac{\sin \alpha \pi }{\pi }  \frac{ x^{-1}  \, dx} {( x  + 2 \cos \pi \alpha  + x^{-1} )} \qquad ( x \in \reals).
\end{equation}
Restrict to $x >0$, and divide by $\P(C_\alpha >0 )$ to obtain $\P(C_\alpha \in dx \giv C_\alpha >0)$.
For $x >0$, make change of variable $s = \log x$, $ds = x^{-1} dx$, $x = e^{s}$ in \eqref{caldensr} to obtain the density 
$\P(\log C_\alpha \in ds \giv C_\alpha >0 ) = f_\alpha(s)$ as in \eqref{faldef}, with constant $2 \pi \P(C_\alpha >0 )$ in place of $(2 \pi \alpha)$.
To check $\P(C_\alpha >0 ) = \alpha$ use the standard formula
\begin{equation}
\label{tancdf}
\P(C > c ) = \frac{1}{2} - \frac{ \arctan (c ) }{\pi}  = \frac{ \arccot (c)  }{\pi}
\end{equation}
and the fact that $0 < \sin \pi \alpha < 1$ for $ 0 < \alpha < 1$, to calculate
\begin{equation}
\label{tanalcdf}
\P(C_\alpha >0 ) =  \P( C \sin \pi \alpha  > \cos \pi \alpha ) = \P( C  > \cot \pi \alpha ) = \frac{\pi \alpha}{\pi} = \alpha.
\end{equation}
This proves (i). Now (ii) and (iii) are probabilistic expressions of the classical Fourier transform
\begin{equation}
\label{keyft}
\frac{1}{2 \pi } \int_{-\infty}^\infty  \frac{ e^{i \lambda s } \, \sin \alpi  } { \cosh s  + \cos \alpi }  \, ds = \frac{ \sinh \alpi \lambda }{\sin \pi \lambda} .
\end{equation}
This Fourier transform is equivalent, by analytic continuation, and the change of variable $x = e^s$  as above, to the classical Mellin transform of a truncated Cauchy density
\begin{equation}
\label{mellintf}
\int_0^\infty  \frac{ x^ r  \, d x } { 1 + 2 x \cos \alpi + x^2 } = \frac{ \pi }{ \sin \alpi } \frac { \sin \alpi r }{\sin  \pi  r } \qquad (|r| < 1 ).
\end{equation}
\citet[Example 4, P. 119]{whittaker1927modern} 
attribute this Mellin transform to Euler, and present it to illustrate a general techique of computing Mellin transforms by calculus of residues.
This Mellin transform also appears as an exercise in complex variables in \citet[Part I, Problem 4.10]{MR0059774}. 
\citep{MR0083847} gave details of the  derivation of the Fourier transform \eqref{keyft} by contour integration.
A more elementary proof of the key Fourier transform \eqref{keyft} is indicated below.
\end{proof}

The Fourier transform \eqref{keyft} appears also in \citet[formula (21)]{zolotarev1957mellin}, 
attributed to \citet[p. 282]{gradryz51}, 
but with a typographical error (the lower limit of integration should be $-\infty$, not $0$). 
\citet[4.23]{MR2964501} 
present some of Zolotarev's results below their (4.23.4), including \eqref{keyft} with the correct range of integration, but missing a factor of $2$: the $1/\pi$ on their left side should be $1/(2 \pi)$ as in \eqref{keyft}.

\citet{MR0083847}  
regarded the family of symmetric densities $f_\alpha(s)$ for $0 \le s < 1$ as a one-parameter extension of the case  $\alpha = \hf$, with
\begin{equation}
\label{levyhalf}
f_{\hf}(s) =  \frac{1}{\pi \cosh s } ~~\longleftrightarrow ~~\phi_{\hf}(\lambda ) = \frac{1}{\cosh \pi \lambda/2 }
\end{equation}
and the limit case $\alpha = 0$ with
\begin{equation}
\label{levychar}
f_0(s) :=  \frac{1}{2 ( 1 + \cosh s )} = \frac{1}{4 \cosh^2 s/2 } ~~\longleftrightarrow ~~ \phi_0(\lambda ) =  \frac{ \pi \lambda}{ \sinh \pi \lambda }.
\end{equation}
These probability densities and their associated  characteristic functions were found  earlier by 
\citet{MR0044774} 
 in his study of the random area
\begin{equation}
\label{levyarea}
\Larea(t):= \hf \int_0^t (X_s dY_s - Y_s d X_s) 
\end{equation}
swept out by the path of two-dimensional a Brownian motion $((X_t,Y_t), t \ge 0 )$ started at $X_0 = Y_0 = 0$.
In terms of the distribution of $\TZS_\alpha$ defined by the above proposition, L\'evy proved that 
\begin{equation}
\label{levarea}
\Larea(t) \ed  \frac{ t}{\pi } S_{\hf}  \mbox{ and } (\Larea(t) \giv X_t = Y_t = 0) \ed  \frac{ t}{ 2 \pi } S_{0}  .
\end{equation}
L\'evy first derived the characteristic functions $\phi_0$ and $\phi_{\hf}$ by analysis of his area functional of planar Brownian motion.
He showed that the distributions of $\TZS_0$ and $\TZS_{\hf}$ are infinitely divisible, each associated with a symmetric pure-jump L\'evy process, whose L\'evy measure he computed.
He then inverted $\phi_0$ and $\phi_{\hf}$ to obtain the densities $f_0$ and $f_{\hf}$ displayed above by appealing to the classical infinite products for the hyperbolic functions.
L\'evy's work on Brownian areas inspired  a number of further studies, which have clarified relations between various probability distributions derived from Brownian paths whose
Laplace or Fourier transforms involve the hyperbolic functions. See 
\citet{MR886959}, 
and
\cite{MR1969794} 
for comprehensive accounts of these distributions, their associated L\'evy processes, and
several other appearances of the same Fourier transforms in the distribution theory of Brownian functionals,
and \citet[\S 0.6]{MR1725357} 
for a summary of formulas associated with the laws of $\TZS_0$ and $\TZS_{\hf}$.
Note from \eqref{cfsal}  and \eqref{levychar} that 
the characteristic function $\phi_\alpha$ of $S_\alpha$ 
is derived from $\phi_0$ by the identity
$$
\phi_0(\lambda) = \phi_0( \alpha \lambda ) \phi_\alpha(\lambda) \qquad 
( 0 \le \alpha \le 1)
$$
corresponding to the identity in distribution
$$
S_0 \ed \alpha S_0 + S_\alpha  \qquad ( 0 \le \alpha \le 1)
$$
where $S_0$ and $S_\alpha$ are assumed to be independent. That is to say, the distribution of $S_0$ is {\em self-decomposable},
as discussed further in  \citet{MR2108161}. 

An easier approach to these Fourier relations \eqref{levyhalf} and \eqref{levychar} for $\alpha = \hf$ and $\alpha = 0$,
which extends to the Fourier transform \eqref{keyft} for all $0 \le \alpha < 1$, is to recognize the distributions
involved as hitting distributions of a Brownian motion in the complex plane. The Cauchy density of $C_\alpha$ in 
\eqref{caldensr}
is well known to be the
hitting density of $X_T$ on the real axis for a complex Brownian motion $(X_t + i Y_t, t \ge 0 )$ started at  the point
on the unit semicircle in the upper half plane
$$
X_0 + i Y_0 = \cos (1 - \alpha) \pi + i \sin(1-\alpha) \pi  = - \cos \alpi + i \sin \alpi
$$
and stopped at the random time $\Treal:= \inf \{t : Y_t = 0 \}$.
Let $X_t + i Y_t = R_t \exp(i W_t )$ be the usual representation of this complex Brownian motion in polar coordinates,
with radial part $R_t$ and continuous angular winding $W_t$, starting from $R_0 = 1$ and $W_0 = (1-\alpha) \pi$.
Then by construction
$$
C_\alpha \ed X_\Treal = R_\Treal 1( W_\Treal = 0 ) - R_\Treal 1( W_\Treal = \pi ).
$$
According to L\'evy's theorem on conformal invariance of Brownian motion, the process
$(\log R_t + i W_t , 0 \le t \le \Treal)$ is a time changed complex Brownian motion $(\Phi(u) + i \Theta(u), u \ge 0)$ :
$$
\log R_t + i W_t  = \Phi(U_t) +  i \Theta (U_t) \mbox{ where } U_t:= \int_{0}^t \frac{ds} {R_s^2}
$$
and $U_{\Treal} = \inf \{u: \Theta(u)  \in \{0, \pi \} \}$.
See \citet{MR841582}  
for further details of this well known construction.  
The conclusion of the above argument is summarized by the following lemma, which combined with the 
next proposition provides a nice explanation of the basic Fourier transform \eqref{keyft}.

\begin{lemma}
The Talacko-Zolatarev distribution of $\TZS_\alpha$ introduced in Proposition \ref{prop:talzol} as the conditional 
distribution of $\log C_\alpha$ given $C_\alpha >0$ may also be represented as
\begin{equation}
\P(\TZS_\alpha  \in \bullet) = \P_{(1-\alpha)\pi }( \Phi_T \in \bullet \giv \Theta_T = 0 ) = \P_{\alpha \pi }( \Phi_T \in \bullet \giv \Theta_T = \pi )
\end{equation}
where $\P_\theta$ governs $(\Theta_t, t \ge 0)$ and $(\Phi_t, t \ge 0)$ two independent Brownian motions, started at $\Theta_0 = \theta \in (0,\pi)$ and
$\Phi_0 = 0$,
and $T:= \inf \{t : \Theta_t = 0 \mbox{ or } \pi \}$. 
\end{lemma}


\begin{proposition}
With the notation of the previous lemma, and the Talacko-Zolatarev densities and characteristic functions $f_\alpha$ and $\phi_\alpha$
defined as in Proposition \ref{prop:talzol},
the joint distribution of
$\Phi_T$ and $\Theta_T$ is determined by any one of the following three formulas, each of 
which holds jointly with a companion formula for $(\Theta = 0)$ instead of $(\Theta = \pi)$, with  
$\theta$ replaced by $\pi - \theta$ on the right side only, so $\sin \theta = \sin(\pi - \theta)$ is unchanged, and $\cos \theta$ is replaced by $\cos (\pi - \theta) = - \cos \theta$:
\begin{itemize}
\item [(i)] The density of $\Phi_T$ on the event $(\Theta_T = \pi)$ with $\P_\theta(\Theta_T = \pi) = \thovpi$ is
\begin{equation}
\label{pthform}
\frac{ \P_\theta( \Phi_T \in ds, \Theta_T = \pi) }{ ds } = \frac{ \sin \theta} { 2 \pi ( \cosh s + \cos \theta ) }  = \frac{\theta}{\pi} f_{\thovpi}(s) .
\end{equation}
\item [(ii)] The corresponding cumulative distribution function is  
\begin{equation}
\label{pthformcdf}
\P_\theta( \Phi_T \le s , \Theta_T = \pi) = \frac{1}{ 2 \pi } \left[ 1 + 2 \arctan( \tan(\theta/2) \tanh(x/2) ) \right] 
\end{equation}
\item [(iii)] The corresponding  Fourier transform is 
\begin{equation}
\label{pthfourier}
\E_\theta \,  e^{ i \lambda \Phi_T } 1 ( \Theta_T = \pi)  =  \frac{ \sinh \theta \lambda }{ \sinh \pi \lambda } = \frac{\theta}{\pi} \phi_{\thovpi}(\lambda).
\end{equation}
\end{itemize}
\end{proposition}
\begin{proof}
By the well known description of hitting probabilities for Brownian motion in terms of harmonic functions,
the $\P_\theta$ distribution of $(\Theta_T, \Phi_T)$ is the harmonic measure on the boundary of the vertical strip $\{(\theta ,s) : 0 < \theta < \pi, s \in \reals \}$
for Brownian motion with initial point  $(\theta,0)$ in the interior of the strip. Formula \eqref{pthform} is then read from the classical formula for the
Poisson kernel in the strip, which gives the hitting density on the two vertical lines. This formula is mentioned in \citet{MR1574272} 
and derived in detail by \citet{MR0132838}. 
As indicated by Widder, the formula for the Poisson kernel for the strip follows easily from the corresponding kernel for the upper half plane,
by the method of conformally mapping $\theta  + i s$ to $e^{i (\theta  + i s )} = e^{-s} e^{i \theta}$. 
This proves (i), and (ii) follows by integration. As for (iii),
it is easily seen that conditionally given $T$ and $\Theta_T$ the distribution of $\Phi_T$ is Gaussian with mean $0$ and variance $T$.
Hence
\begin{equation}
\label{knightlt}
\E _\theta \,   e^{ i \lambda \Phi_T }  1 (\Theta_T  = \pi) = \E _\theta  e^{ - \hf \lambda^2 T} 1(\Theta_T  = \pi) =  \frac{ \sinh \theta \lambda }{ \sinh \pi \lambda }
\end{equation}
where the last equality is a well known formula for one-dimensional Brownian motion
\cite[Exercise II.3.10]{MR1725357}, 
which holds because $(\exp( \pm \lambda \Theta_t - \hf \lambda^2 t ), t \ge 0)$ is a martingale for each choice of sign $\pm$ and $\lambda >0$. 
The average of these two martingales is $M_{\lambda,t}:= \sinh( \lambda \Theta_t) \exp( - \hf \lambda^2 t )$.
So $\P_\theta$  governs $(M_{\lambda,t}, t \ge 0)$ as a martingale with continuous paths which starts at $M_{\lambda,0} = \sinh (\lambda \theta )$,  
and is bounded by $0 \le M_{\lambda,t} \le \sinh \pi \lambda$ for $0 \le t \le T$. But $\sinh (0) = 0$ makes $\sinh( \lambda \Theta_T) = \sinh (\lambda \pi ) 1(\Theta_T = \pi)$, so
$$
\sinh (\lambda \theta ) =  \E M_{\lambda,0} = \E M_{\lambda,T} = \E \sinh (\lambda \pi )  e^ {- \hf \lambda^2 T } 1(\Theta_T  = \pi).
$$
As a check on \eqref{pthfourier}, its limit as $\lambda \to 0$ gives $\P_\theta( \Theta_T = \pi) = \theta/\pi$.
\end{proof}


\subsection{Laplace and Mellin transforms}

The {\em Laplace transform} of a non-negative random variable $X$,
\begin{equation}
\phi_X(\lambda): = \E e^{- \lambda X } = \int_0^\infty e^{- \lambda x } \P( X \in dx ),
\end{equation}
can always be interpreted probabilistically as follows for $\lambda \ge 0$. Let $\eps \ed \gamma(1)$ be a standard exponential
variable independent of $X$. By conditioning on $X$,
\begin{equation}
\label{eq:surviv}
\phi_X(\lambda ) = \P( \eps > \lambda X ) = \P( \eps/ X > \lambda ) \qquad (\lambda \ge 0).
\end{equation}
This basic formula presents $\phi_X(\lambda)$ as the {\em survival probability function} of the random ratio $\eps/X$, whose distribution 
 is the {\em scale mixture of exponential distributions}, with a random inverse scale parameter $X$. 
See \citet{MR2011862} 
for much more about such scale mixtures of exponentials.
This formula  \eqref{eq:surviv}
works with the convention
$\eps/X = +\infty$ if $X = 0$.
For instance, if $X = T_\alpha$ has the standard stable$(\alpha)$ law with Laplace transform \eqref{stabal}
then \eqref{eq:surviv} gives
\begin{equation}
\label{eq:survivstable}
\P( \eps/ T_\alpha  > \lambda ) = \exp( - \lambda^\alpha )
\end{equation}
and hence for $\lambda = x^{1/\alpha}$
\begin{equation}
\label{eq:survivstable1}
\P( (\eps/ T_\alpha)^\alpha   > x ) = \P( \eps/ T_\alpha > x^{1/\alpha} )  = \exp( - x).
\end{equation}
That is to say, in view of the uniqueness theorem for Laplace transforms, the standard stable$(\alpha)$ distribution of $T_\alpha$ is
uniquely characterized by the identity in law
\begin{equation}
\label{eq:shanbag}
\left( \frac{ \eps } {T_\alpha } \right) ^\alpha \ed \eps
\end{equation}
where $\eps \ed \gamma(1)$ is an exponential variable with mean $1$, independent of $T_\alpha$. Equate real moments in \eqref{eq:shanbag}
to see that the distribution of $T_\alpha$ has 
Mellin transform
\begin{equation}
\label{talmel}
\E T_\alpha ^{\alpha r } = \frac{ \Gamma(1 - r) }{ \Gamma(1 - \alpha r ) } \qquad  |r| < 1 .
\end{equation}
This provides another characterization of the standard stable$(\alpha)$ law of $T_\alpha$, by uniqueness of Mellin transforms.
This derivation of \eqref{eq:shanbag} and \eqref{talmel} is due to \citet{MR0436267}. 
A more general Mellin transform for stable laws appears much earlier in \cite[Theorem 3]{zolotarev1957mellin}. 

Consider now the ratio $R_\alpha:= T_\alpha/T'_\alpha$ of two independent standard stable$(\alpha)$ variables. Immediately from  \eqref{talmel}, the
Mellin transform of $R_\alpha^\alpha $ is
\begin{equation}
\label{talmel1}
\E R_\alpha ^{\alpha p } = 
\frac{ \Gamma(1 + p) }{ \Gamma(1 + \alpha p ) } \frac{ \Gamma(1 - p) }{ \Gamma(1 - \alpha p ) } 
=
\frac{ 1} {\alpha } \frac{ \Gamma(p) }{ \Gamma(\alpha p ) } \frac{ \Gamma(1 - p) }{ \Gamma(1 - \alpha p ) } 
=
\frac{ 1} {\alpha } \frac{ \sin p \alpha \pi }{\sin p \pi } 
\qquad  |p| < 1 
\end{equation}
by two applications of the reflection formula for the gamma function $\Gamma(1-z) \Gamma(z) = \pi/ \sin z \pi$.
Compare with 
\eqref{cfsal} 
to see the identity in distribution $R_\alpha^\alpha \ed \TZS_\alpha$ 
for $\TZS_\alpha$ as in in Proposition \ref{prop:talzol}, that is 
\begin{equation}
\label{ralphalph}
\P(R_\alpha^ \alpha \in dx ) = \frac{\sin \alpha \pi }{\alpha \pi }  \frac{dx} {( 1 + 2 x \cos \pi \alpha  + x^2 )} \qquad ( x > 0 ).
\end{equation}
Equivalently, by the change of variable $r = x^{1/\alpha}$, so $x = r^{\alpha}$, $dx = \alpha r^{\alpha - 1} dr$, 
\begin{equation}
\label{ralphalphr}
\P(R_\alpha \in dr ) = \frac{\sin \alpha \pi }{\pi}  \frac{ r^{\alpha - 1} \, dr } { ( 1 + 2 r^\alpha  \cos \pi \alpha  + r^{2 \alpha} ) } \qquad ( r > 0 ).
\end{equation}

By calculus, the 
density \eqref{ralphalphr} of $R_\alpha$ has derivative at $r >0$ which is is a strictly negative function of $r$ multiplied by
\begin{equation}
\label{zquad}
(1 + \alpha) x^2 + 2 x \cos \alpha \pi + 1 - \alpha  \mbox{ where } x:= r^\alpha .
\end{equation}
Analysis of this quadratic function of $x$ explains the qualitative features of the densities of $R_\alpha$ displayed in Figure \ref{fi:zolo} 
for selected values of $\alpha$.

\begin{figure}[h!]
\includegraphics[width=\linewidth]{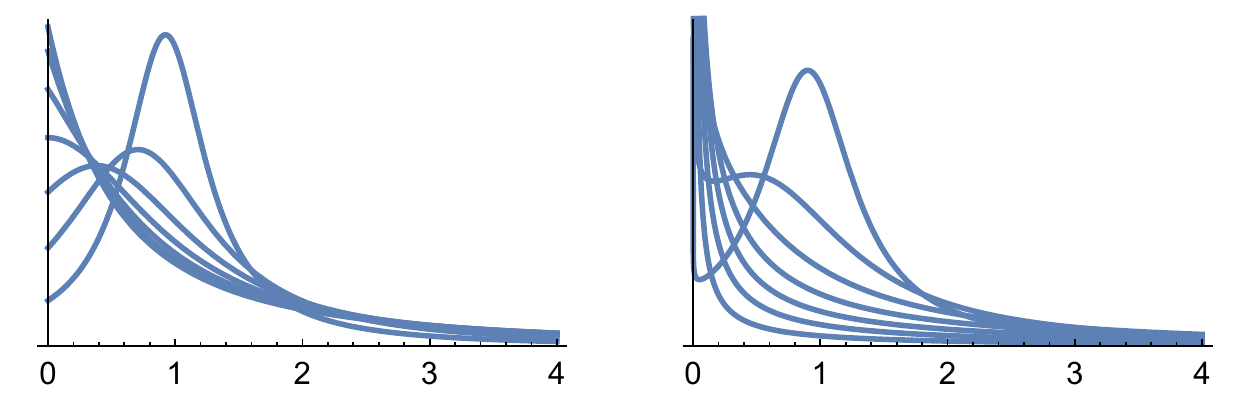}
\caption{ {\em Probability densities of $R_\alpha^\alpha$ and $R_\alpha:= T_\alpha/T'_\alpha$ for $\alpha = k/8, 1 \le k \le 7$.}
The $7$ densities of $R_\alpha^\alpha$ in the left panel are those of the scaled Cauchy 
variable $C_\alpha$ in \eqref{calpha} conditioned to be positive. The curves 
are identified by their values at $0$, which  decrease as $\alpha$ increases, and their values at $1$ which
increase with $\alpha$. The corresponding densities of $R_\alpha$ can be identified similarly in the right panel.
By unimodality of the Cauchy density, in the left panel each density of $R_\alpha^\alpha$ is unimodal, with maximum density at $0$ for $\alpha \le \hf$,
and at $\sin (\alpha - \hf)\pi$ for $\alpha \ge \hf$. Each density of $R_\alpha$ in the right panel has an infinite maximum achieved at $0+$.
The discrimant of the quadratic \eqref{zquad} is $\Delta(\alpha):= 2 (\cos^2 \alpha \pi  + \alpha^2 - 1)$ which is negative for
$\alpha \le \alpha_c$, where $\alpha_c \approx 0.736484$ is the unique root $\alpha \in (0,1)$ of the equation $\Delta(\alpha) = 0$.
So the density of $R_\alpha$ is strictly decreasing for $\alpha \le \alpha_c$, with strictly negative derivative for $\alpha < \alpha_c$,
and with a unique point of inflection for $\alpha = \alpha_c$ at $(\sqrt{ 1 - \alpha_c^2}/( 1 + \alpha_c))^{1/\alpha_c} \approx 0.278018$.
For $\alpha > \alpha_c$, as for the top two curves with $\alpha = 6/8$ and $\alpha = 7/8$, the density of $R_\alpha$ is bimodal, with a local minimum at $r_{-}(\alpha)$ and
a local maximum at $r_+(\alpha)$ where $r_{\pm} (\alpha):= (x_{\pm}(\alpha))^{1/\alpha}$ for $x_\pm(\alpha)$ the two roots in $[0,1]$ of the quadratic \eqref{zquad}.
A common feature of the laws of $R_\alpha^\alpha$ and $R_\alpha$ for all $0 < \alpha < 1$ is that each law has median $1$, due to $R_\alpha \ed R_\alpha^{-1}$,
and each law has infinite mean.
As $\alpha \uparrow 1$, both laws converge to the distribution degenerate at $1$. But as $\alpha \downarrow 0$, the behavior is different.
At each $x >0$, the density of $R_\alpha^\alpha$ converges to $(1 + x)^{-2}$, which is the density of the limit in distribution of $R_\alpha^\alpha$.
In parallel with this convergence, as $\alpha \downarrow 0$, the density of $R_\alpha$ converges pointwise to $0$, as the distribution of
$R_\alpha$ converges vaguely to an atom of $\hf$ at $0$ and an atom of $\hf$ at $+\infty$. 
This pointwise convergence of densities as $\alpha \downarrow 0$ is apparent in both panels.
}
\label{fi:zolo}
\end{figure}

\begin{figure}[h!]
\includegraphics[width=\linewidth]{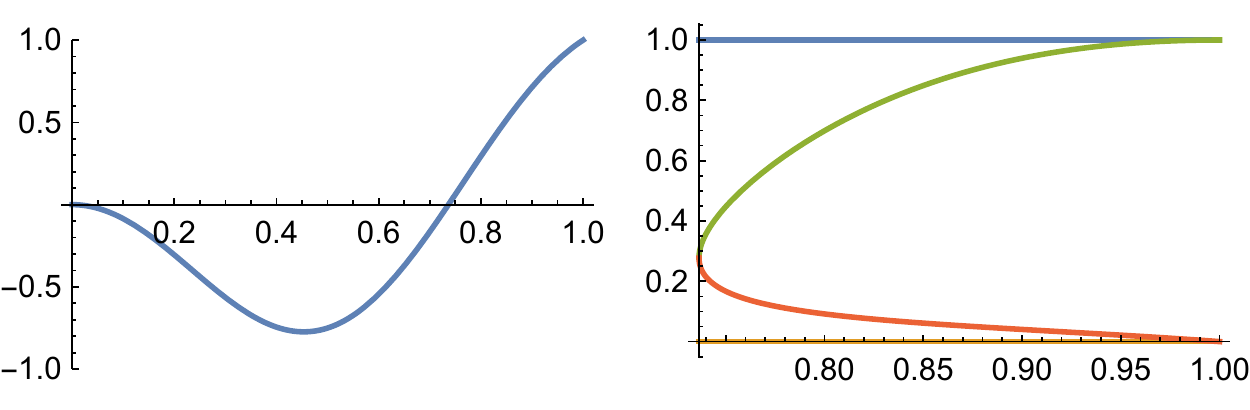}
\label{fi:zolo1}
\caption{ {\em Discrimant and locations of the minimum and maximum of the density of $R_\alpha:= T_\alpha/T'_\alpha$.}
Half the discriminant $\Delta(\alpha)$ of the quadratic equation \eqref{zquad} is $\cos^2 \alpha \pi + \alpha^2 - 1$, as plotted in the left panel,
with $\alpha_c \approx 0.736484$ the unique root of this function in $(0,1)$.
The right panel shows the two graphs of $r_{\pm} (\alpha):= (x_{\pm}(\alpha))^{1/\alpha}$ for $x_\pm(\alpha)$ the two roots in $[0,1]$ of the quadratic equation \eqref{zquad},
for $\alpha_c  \le \alpha < 1$. The lower curve $r_-(\alpha)$ gives the location of the unique minimum in $(0,1)$ of the density of $R_\alpha$. This location decreases from
$r_{\pm}(\alpha_c) \approx 0.278018$ to $0$ as $\alpha$ increases from $\alpha_c$ to $1$.
The upper curve $r_+(\alpha)$ is the location of the unique local maximum of the density $(0,\infty)$. This modal value is always less than $1$, and increases from
$r_{\pm}(\alpha_c) \approx 0.278018$ to the median value of $1$ as $\alpha$ increases from $\alpha_c$ to $1$.
}
\end{figure}

\subsection{Cauchy-Stieltjes transforms}

For a real valued random variable $X$, the {\em Cauchy-Stieltjes transform of $X$} is commonly defined to be the function of a complex variable $z$
\begin{equation}
\label{cauchy-stieltjes}
G_X(z):= \E ( z  - X )^{-1}   \qquad (z \notin \reals).
\end{equation}
There are inversion formulas both for this transform, as well as
for the {\em generalized Cauchy-Stieltjes transform of $X$ of order $\theta$}, say $G_{X,\theta}(z)$
 obtained by replacing the power $-1$ in \eqref{cauchy-stieltjes} by $-\theta$:
\begin{equation}
\label{cauchy-stieltjes-gen}
G_{X,\theta}(z):= \E ( z  - X )^{-\theta}   \qquad (z \notin \reals).
\end{equation}
See \citet{MR3484407} 
for a recent article about this transform with references to earlier work.
For $X$ with values in $[0,1]$ it is more pleasant to deal with the variant of this transform
\begin{equation}
\label{cauchy-stieltjesth}
\E ( 1  - \lambda X )^{- \theta}   = \sum_{n=0}^\infty \frac{ (\theta)_n }{n!} \E X^n \lambda ^n   =  \lambda^{-\theta} G_{X,\theta}(\lambda^{-\theta}) \qquad (|\lambda|< 1)
\end{equation}
where the series is convergent and equal to $\E ( 1  - \lambda X )^{- \theta}$ for every $|\lambda | <  1$ by dominated convergence. 
A distribution of $X$ on $[0,1]$ is uniquely determined by its moment sequence $(\E X^n ,  n = 0,1,2, \ldots)$,
hence also by its generalized Cauchy-Stieltjes transform of order $\theta$, for any fixed $\theta >0$.
For unbounded non-negative $X$, including $X$ with $\E X = \infty$, for which there is not even a partial series expansion \eqref{cauchy-stieltjesth} for $\lambda$ in any neighbourhood of $0$,
it is typically easier to  work with
\begin{equation}
\label{cauchy-stieltjesthlap}
\E ( 1  + \lambda X )^{- \theta}  =  (- \lambda)^\theta  G_{X,\theta}((- \lambda)^{-\theta}) \qquad ( \lambda \ge 0).
\end{equation}
Here the left side is evidently a well defined and analytic function of $\lambda $ with positive real part. The right side may be understood
by analytic continuation of $G_{X,\theta}(z)$ from non-real values of $z$. But arguments by analytic continuation can often be avoided by the
following key observation. By introducing $\gamma(\theta)$ with gamma$(\theta)$ distribution, independent of $X$, and conditioning on $X$, the
expectation in \eqref{cauchy-stieltjesthlap} is
\begin{equation}
\label{cauchy-stieltjesthlap2}
\E ( 1  + \lambda X )^{- \theta}  =  \E \exp [ - \lambda \gamma(\theta) X ]  \qquad (\lambda \ge 0),
\end{equation}
that is the ordinary Laplace transform of $\gamma(\theta) X$. This determines the distribution of $X$,
by uniqueness of Laplace transforms, and the the following lemma which has been frequently exploited 
\citep[p. 358]{MR1825154}, 
\citet[1.13, 4.2, 4.24]{MR2964501}, 
\cite[Theorem 3]{MR3217784}. 
As a general rule, in reading formulas involving generalized Stieltjes transforms of probability distributions
of $X$, especially $X \ge 0$, matters are often simplified by interpreting the generalized Stieltjes transform as the Laplace transform of $\gamma(\theta) X$. 

\begin{lemma}
\label{lem:gammacancel}
{\em [Cancellation of independent gamma variables]}
For random variables or random vectors $X$ and $Y$, 
and $\gamma(\theta)$ with gamma$(\theta)$ distribution independendent of both $X$ and $Y$, for each real $a$ there is the equivalence of identities in distribution 
\begin{equation}
\label{gammacancel}
\gamma(\theta) ^a X \ed \gamma(\theta)^a  Y \qquad \iff \qquad X \ed Y .
\end{equation}
\end{lemma}
\begin{proof}
Consider first the case of real random variables. Obviously $\P( \gamma(\theta) ^a X \in B) = \P(X \in B)$ if $B$ is any of the subsets $(-\infty,0)$, $\{0\}$ or $(0,\infty)$ of $\reals$.
So by conditioning it may as well be assumed that both $X$ and $Y$ are strictly positive, when there is no difficulty in taking logarithms.
It is known \citep{MR1300491} 
that the distribution of $\log \gamma(\theta)$  is infinitely divisible, hence has a characteristic function which does not vanish. 
The conclusion in the univariate case follows easily, by characteristic functions. An appeal to the Cram\'er-Wold
 theorem takes care of the multivariate case.
\end{proof}

To illustrate these ideas, let us derive the ordinary Cauchy-Stieltjes transform of the ratio $R_\alpha:= T_\alpha/T'_\alpha$ of two i.i.d. standard stable
$(\alpha)$ variables, whose Mellin transform and probability density were already indicated above.
From above, the problem is to calculate
\begin{equation}
\label{cauchy-stieltjesthlapr}
\E ( 1  + \lambda R_\alpha )^{- 1}  =  \E \exp ( - \lambda \eps T_\alpha/T'_\alpha)
\end{equation}
for independent random variables $\eps \ed \gamma(1)$ and $T_\alpha \ed T'_\alpha$.
But we already know from 
\eqref{eq:shanbag}
 that $\eps /T'_\alpha \ed \eps^{1/\alpha}$. So 
\begin{equation}
\label{cauchy-stieltjesthlapr3}
\E ( 1  + \lambda R_\alpha )^{- 1}  =  \E \exp ( - \lambda \eps^{1/\alpha }  T_\alpha ) = \E \exp ( - \lambda ^\alpha \eps ) = (1 + \lambda^\alpha)^{-1}.
\end{equation}
Thus the distribution of $R_\alpha$ is uniquely characterized by the simple Cauchy-Stieltjes transform
\begin{equation}
\label{cauchy-stieltjesthlapr4}
\E ( 1  + \lambda R_\alpha )^{- 1}  =  (1 + \lambda^\alpha)^{-1} \qquad (\lambda > 0).
\end{equation}
It is notable that the explicit 
formula
\eqref{ralphalphr}
for the density of $R_\alpha$ with Laplace-Stieltjes transform $(1 + \lambda^\alpha)^{-1}$ is much
simpler than the corresponding inversion for the common distribution of $\eps R_\alpha \ed \eps R_\alpha^{-1} \ed \eps^{1/\alpha }  T_\alpha $ which has 
$(1 + \lambda^\alpha)^{-1} $ as its ordinary Laplace transform:
\begin{equation}
\label{mlfunction}
\P( \eps R_\alpha > x ) = E_\alpha( - x^\alpha) \qquad ( x \ge 0 )
\end{equation}
where
$$
E_\alpha(z):= \sum_{k=0}^\infty \frac{ z ^ k } { \Gamma(k \alpha + 1)} \qquad  (z \in \complex)
$$
is the classical {\em Mittag-Leffler function} with parameter $\alpha$. This is an entire function of $z \in \complex$, for each $\alpha \in \complex$ with
strictly positive real part, with $\alpha \in (0,1)$ here.
This formula was found by \citet{MR1054728}. 
See also \cite[(3.9) and (4.37)]{MR1829592} 
for closely related transforms, and \citet{MR3244285} 
for a recent survey of Mittag-Leffler functions and their applications.
Compare also with the density of $T_\alpha$, given by \citet{MR0018286} 
\begin{equation}
\label{pollard}
\P( T_\alpha \in dt )/dt = \frac{1} {\pi} \sum_{k = 0}^\infty \sin (\alpha k \pi ) \frac{ (-1)^{k+1} \Gamma ( k \alpha + 1 ) }{ k! \, t^{k \alpha + 1 }  } .
\end{equation}
Only for $\alpha = \hf$, when $T_{\hf} \ed 1/(2 \gamma(\hf))$ is there substantial simplification of this series formula. But see
\cite{MR2740992} 
for explicit expressions for the density \eqref{pollard} in terms of the Meijer $G$ function for rational $\alpha$, and
\cite{MR870201} 
for a general representation of stable densities in terms of Fox functions.
See also \citet{ho2007gibbs}. 

Returning to the context of  random discrete distributions, if $P_{\alpha,0}$ is governed by the $(\alpha,0)$ model defined by
normalizing the jumps of a stable$(\alpha)$ subordinator on some fixed interval of length say $s >0$, then it is evident that
for $X = X_p$ the indicator of an event of probability $p$, the distribution of the $P_{\alpha,0}$ mean of $X_p$ is determined by
\begin{equation}
\label{malthzero1}
M_{\alpha,0}(X_p)  \ed \frac{ T_\alpha(p) }{ T_\alpha(1)}  \ed \frac{    p^{1/\alpha} T_\alpha }{ p^{1/\alpha} T_\alpha + q^{1/\alpha} T'_\alpha } \ed \frac{1} {1 + c R_\alpha } 
\end{equation}
where $(T_\alpha(s), s \ge 0 )$ is the stable$(\alpha)$ subordinator with $T_\alpha(s) \ed s^{1/\alpha} T_\alpha$ for 
$T_\alpha$ the standard stable$(\alpha)$ variable as above, and $c:= (q/p)^{1/\alpha}$ for $q := 1 - p$.
Here the second $\ed$ appeals to the decomposition of $T_\alpha(1)$ into two independent components $T_\alpha(1) = T_\alpha(p) + (T_\alpha(1) - T_\alpha(p) )$
with $T_\alpha(p) \ed p^{1/\alpha} T_\alpha$ and $T_\alpha(1) - T_\alpha(p) \ed q^{1/\alpha} T_\alpha$.
The distribution of $M_{\alpha,0}(X_p)$ is thus obtained from that of $R_\alpha$ by a simple change of variable.
Moreover, for any real $X$, the identity
$$
\left( 1 + \frac{ \lambda } { 1 + c X } \right)^{-1} = 1 - \frac{\lambda}{(1 + \lambda)} \left( 1 + \frac{c X}{(1 + \lambda)} \right)^{-1}
$$
allows the Cauchy-Stieltjes transform of $(1 + c X )^{-1}$ to be expressed directly in terms of that of $X$. In particular, for the ratio of independent
stable variables $X = R_\alpha$ with the simple Cauchy-Stieltjes transform \eqref{cauchy-stieltjesthlapr4}, and $c:= (q/p)^{1/\alpha}$ with $q := 1 - p$,
this algebra simplifies nicely to give in  \eqref{malthzero1}
\begin{equation}
\label{lampertipqal}
\E ( 1 + \lambda M_{\alpha,0}(X_p) )^{-1}  = \frac{ q + p ( 1 + \lambda)^{\alpha - 1}} { q + p(1 + \lambda)^\alpha } .
\end{equation}
This is the Stieltjes transform \eqref{lampertistieltjes} found by Lamperti.
See
\cite[\S 4]{MR1478738} 
for further discussion.

\section{Some basic theory of $P$-means}
\label{sec:basic}
This section presents some general theory of $P$-means, for an arbitrary random discrete distribution $P$, and its relation to Kingman's theory of partition structures,
relying only the simplest examples to motivate the development. This postpones to Section \ref{sec:twomeans} the study of the rich collection of examples associated with the $(\alpha,\theta)$ model.

\subsection{Partition structures}

\citet{MR509954} 
introduced the concept of the {\em partition structure} associated with sampling from a random probability distribution $F$.
That is, the collection of probability distributions of the random partitions $\Pi_n$ of the set $[n]:= \{1, \ldots, n \}$, 
generated by a {\em random sample $Y_1, \ldots, Y_n$ from $F$,} meaning that conditionally given $F$
the $Y_i$ are i.i.d. according $F$.
The blocks of $\Pi_n$ are the equivalence classes of the restriction to $[n]$ of the random equivalence relation $i \sim j$
iff $Y_i = Y_j$. 
A convenient encoding of this partition structure is provided by its {\em exchangeable partition probability function (EPPF)}
\citep{MR1337249}. 
This is a function $p$ of
{\em compositions $(n_1, \ldots, n_k)$ of $n$}, that is to say sequences of $k$ positive integers $(n_1, \ldots , n_k)$ with $\sum_{i=1}^k n_i = n$ for some $1 \le k \le n$.
The function $p(n_1, \ldots, n_k)$ gives, for each {\em particular} partition $\{B_1, \ldots, B_k\}$ of $[n]$ into $k$ blocks,
the probability
\begin{equation}
\label{eppfdef}
\P(\Pi_n =  \{B_1, \ldots, B_k\} )  = p( \#B_1,  \ldots, \#B_k) ,
\end{equation}
where $\# B_i$ is the size of the block $B_i$ of indices $j$ with the same value of $Y_j$.
A random partition $\Pi_n$ of $[n]$ is called {\em exchangeable} iff its 
distribution is invariant under the natural action of permutations of $[n]$ on partitions of $[n]$.
Equivalently, its probability function is of the form \eqref{eppfdef} for some function
$p(n_1, \ldots, n_k)$ that is non-negative and symmetric. 
The sum of these probabilities \eqref{eppfdef}, over all partitions
$\{B_1, \ldots, B_k\}$ of $[n]$ into various numbers $k$ of blocks, must then equal $1$. This constraint is most easily expressed in terms
of the 
associated {\em exchangeable random composition of $n$}
$$
\Nexn:= (\Nex_{1:n}, \Nex_{2:n} , \ldots, \Nex_{K_n:n})
$$
defined by listing the sizes of blocks of $\Pi_n$ in an exchangeable random order. This means that conditionally given 
the number $K_n$ of components of $\Pi_n$ equals $k$ for some $1 \le k \le n$, and
that $\Pi_n = \{B_1, \ldots, B_k \}$ for some particular sequence of blocks $(B_1, \ldots, B_k)$, which may be listed in any order, for instance their order of least elements,
$\Nexn:= \#B_{\sigma(1)}, \ldots, \#B_{\sigma(k)}$ where 
$\sigma$ is a uniform random permutation of $[k]$.
As indicated in \citet[(2.8)]{MR2245368}, 
the usual probability function of this random composition of $n$ is
the {\em exchangeable composition probability function (ECPF)}
\begin{equation}
\label{nexnform}
\P(\Nexn = (n_1, \ldots, n_k)) = \pex(n_1, \ldots, n_k):= \frac{1 }{k!} \binom{n}{n_1, \ldots, n_k} p(n_1, \ldots, n_k).
\end{equation}
These probabilities must sum to $1$ over all compositions of $n$. So the normalization condition on an EPPF is that
for $\pex$ derived from $p$ using the multiplier in \eqref{nexnform},
\begin{equation}
\label{pexsum1}
\sum_{k = 1}^n \sum_{(n_1, \ldots, n_k)} \pex(n_1, \ldots, n_k) = 1  .
\end{equation}
Here and in similar sums below, $(n_1, \ldots, n_k)$ ranges over the set of $\binom{n-1}{k-1}$ compositions of $n$ into $k$ parts.
To understand \eqref{nexnform}, observe that putting the components of $\Pi_n$ in an exchangeable random order creates a random ordered partition of $[n]$, with block sizes
$\Nexn$. So $\P(\Nexn:= (n_1, \ldots, n_k))$ is the sum,
 over all ordered partition of  $[n]$ into $k$ blocks of the specified sizes,
 of the probability of each ordered partition of those sizes. Each particular ordered partition has probability $p(n_1, \ldots, n_k)/k!$, and
 the number of these ordered partitions with sizes $(n_1, \ldots, n_k)$ 
is the multinomial coefficient.

For $\Pi_n$ generated by sampling from a random discrete distribution with atoms of sizes $(P_j)$,
let $(J_1, \ldots, J_n)$ denote the corresponding sample of positive integer indices. Then for each particular partition $\{B_1, \ldots, B_k\}$ of $[n]$ as in \eqref{eppfdef}
$$
\P\left( ( \Pi_n =  \{B_1, \ldots, B_k\} ) \bigcap_{i=1}^k \bigcap_{\ell \in B_i } (J_\ell = j_i ) \right) =  \E \prod_{i=1}^k P_{j_i}^{n_i} \mbox{ with } n_i := \#B_i.
$$
Hence, by conditioning on $\Pbul$,
\begin{equation}
\label{eppf}
p(n_1, \ldots, n_k):= \sum_{(j_1, \ldots, j_{k})} \E \prod_{i=1}^k P_{j_i}^{n_i} .
\end{equation}
where the sum is over all sequences of $k$ distinct positive integers $(j_1, \ldots, j_k)$.
As observed by Kingman, as $n$ varies, the partition structure associated with sampling from a random distribution is subject to a consistency
condition: the restriction of $\Pi_{n+1}$ to $[n]$ must be $\Pi_n$ for every $n \ge 1$.
In terms of the EPPF, this consistency condition implies
\begin{equation}
\label{consistnnm}
p(\nnn) = \sum_{i = 1}^{k+1} p( \nnn^{(i+)} )
\end{equation}
where $\nnn = (n_1, \ldots, n_k)$ ranges over compositions of $n$, and  $\nnn^{(i+)}$ for $1 \le i \le k +1$ is $\nnn$ with the $i$th component incremented by $1$, meaning 
for $(\nnn, 1)$ obtained by appending a $1$ to $\nnn$ for $i = k+1$.
See \citet[\S 3.2]{MR2245368} 
for further discussion.

The instance of the general formula \eqref{intSg}, when $(S,\SS)$ is the unit interval $[0,1]$ with Borel sets, and the 
$Y_j = U_j$ are i.i.d. uniform $[0,1]$ variables, independent of $P$, is of particular importance.
Write $F_P$ for the random probability measure on $[0,1]$ which sprinkes the atoms of $P$ at i.i.d. uniform random locations.
So by definition,
for all bounded or non-negative measurable $g$ 
\begin{equation}
\label{gmpu}
\int_{0}^1  g(u) F_P(du) =  M_P( g(U) ):= \sum_{j = 1}^\infty  g(U_j) P_j 
\end{equation}
In particular, for $g(u) = 1 (u \le v)$, the indicator of the interval $[0,v]$, the random cumulative distribution function (c.d.f.) of $F_P$
is 
\begin{equation}
\label{fmpu}
F_P[0,v]:= M_P( 1(U \le v )):= \sum_{j = 1}^\infty  1(U_j \le v)  P_j  \qquad ( 0 \le v \le 1) .
\end{equation}
Note that $F_P[0,0] = 0$ and $F_P[0,1] = 1$ almost surely.

The following proposition summarizes some well known facts:

\begin{proposition}
\label{propkk}
[\citet{kallenberg1973canonical}, 
\citet{MR509954}] 
The random c.d.f. $F(v):= F_P[0,v]$, derived as above for $0 \le v \le 1$ from a random discrete distribution $P$, is a process with exchangeable increments,
meaning that for each $m = 1, 2, \ldots$ the sequence $(F(i/m) - F((i-1)/m),1\le i \le m )$ is exchangeable.  The collection of 
distributions of these exchangeable sequences is an encoding of the partition structure generated by $P$, 
as is the collection of finite-dimensional distributions of $\Pdec$, the ranked re-ordering of $P$, and
the collection of finite-dimensional distributions of $\Pst$, the size-biased permutation of $P$. In other words, for two random discrete distributions $P$ and $Q$,
with associated random c.d.f.s with exchangeable increments $F_P$ and $F_Q$,
and exchangeable partition probability functions $p_P$ and $p_Q$, 
the following conditions are equivalent:
\begin{itemize}
\item $\Pdecbul \ed \Qdecbul$ 
\item $\Pst \ed \Qst$
\item $p_P(\nnn) = p_Q(\nnn)$ for all compositions of positive integers $\nnn$;
\item $F_P$ and $F_Q$ share the same finite dimensional distributions.
\end{itemize}
\end{proposition}

\begin{proof}
As indicated by Kallenberg, the finite-dimensional distributions of $F = F_P$ determine those of the list $\Pdec$ of ranked jumps of $P$, and conversely.
It is obvious that the laws of $\Pdec$ and $\Pst$ determine each other, and that either of these laws determines the EPPF $p_P$, by application
of formula \eqref{eppf} with $P$ replaced by $\Pdec$ or $\Pst$. That the law of $\Pdec$ can be recovered from the partition structure was shown by
\citet{MR509954}. 
\end{proof}

See also \citet[Theorem 3.1]{MR2245368} 
for an explicit formula expressing the EPPF in terms of product moments derived from $\Pst$.

A nice exercise in Kallenberg's encoding of $P$ by an exchangeable random c.d.f. $F:= F_P$ is provided by the following construction, 
proposed by \citet*[Example 2.10]{MR617593}, 
in an insightful review article which appeared a year before the general theory of partition structures was offered by
\citet{MR509954}. 
Suppose $P$ is a random discrete distribution with $\P(P_i >0) = 1$ for each $i = 1,2, \ldots$. Let $(U_i)$ be a sequence of
i.i.d. uniform variables, independent of $P$, and for each $0 < p < 1$ consider the sequence $P_i 1 ( U_i \le p)$ obtained by annihilating each $P_i$ with
$U_i >p$ and keeping each $P_i$ with $U_i \le p$. Then a new random discrete distribution $P(p)$, called a {\em $p$-thinning} or {\em $p$-screening} of $P$,
is obtained by ignoring the annihilated entries $P_i$ with $U_i > p$, and listing the remaining entries of $P_i$ with $P_i \le p$ in their original order, 
renormalized by their sum $F(p):= \sum_{i } P_i 1(U_i \le p)$. More precisely, the $j$th entry of $P(p)$ is $P_j(p) := P_{\tau(p,j)}/F(p)$ where $\tau(p,j)$ is the $j$th
index $i$ with $U_i \le p$. So  $\tau(p,j)$ is the sum of $j$ independent copies of $\tau(p,1)$ with the geometric$(p)$ distribution $\P(\tau(p,1) = k) = p q^{k-1}$ for $q:=1-p$,
and the sequence of indices $(\tau(p,j), j = 1,2, \ldots)$ is independent of $P$.
In terms of the random c.d.f. with exchangeable increments $F(u):= \sum_{i} P_i 1 ( U_i \le u)$, whose jumps in some order are the $P_i$, the $p$-thinning $P(p)$ is by construction a listing of jumps of the 
random c.d.f. with exchangeable increments $(F(u p)/F(p), 0 \le u \le 1)$.  In terms of $P$-means, for suitable distributions of $X$, the $P(p)$-mean of $X$ is the ratio
of two jointly distributed $P$-means:
\begin{equation}
\label{ratio}
M_{P(p)} (X) = \frac{ M_P( X 1(U \le p ) ) }{ M_P (1 (U \le p) )} := \frac{ \sum_{i}  X_i 1( U_i \le p )P_i  } { \sum_{i} 1 ( U_i \le p) P_i}  .
\end{equation}
A particularly appealing instance of this construction is described by the following proposition:
\begin{proposition}
\label{prp:gemthin}
{\em \citep*[Theorem 2.5]{MR617593} }
If $P$ is governed by the GEM$(0,\theta)$ model $P_j := H_j \prod_{i = 1}^{j-1} H_i$ for i.i.d. random factors $H_i$ with $H_i \ed \beta_{1,\theta}$ for some $\theta >0$,
then 
\begin{itemize}
\item [(i)] the random fraction $F(p)$ has beta$(p\theta, q\theta )$ distribution for $q:= 1-p$;
\item [(ii)] the $p$-thinned random discrete distribution $P(p)$ has GEM$(0, p\theta)$ distribution;
\item [(iii)] the fraction $F(p)$ is independent of the random discrete distribution $P(p)$.
\end{itemize}
\end{proposition}
\begin{proof}
As indicated by Patil and Taillie, this is a consequence of the representation of $P$ by random sampling from the random c.d.f. $F(u) = \gamma(u \theta)/\gamma(\theta)$ derived from 
the standard gamma subordinator.  
See \citet[\S 4.2]{MR2245368} 
for a proof of McCloskey's result that the size-biased representation of jumps of this $F$ gives $P$ governed by the GEM$(0,\theta)$ model
with i.i.d. beta$(1,\theta)$ distributed residual factors.
Granted the gamma representation of $P$, part (i) is just the basic beta-gamma algebra \eqref{betafromgam}. Part (ii) holds by the identification  of
$F(up)/F(p) = \gamma( u p \theta)/\gamma(p \theta), 0 \le u \le 1$ as the c.d.f. with exchangeable increments associated with $P(p)$.
Part (iii) appeals to independence part \eqref{betagamindpt} of the beta-gamma algebra, which  makes 
$F(p) = \gamma(p \theta)/\gamma(\theta)$ independent of the process $(F(up)/F(p), 0 \le u \le 1)$, hence also independent of its list of
jumps $P(p)$ in their order of discovery by a process of uniform random sampling.
\end{proof}

As remarked by Patil and Taillie, the above proposition holds also with GEM$(0,\theta)$ replaced by its decreasing rearrangement,
the Poisson-Dirichlet $(0,\theta)$ distribution.  Various components of the proposition can be broken down and generalized as follows.
\begin{proposition}
Let $P(p)$ be the random discrete distribution obtained by $p$-thinning of a random discrete distribution $P$ with $\P(P_i >0) = 1$ for each $i = 1,2, \ldots$. 
\begin{itemize}
\item [(i)] if $P = P^\downarrow$ is in ranked order, then  so is $P(p)$;
\item [(ii)] if $P = P^*$ is in size-biased random order, then  so is $P(p)$;
\end{itemize}
Suppose $P$ is a list of jumps of the random c.d.f. $F$ with exchangeable increments defined by normalization of a subordinator $A$, 
say $F(u) = A(\theta u)/ A(\theta), 0 \le u \le 1$, for some fixed $\theta >0$, then 
\begin{itemize}
\item [(iii)] $P(p)$ is a list of  normalized jumps of the same subordinator on the interval $[0, p \theta]$ instead of $[0,\theta]$.  
\item [(iv)] if $P$ is in either ranked or size-biased order, then the following two conditions are equivalent:
\begin{equation}
\label{eqdp}
P(p) \ed P \mbox{ for every } 0 < p < 1 ;
\end{equation}
\begin{equation}
\label{stablesub}
\mbox{ $A$ is a stable $(\alpha)$ subordinator  for some $0 < \alpha < 1$.}
\end{equation}
in which case $P^*$ is governed by the GEM$(\alpha,0)$ model with independent residual factors $H_i \ed \beta_{1- \alpha, i \alpha}$ for $i = 1,2, \ldots$.
\end{itemize}
\end{proposition}
\begin{proof}
Part (i) is obvious. To see part (ii), observe that $P= P^*$ may be constructed by listing the jumps of the associated random c.d.f. with exchangeable increments $F$ in the order they
are discovered by a process of random sampling from $F$. But then by construction as above, $P(p)$ is the list of sizes of jumps of $F$ in $[0,p]$, relative to their sum $F(p)$, in the order of
their discovery in samping from $F$. But the successive values of the sample from $F$ which fall in $[0,p]$ form a sample from $F$ conditioned on $[0,p]$. Thus $P(p)$ is just the list of atoms
of this random conditional distribution in their order of their discovery by a process of random sampling, and it follows that $P(p)$ is in size-biased random order.
Part (iii) is just a reprise of part (ii) of the previous proposition, with  a general subordinator instead of the gamma process.
As for part (iv), if $F$ is derived from a stable subordinator, it is easily seen that the distribution of the process $(F(u p)/F(p) = A(u p ) / A(p), 0 \le u \le 1)$ does not
depend on $p$. Hence $P(p) \ed P$, for either ranked or size-biased ordering of $P$, by (i) and (ii).
Conversely, it is known \citep[Lemma 7.5]{MR1168191} 
that for a subordinator $A$ the distribution of $A(1)$ is determined up to a scale factor by that of the process $(A(u)/A(1), 0 \le u \le 1)$. 
If $P(p) \ed P$ for all $0 < p < 1$, then the distribution of $(F(u p)/F(p) = A(u p ) / A(p), 0 \le u \le 1)$ is the same for all $0 < p < 1$, hence $A(p) \ed c(p) A(1)$
for some constant $c(p)$. It is well known that for a subordinator $A$ this condition implies that $A$ is stable with some index $\alpha \in (0,1)$ as indicated in \eqref{stablesub}.
\end{proof}

The only part of Proposition  \ref{prp:gemthin} which does not extend to a subordinator more general than the gamma process is the independence of $F(p)$ and $P(p)$.
This is a consequence of independence of  $A(t)$ and $(A(u t)/A(t), 0 \le u \le t )$, which is well known to be a characteristic property of $A(t) = a \gamma( b t)$ for some $a, b>0$.
See \citet[\S 4.2]{MR2245368} 
and work cited there.
See also \citet{MR2004330} 
and \cite{MR2074696} 
for more about bridges with exchangeable increments obtained by normalizing a subordinator.

The construction of infinitely divisible {\em semi-stable laws} by 
\citet[\S 58]{levy54} shows for each fixed $q \in (0,1)$ there exist non-stable subordinators 
such that \eqref{eqdp} holds if $p = q^n$ for some $n = 1,2, \ldots$ but not for all $0 < p < 1$.
Let 
$P_{(\alpha,0)}$ denote a random discrete distribution governed by the $(\alpha,0)$ model, say in size-biased order for simplicity, but
it could just as well be ranked.
Part (iv) of the above proposition implies
that for each probability distribution $\pi$ on $(0,1)$, which might be regarded as a prior distribution on the stability index $\alpha$, the formula
\begin{equation}
\label{almix}
\P( P \in \bullet) = \int_{(0,1)} \pi (d\alpha) \P( P{(\alpha,0)}\in \bullet)
\end{equation}
defines a mixture of $(\alpha,0)$ laws, which governs $P$ with the invariance property \eqref{eqdp} under $p$-thinning for all $0 < p < 1$.

\begin{problem}
Are there any other laws besides \eqref{almix} of random discrete distributions $P$ such that $\P(P_i >0) = 1$ for all $i$ and $P(p) \ed P$ for all $0 < p < 1$?
\end{problem}

See 
\citet{MR1742892} 
and
\citet{MR1768841}  
for various constructions of $P_{(\alpha,0)}$ governed by the $(\alpha,0)$ model as a stochastic process indexed by $\alpha \in (0,1)$.


\subsection{$P$-means and partition structures}

The present point of view is that the collection of distributions of $P$-means $M_P(X)$, indexed by various distributions of $X$, 
should be regarded as yet another encoding of the partition structure  associated with $P$.
That point of view is justified by the following corollary of Proposition \ref{propkk}, which does not seem to have been
pointed out before. Call a random variable {\em simple} if it takes only a finite number of possible values.

\begin{corollary}
\label{crl:means:kingman}
{\em [Characterization of partition structures by $P$-means]}
For each random discrete distribution $\Pbul$, the collection of distributions of its $\Pbul$-means $M_P(X)$, as $X$ ranges
over simple random variables, is an encoding of the partition structure of $\Pbul$.
That is to say, for any two random discrete distributions $P$ and $Q$, the condition
\begin{itemize}
\item $M_P(X) \ed M_Q(X)$ for every simple $X$
\end{itemize}
can be added to the list of equivalent conditions in the Proposition \ref{propkk}.
\end{corollary}
\begin{proof}
As remarked earlier around \eqref{ranksb}, it the distribution of $M_P(X)$ remains unchanged if 
$P$ is replaced by $\Pdec$, and the same for $Q$ instead of $P$. So $\Pdec \ed \Qdec$ implies $M_P(X) \ed M_Q(X)$.
For the converse, the Cram\'er-Wold theorem shows that the finite-dimensional distributions of $F_P$ are determined
by the collection of one-dimensional distributions of finite linear combinations of $F_P[0,v], 0 \le v \le 1$, each of which is a $P$-mean
by application of \eqref{gmpu}:
$$
\sum_{i} \, a_i \, F_P[0,v_i] = M_P\left(\sum_i  \, a_i 1(U \le v_i) \right) .
$$
So $M_P(X) \ed M_Q(X)$ for all simple $X$ implies that the finite dimensional distributions of $F_P$ and $F_Q$ are the same.
Hence the conclusion, by the preceding proposition.
\end{proof}

Part of how the partition structure of $P$ is determined by the distributions of $P$-means $M_P(X)$, as the distribution of $X$ varies,
is found by consideration of the $P$-means of indicator variables $X$, that is $X = 1(U \le v)$ whose $P$-mean
is $F_P(v)$. So there is the following proposition, which also does not seem to have been noticed before, though it is the
easiest case for an indicator variable of the general moment formula for $P$-means, due to Kerov,
 which is presented later in Corollary \ref{mainthm}.

\begin{proposition}
\label{crlpgf}
Let $F(v):= F_P[0,v]$ be the random cumulative distribution function with exchangeable increments on $[0,1]$ derived from 
a random discrete distribution $\Pbul$, and  let $K_n$ be the number of distinct values in 
a random sample of size $n$ from either $\Pbul$ or from $F$.
Then the $n$th moment of $F(v)$ is  a polynomial in $v$ of degree at most $n$, which equals the
probability generating function of $K_n$ evaluated at $v$:
\begin{equation}
\label{pgfform}
\E  [ F(v)]^n = \E v^{K_n} = \sum_{k=1}^n \P(K_n = k)  v^k 
\end{equation}
where $\P( K_n = k )$ is determined by the ECPF $\pex$ of $P$ according to the formula
\begin{equation}
\label{kndist}
\P( K_n = k ) = \sum_{(n_1, \ldots, n_k) } \pex(n_1, \ldots, n_k)  
\end{equation}
where the sum is over all $\binom{n-1}{k-1}$ compositions of $n$ into $k$ parts.
Consequently, the collection of one-dimensional distributions of $K_n$, for $n = 1,2, \ldots$  determines the
collection of one-dimensional distributions of $F(v)$ for $0 \le v \le 1$, and vice versa.
\end{proposition}
\begin{proof}
Formula \eqref{pgfform} displays two different ways of evaluating the probability of the event 
$E:= \cap_{1 \le i \le n} (V_i \le v)$ for a random sample $V_1, \ldots , V_n$ from $F$.
On the one hand, $\P(E \giv F) = [F(v)]^n$. On the other hand, $\P(E \giv K_n = k) = v^k$, because given $k$ distinct values of the $V_i$, these values are $k$ 
independent uniform $[0,1]$ variables $U_j, 1 \le j \le k$, which all fall to the left of $v$ with probability $v^k$.
\end{proof}

It is known 
\citep{MR2238047} 
that another equivalent condition is equality in distribution of the two sequences $(K_n, n \ge 1)$ generated by sampling
from $P$ and $Q$ respectively. 

\begin{problem}
Does equality of the one-dimensional distributions of $K_n$, generated by sampling from $P$ and $Q$ for each $n$, imply equality of partition structures?
\end{problem}
By Corollary \ref{crlpgf}, this condition is the same as equality of one-dimensional distributions of $F_P[0,p]$ and $F_Q[0,p]$ for each $0 \le p \le 1$.
So the issue is whether the finite-dimensional distributions of an increasing process with exchangeable increments are determined by its one-dimensional distributions. 
[\citet{kallenberg1973canonical}, 
established a result in this vein, that the distribution of any process on $[0,1]$ with exchangeable increments and continuous paths is determined by its one-dimensional distributions.  

It appears that the distribution of an exchangeable random partition $\Pi_n$ on $[n]$, with restrictions $\Pi_m$ to $[m]$
for $m \le n$, is determined by the collection of distributions of $K_m$, the number of blocks of $\Pi_m$, for $1 \le m \le n$,
for $n \le 11$ but not  for $n = 12$.
To see this, consider the $\np(n)$ probabilities of individual partitions of $n$ in the distribution of the
partition of $n$ induced by the ranked block sizes of $\Pi_n$, where $\np(n)$ is the number of partitions of $n$.
These $\np(n)$ probabilities are subject only to the constraints of being non-negative, with sum $1$, so the range of $\np(n)-1$ of these probabilities
contains some open ball in $\reals^{\np(n) - 1}$. The $\P(K_m = k)$ for $1 \le k <  m \le n$ then form a collection of $\binom{n-1}{2}$ linearly independent linear
combinations of the $\np(n)$. It is easily checked that $\np(n) - 1 \le \binom{n-1}{2}$ for $1 \le n \le 11$, but
$\np(12) - 1 = 76 > 66 = \binom{12}{2}$. Hence the conclusion. However, it does not seem at all obvious how to construct such an example
which is part of an infinite partition structure derived by sampling from a random discrete distribution.

The following proposition develops the meaning of the terms $\pex(n_1, \ldots, n_k)$ in the sum \eqref{kndist} for $\P(K_n = k)$, 
in the context of the preceding proof.

\begin{proposition}
Let $V_1, \ldots, V_n$ be a sample from $F_P$, meaning that
\begin{equation}
\label{Vdef}
\P \left ( \cap_{i=1}^n (V_i \le v_i) \giv F_P \right) = \prod_{i=1}^n F_P[0,v_i] \qquad ( 0 \le v_i \le 1).
\end{equation}
Let $K_n$ be the number of distinct values among $V_1, \ldots, V_n$,  and let $\Nexn$ be the numbers of repetitions of these values in the 
sample $V_1, \ldots, V_n$, in increasing order of $V$-values.
Then $\Nexn$ is an exchangeable random composition of $n$ with the probability function $\pex$ featured in formulas \eqref{nexnform} and \eqref{kndist}.
\end{proposition}
\begin{proof}
By construction, $K_n$ is the number of blocks of $\Pi_n$, the random partition of $[n]$ generated by sampling from $P$. 
On the event of probability one that there are no ties among the $U$-values, the association $V_i = U_{J_i}$ pairs distinct $V$-values with distinct $J$-values 
in a sample $J_1, \ldots, J_n$ of indices of $P$.
Thus $K_n$ is the number of distinct values in a sample of size $n$ from $P$, and  the distinct $V$-values are the uniform order statistics
$$
U_{1:K_n} < U_{2:K_n} < \cdots < U_{K_n:K_n}
$$
where for $ k = 1,2, \ldots $ the $U_{1:k} < U_{2:k} < \cdots < U_{k:k}$ are the order statistics of the first $k$ i.i.d. uniform variables $U_1, \ldots, U_k$.
It is well known that $U_i = U_{\sigma_k(i):k}$ for a random permutation $\sigma_k$ of $[k] $ that is independent of these $k$ order statistics. Hence
$\Nexn$ is an exchangeable random composition  whose probability function \eqref{nexnform} encodes the partition structure of $P$.
\end{proof}


\subsection{$P$-means as conditional expectations}

The point of view taken here is that a random discrete distribution $P$ may be regarded as a probabilistic mechanism for turning a suitable random variable 
$X$ into another random variable $M_P(X)$. Considered in this way, $M_P$ becomes an operator on random variables $X$, 
whose properties are those of a conditional expectation operator. In the first instance, the definition $M_P(X):= \sum_j X_j P_j$,
makes $M_P$ an operator on probability distributions, which converts the common distribution of $X$ and the $X_j$ 
into the distribution of the new random variable $M_P(X)$. There is no specification of which of the many identically distributed variables $X_j$ 
should be regarded as $X$. 
This construction of $\tilX:= M_P(X)$ puts $\tilX$ on the same probability space as all the copies $X_j$ of $X$. But the joint distribution of $\tilX$ and $X_j$ will typically depend on $j$.
So there is no well defined joint distribution of $\tilX$ and a generic representative $X$ of the terms $X_j$ without some further precision. 
For instance, if $\E(X) = 0$ and $\E X^2 < \infty$, then the covariance 
$$
\E (\tilX X_j )  = ( \E  P_j) \E X^2
$$
will typically depend on $j$. Only  exceptionally, as in the case of 
exchangeable $P_1, \dots, P_m$, 
does the joint law of $(\tilX,X_j)$ not depend on $j$ for some finite range $1 \le j \le m$.
This apparent lack of a  joint distribution of $X$ and $\tilX:= M_P(X)$ should be contrasted with conditional expectations $\tilX:= \E(X \giv \GG)$ 
for $\GG$ any sub $\sigma$-field of events in a probability space $(\Omega, \FF, \P)$ on which $X$ is defined and integrable. For then $\tilX$ and $X$ are
defined on the same probability space, with an induced joint probability distribution $\P( (X,\tilX) \in \bullet)$ on $\reals^2$.

There are however many indications in the  literature of particular $P$-means, that the operation which transforms  a random variable $X$ into 
$M_P(X)$ shares properties of a conditional expectation operator $\E(X \giv \GG)$. 
Most obviously, $M_P$ is a {\em positive operator}:
$X \ge 0$ implies $M_P(X) \ge 0$, and $M_P$ is a {\em linear operator}, meaning that if $(X,Y)$ has some arbitrary joint distribution,  
such that both $\tilX:= M_P(X)$ and $\tilY:= M_P(Y)$ are well defined almost surely, then the natural construction of a random pair $(\tilX, \tilY) := M_P(X,Y)$, using
one copy of $P$ and an i.i.d.  sequence $(X_j, Y_j)$ of copies of $(X,Y)$, makes
$$
M_P( a X + b Y ) = a M_P(X)  + b M_P(Y) .
$$
It is also easily shown there is a monotone convergence theorem for $P$-means:  with the same coupling construction 
\begin{equation}
\label{monconv}
\mbox{ $0 \le X_n \uparrow X$ as $n \to \infty$ implies $0 \le M_P(X_n) \uparrow M_P(X)$ a.s.} 
\end{equation}
All of which supports the idea that $P$-means should be regarded as some kind of conditional expectation operator.
In fact, for any prescribed distribution of $X$ on an abstract measurable space,  there is the following
{\em canonical construction} of $X$ jointly with a sequence of i.i.d. copies $(X_j)$ of $X$ and a random discrete $P$ with any
desired distribution, and a suitable $\sigma$-field of events $\GG$, which makes
$$M_P[ g(X) ] = \E [ g(X) \giv \GG ] \qquad a.s.$$
for all bounded or non-negative measurable functions $g$.
Assume that the $(X_j)$ and $(P_j)$ are defined together with a uniform $[0,1]$ variable $U$, as needed for further randomization,
on some probability space  $(\Omega, \FF, \P)$, with $(X_j)$ , $(P_j)$ and $U$ independent.
Conditionally given $(X_j)$ and $P = (P_j)$ let $J$ be a random draw from $P$:
$$
\P(J= j \giv X_1, X_2, \ldots, P_1, P_2, \ldots ) = P_j \qquad ( j = 1,2, \ldots ),
$$
which may be constructed in the usual way by letting 
$$
\mbox{ $J= j$ if $\sum_{i = 1}^{j-1} P_i < U \le \sum_{i= 1}^j P_i$.  }
$$
Then set
$$ 
X:= X_J .
$$
So $X$ is not any particular $X_j$, but $X = X_J$ for $J$ picked at random according to $P$, independently of the entire sequence of $X_j$-values.
Then the following proposition is easily verified:
\begin{proposition}
\label{prp:canon}
Let $X:= X_J$ be defined in terms of an i.i.d. sequence $(X_j)$ and a random discrete distribution $(P_j)$ independent of $(X_j)$ by this canonical construction,
with the random index $J$ picked according to $P$, independently of $(X_j)$. 
Then
\begin{itemize}
\item the distribution of $X$ is the common distribution of the $X_j$;
\item for each measurable function $g$ with $\E |g(X)| < \infty$, let the $P$-mean of $g(X)$  be defined by
$$
M_P[ g(X) ] := \sum_{j = 1}^\infty  g(X_j) P_j  .
$$
Then the series converges absolutely both almost surely and in $L^1$, and $M_P[ g(X) ]$ is the conditional expectation
$$
M_P[ g(X) ] = \E[ g(X) \giv  X_1, X_2, \ldots, P_1, P_2 , \ldots ]   \mbox{ a.s. }
$$
\item
In particular, if $X$ is real-valued with $\E |X| < \infty$, and $\tilX:= M_P(X)$, then
$$
\E( X \giv \tilX ) = \tilX
$$
so the sequence $(\E X, \tilX , X )$ is a three term martingale.
\end{itemize}
\end{proposition}

Consequently, for each random discrete distribution of $P$, the transformation from the distribution of $X$ to that of its $P$-mean $\tilX$
enjoys {\em all} the well known general properties of conditional expectation operator.
So $P$-means should be properly be understood, like conditional expectations,  as a kind of partial averaging operator.
Some of these properties of $P$-means inherited from conditional expectations are listed in the following corollary.
Recall that the {\em convex partial order} on the distributions of real valued random variables $X$ and $Y$ with finite means is defined by
$X \lecx Y $ iff
\begin{equation}
\label{jensen}
\E \phi( X ) \le \E \phi (Y) \mbox{ for every convex function $\phi$.  }
\end{equation}
This relation $X \lecx Y$ should be understood as a relation between the distributions of $X$ and of $Y$, subject to $\E|X| <\infty$ and $\E|Y| <\infty$,
comparable to the usual stochastic order $X \led Y$, meaning that $\E \phi( X ) \le \E \phi (Y) $ for all bounded increasing $\phi$.
Because every convex function $\phi(x)$ is bounded below by some affine function $a x + b$, the assumption $\E|X| <\infty$ implies 
$\E \phi(X)$ has a well defined value which is either finite or $+\infty$ for every convex $\phi$, and similarly for $Y$.  So  for $X$ and $Y$
with both $\E |X| < \infty$ and $E |Y| <\infty$, the meaning of the condition \eqref{jensen} can be made more precise 
in either of the following equivalent ways:
\begin{itemize}
\item \eqref{jensen} holds for all convex $\phi$, allowing $+\infty$ as a value on one or both sides;
\item \eqref{jensen} holds for all convex $\phi$ such that both $\E \phi(X)$ and $\E \phi(Y)$ are finite.
\end{itemize}
It is known \citep*[\S 2.A]{MR2265633} 
that further equivalent conditions are
\begin{itemize}
\item $\E X = \E Y $ and the inequality \eqref{jensen} holds for $\phi(x) = (x-a)_+$ for all $a \in \reals$;
\item $\E X = \E Y $ and the inequality \eqref{jensen} holds for $\phi(x) = |x-a|$ for all $a \in \reals$.
\end{itemize}
Given some prescribed distributions on the line for $X$ and for $Y$, a {\em coupling of $X$ and $Y$} is a construction of random variables
$X$ and $Y$ with these distributions on a common probability space.
It is a well known that $X \led Y$ is equivalent to existence of a coupling of $X$ and $Y$ with $\P(X \le Y) = 1$: 
simply take $X = F_{X}^{-1}(U)$ and $Y = F_{Y}^{-1}(U)$ where $F_X^{-1}$ and $F_Y^{-1}$ are the usual inverse distribution functions, and $U$ has
uniform $[0,1]$ distribution. 

By Jensen's inequality for conditional expectations,  $X \lecx Y$ is implied by
\begin{itemize}
\item there exists a {\em martingale coupling of $X$ and $Y$}, that is a construction of $X$ and $Y$ with $\E(Y \giv X) = X$.
\end{itemize}
That remark is all that is 
needed to deduce the following Corollary from Proposition \ref{prp:canon}.
It is a well known result of Strassen that $X \lecx Y$ implies the existence of a martingale coupling of $X$ and $Y$.
But the construction is quite difficult and not explicit in general.
See \citet*{MR2808243} 
and \citet*{MR3706738} for this result and more about the convex order.  

\begin{corollary}
\label{crl:convex}
Let $X$ be a random variable with $\E |X | < \infty$, and let $\tilX:= M_P(X)$ be its $P$-mean for some random discrete distribution $P$.
Then $\tilX \lecx X$.  In particular:
\begin{itemize}
\item [(i)] $\E |\tilX | \le \E|X| < \infty$ and $\E \tilX = \E X$. 
\item [(ii)]If $\E |X|^r < \infty$ for some $r >1$ then $\E | \tilX |^r \le \E |X|^r < \infty$.
\item [(iii)] The distributions of $X$ and $\tilX$ cannot be the same, except if either $\P(X = x)= 1$ for some $x$, or $\P( P_j = 1  \mbox{ for some } j) = 1$.
\end{itemize}
\end{corollary}
\begin{proof}
All but part (iii) follow immediately from Proposition \ref{prp:canon}.  These statements also follow from the definition 
$\tilX:= \sum_j X_j P_j$ by applying Jensen's inequality $\phi(\sum_j X_j P_j) \le \sum_j \phi(X_j) P_j$ before taking expectations.
As for (iii), it is well known 
\citep[Exercise 5.1.12]{MR2722836} 
that if a martingale pair $(\tilX, X)$ has $\tilX \ed X$, then $\P(\tilX = X) = 1$. It is easily seen that for $\tilX:= M_P(X)$
this can only be so in one of the two exceptional cases indicated.
\end{proof}

Part (i) of this Corollary, and the instance of part (ii) for $r= n$ a positive integer,
 can also be deduced from the formula for $\E \tilX^n$ presented later in Corollary \ref{mainthm}.
Part (iii) appears
in \citet[Proposition 3]{yamato1984characteristic} 
for the case of Dirichlet $(0,\theta)$ means.

As an operator mapping a distribution of $X$ to a distribution of $\tilX$, one property of $P$-means extends those of a typical conditional expectation  operator: 
the $P$-mean of $X$ may be well defined and finite by almost sure convergence, even if $\E |X| = \infty$.  For instance, there is the following easy generalization of a result of 
\citet{yamato1984characteristic} 
for Dirichlet $(0,\theta)$  means,
and \citet{van1987random} 
for the uniformly weighted mean $X_1 P_1 + X_2 (1-P_1)$ for $P_1$ with uniform distribution on $[0,1]$.

\begin{proposition}
\label{prp:yamagen}
Suppose that $X \ed a + b Y$ for some fixed $a$ and $b$ and $Y$ with the standard Cauchy distribution $\P(Y \in dy ) = \pi^{-1}(1 + y^2)^{-1} dy$.
Then, no matter what the random discrete distribution $P$,  the $P$-mean $\tilX$ is well defined as an almost surely convergent series, 
with $\tilX \ed X$.  
\end{proposition}
\begin{proof}
This can be shown by a computation with characteristic functions after conditioning on $P$, as in \citet{yamato1984characteristic}.
Alternatively, 
using the well known scaling property $Y(p) \ed p Y(1)$
of a standard Cauchy process with stationary independent increments $(Y(t), t \ge 0 )$, assumed independent of $P$,
the $P$-mean $\tilX$ may be constructed as the limit of $a + b Y(\Sigma_{i=1}^j P_i)$ as $j \to \infty$. It is easily seen
by conditioning on $(P_1, P_2, \ldots)$ that the limit exists and equals $a + b Y(1)$ almost surely.
\end{proof}

For the case of $\tilX = X_1 P_1 + X_2 (1-P_1) $ with $P_1$ uniform on $[0,1]$,
\citet[Theorem 2]{van1987random} obtained the conclusion of this proposition by a more complicated argument involving Stieltjes transforms.
But he also obtained a converse:  the equality in distribution $\tilX \ed X$ implies that $X \ed a + b Y$ for
some real $a$ and $b$ and $Y$ standard Cauchy. It appears that this converse is true under very much weaker conditions on $P$.
But some condition is required to avoid the case $P_2 = 1- P_1$ with the distribution of $P_1$ concentrated on terms of a geometric progression $(q^n, n = 1,2, \ldots)$ for some $0 < q < 1$.
For \citet[\S 58]{levy54}  established the existence of infinitely divisible {\em semi-stable laws} of $X$ such $X \ed p X + (1-p) X$ if $p = q^n$ for some $n$,
besides the family of {\em strictly stable} Cauchy laws $a Y + b$, which is characterized by this property for all $p \in (0,1)$.

\subsection{Refinements}
\label{sec:refine}
For $P$ and $R$ two random discrete distributions, say that $R$ {\em is a refinement of $P$} if there is a coupling of $P$ and $R$ on
a common probability space such that that both $P = (P_i)$ and $R = (R_i)$ may be indexed by $i \in \nats:= \{1,2, \ldots\}$ in the usual way, while some rearrangement of atoms of $R$ may be indexed by 
$(i,j) \in \nats^2$ as $R_{i,j}$ with
$$
P_i = \sum_{j \in \nats} R_{i,j} \qquad ( i \in \nats ) .
$$
The following proposition provides a simple explanation of many monotonicity results for $P$-means:
\begin{proposition}
\label{prp:refine}
If $R$ is a refinement of $P$, then $M_R(X) \lecx M_P(X)$ for every $X$ with $\E |X| < \infty$. 
\end{proposition}
\begin{proof}
It must be shown that for arbitrary convex $\phi$, and $X$ with $\E|X| < \infty$
\begin{equation}
\label{phiinc}
\E \phi \left( \Sigma_{i,j} X_{i,j} R_{i,j} \right) \le \E \phi \left( \Sigma_{i} X_i P_i \right)
\end{equation}
where $(X_{i,j}, i,j \in \nats)$ is a doubly indexed array of copies of $X$, independent of $R$,
and $(X_{i}, i \in \nats)$ is a singly indexed list of copies of $X$, independent of $P$.
By conditioning on the coupling $(P,R)$, it is enough to establish 
\eqref{phiinc} for a fixed, non-random discrete distribution $R$, which is a refinement of some other
fixed, non-random discrete distribution $P$.
A further reduction, by easy limit arguments, shows it is enough to establish \eqref{phiinc}
when $R$ has only a finite number of non-zero atoms. Moreover, by induction on the number these atoms, it is enough to consider the case 
when only one atom of $P$ is split to obtain $R$ from $P$.  That case reduces easily by conditioning and scaling to the base case 
$\E \phi( M_R(X))  \le \E \phi(X)$ of Corollary \eqref{crl:convex}.
\end{proof}

By general theory of the convex order of distributions on the line, recently reviewed by
\citet*{MR3706748}, 
the above proposition implies it is possible to realize the sequence 
$$( \E X, M_R(X), M_P(X), X)$$ 
on a suitable probability space as a four term martingale.
It is well known however that the general construction of such a martingale, from a sequence of distributions increasing  in the convex order, is not
at all explicit or elementary, and the proof sketched above does not help much either.
So it is natural to ask if the canonical martingale construction of $(M_P(X),X)$ in Proposition \ref{prp:canon} can be extended 
to provide an explicit martingale $(M_R(X), M_P(X), X)$ on a suitable probability space, whenever $R$ is a refinement of $P$.
The following argument shows how this is possible. But the argument is quite tricky, and it does not seem obvious how to extend it 
to a sequence of successive refinements in any nicer way than by forcing the martingale to be Markovian with prescribed two-dimensional laws.

\begin{proof}[Martingale proof of Proposition \ref{prp:refine}]
The aim is to construct $R$ and $P$ jointly with $X$ on some common probability space $(\Omega, \FF, \P)$ 
so that $M_R(X) = \E(X \giv \RR)$ and $M_P(X) = \E(X \giv \PP)$ for some sub $\sigma$-fields $\RR \subseteq \PP \subseteq \FF$.
Note well that while $R$ is a refinement of $P$, the associated $\sigma$-field $\RR$ must be coarser than $\PP$.
It is possible to make such a construction quite generally. But the definition of the $\sigma$-fields involved is tricky.
So as in the previous proof, let us rather argue that by conditioning on $(R,P)$ it is enough to consider the case of deterministic $R$ and $P$.
So consider a fixed pair of discrete distributions $(R,P)$, and let $(I,J)$ be a random element of $\nats^2$ which conditionally given $X_{\dbul}:= (X_{i,j}, i,j \in \nats)$ 
is a pick from $R$:
\begin{equation}
\label{ijdef}
\P(  (I,J)  = (i,j) \giv X_{\dbul} ) = R_{i,j} \qquad ( i,j \in \nats )
\end{equation}
and set
\begin{equation}
\label{ijdef1}
X:= X_{I,J} = \sum_{i,j} X_{i,j} 1 ((I,J)  = (i,j) )
\end{equation}
to make
\begin{equation}
\label{ijdef2}
\E (X \giv X_{\dbul} ) = M_R(X):= \sum_{i,j} X_{i,j} R_{i,j} .
\end{equation}
To involve $P$ as well, for $i$ with $P_i >0$ let $J_i$ be a random index with the conditional distribution of $J$ given $I= i$, that is
$\P(J_i = j) = R_{i,j}/P_i$. Suppose that the $J_i$ are independent, forming a sequence $J_\bul:= (J_i)$ with $i$ ranging over $\{i:P_i >0\}$.
Assume further that the sequence $J_{\bul}$ is independent of the double array $X_{\dbul}$ of copies of $X$.
Now a random pair $(I,J)$ as in \eqref{ijdef}, and $X:= X_{I,J}$ subject to \eqref{ijdef2}, is conveniently constructed from
the double array $X_{\dbul}$ of copies of $X$ and the sequence of conditional indices $J_\bul$ as $J:= J_{I}$ for a single random index $I$ with
\begin{equation}
\label{indptP}
\P(  I  = i \giv X_{\dbul} , J_\bul) =  P_i \qquad (i \in \nats)
\end{equation}
so that
\begin{equation}
\label{ijdef3}
X:= X_{I,J} = \sum_{i} X_{i,J_i} 1 (I = i)
\end{equation}
and hence
\begin{equation}
\label{ijdef4}
\E (X \giv X_{\dbul} , J_\bul ) = M_P(X):= \sum_{i} X_{i,J_i} P_i 
\end{equation}
where it is easily argued that $(X_{i,J_i})$ is a sequence of independent copies of $X$, with
this sequence independent of $P$ by \eqref{indptP}.
Thus we obtain a coupled pair of representations 
$M_R(X) = \E(X \giv \RR)$ and $M_P(X) = \E(X \giv \PP)$ with $\RR \subseteq \PP$ for
$\RR$ the $\sigma$-field generated by $X_{\dbul}$,
and $\PP$ generated by $X_{\dbul}$ and $J_\bul$. Hence the desired conclusion  \eqref{phiinc}, by Jensen's inequality
for conditional expectations.
\end{proof}

As an application of this proposition, 
there are known constructions of the $(0,\theta)$ model which are refining as $\theta$ increases
\citep*{MR2351686}. 
For instance, let $(V_i, Y_i)$ be the points of a Poisson process with intensity $dv dy /(1-v)$ in the strip $(0 < v < 1) \times (0 < y < \infty)$. Then let
$P_{0,\theta,j}$ be the length of the $j$th component interval of the relative complement in $[0,1]$ of the random set of points $\{ V_i : 0 < Y_i \le \theta \}$,
reading the intervals from left to right. As shown by \citet{MR645134}, 
this construction makes $P_{0,\theta,j} = H_{j,\theta} \prod_{i=1}^{j-1}(1- H_{i,\theta})$ where the
$H_{j,\theta}$ are i.i.d. copies of $\beta_{1,\theta}$, which is the characteristic property of the size-biased ordering of the $(0,\theta)$ model.
This construction refines the random discrete distributions $P_{0,\theta}$ as $\theta$ increases, hence the following corollary of Proposition \ref{prp:refine}:

\begin{corollary}
\label{crl:refine}
{\em \citep*[Theorem 1.2]{MR3706748} } 
For every $X$ with $\E |X| < \infty$, as $\theta$ increases on $[0,\infty)$ the family of distributions of $(0,\theta)$ means of $X$ is decreasing 
in the convex order of distributions on the line, starting from the distribution of $X$ at $\theta = 0$, and converging to the constant $\E(X)$ in the limit as $\theta \uparrow \infty$.
\end{corollary}

See \citep*{MR3706748} 
for many more refined results regarding the family of
{\em Dirichlet curves} in the space of probability distributions on the line,  
meaning the laws of $(0,\theta)$ means of a fixed distribution of 
$X$ as a function of $\theta$. It is an implication of Corollary \ref{crl:refine} and a well known result of Kellerer, discussed further in 
\citep*[\S 2]{MR3706748}, 
that for each distribution of $X$ with finite mean, it is possible to construct a Markovian reversed martingale $(\tilX_\theta, \theta \ge 0)$ with 
$\tilX_0 = X$ and $\lim_{\theta \to \infty} \tilX_{\theta} = \E X$ almost surely, such that $\tilX_\theta \ed M_{0,\theta}(X)$ for each $\theta \ge 0$.
However, there is no known way to explicitly construct the transition kernel of such a Markov process. 
The construction indicated above gives an explicit enough process 
\begin{equation}
\label{ignatovtheta}
\tilX_\theta := \sum_{j=1}^\infty X_j P_{\theta, j}
\end{equation}
 for $X_j$ i.i.d. copies of $X$ and $(P_{\theta, j}, j = 1,2, \ldots)$ 
the family of coupled copies of GEM$(0,\theta)$ generated by Ignatov's  Poisson construction. Even for the simplest choice of Bernoulli $(p)$ distributed $X_j$, when we know 
$\tilX_\theta \ed \beta_{p \theta, q \theta}$, it seems difficult to provide any explicit description of the joint law of $(\tilX_{\theta} , \tilX_\phi)$ for $0 < \theta < \phi$,
or even to determine whether or not this process is Markovian, or a reversed martingale. 
It is known however \citep*{MR2351686} 
that a corresponding process of compositions of $n$, obtained by sampling from this model, is Markovian with a simple transition mechanism, and it might be possible to proceed from this
to some analysis of $(\tilX_\theta , \theta \ge 0)$ defined by \eqref{ignatovtheta}.

One final remark about Proposition \ref{prp:refine}. 
The converse is completely false. Consider the classical example with $P_n$ the deterministic uniform distribution on $[n]$, discussed  further
in Section \ref{sec:arithmetic}.
It is well known that $M_{P_n} (X) := (X_1 + \cdots + X_n)/n$ is a reversed martingale, for any distribution of $X$ with $\E|X| < \infty$.
So the distribution of $M_{P_n} (X) $ is decreasing in the convex order, but $P_n$ is a refinement of $P_m$  iff $m$  divides $n$.

\begin{problem}
What more explicit condition on a pair of random discrete distributions $P$ and $R$ is equivalent to 
$M_R(X) \lecx M_P(X)$ for all $X$ with a finite mean?
\end{problem}

Even for deterministic $P$ and $R$ this seems to be a non-trivial problem. 
A discussion of various {\em measures of diversity} for random discrete distributions, and concepts  of comparison  of 
$P$ and $R$ with respect to such measures, with many references to earlier work, was provided by
\citet*{MR617593}. 
That article discusses relations between four different partial orderings on distributions of random discrete distributions, each of which provides some sense in which 
$R$ may be {\em stochastically more diverse than} $P$, denoted SD2, SD3, SD4, SD5.
It appears that all of these orderings are implied by the ordering by refinement, call it SD1, as that notation was not used by Patil and Taillie,
and the refinement ordering SD1 seems to be both the simplest and strongest of all these orderings.  
Already in \citet{fisher1943solo} 
there is the idea that in his limit model for species sampling, called here the  $(0,\theta)$ model,
the parameter $\theta > 0$ (which Fisher called $\alpha$, not to be confused with the second parameter $\alpha \in (0,1)$ of the $(\alpha,\theta)$ model)
should be regarded as some kind of index of diversity in the random distribution of species frequencies in the population.
This idea was confirmed by \citet[Theorem 2.9]{MR675883}, 
according to which  the $(0,\theta)$ family is increasing in stochastic diversity according to the partial order SD3.  As discussed above, the $(0,\theta)$ family is increasing in the refinement order SD1, hence also
in all of the other orders considered by Patil and Taillie.
A sixth partial order, say SD6, defined by $M_R(X) \lecx M_P(X)$ for all $X$ with a finite mean, is implied by SD1, 
and is perhaps the same as one of partial orders proposed by Patil and Taillie.
One of these partial orders, denoted  SD4 by Patil and Taillie, is the condition  that
$\Rdec[n]:= \sum_{i = 1}^n \Rdec_i$ is stochastically smaller than $\Pdec[n]$ for each $n$:
\begin{equation}
\label{stochineq}
\mbox{ $\Rdec[n] \led \Pdec[n]$ for every $n = 1,2, \ldots$.  \qquad (SD4)}
\end{equation}
That is to say, for each fixed $n$ it is possible to construct
a coupling of $\Rdec$ and $\Pdec$ with $\P( \Rdec[n] \le \Pdec[n]) = 1$. A stronger stochastic ordering condition, say SD7, with SD7 $\implies$ SD4, 
is that there exists a single coupling of $\Rdec$ and $\Pdec$ such that 
\begin{equation}
\label{alln}
\P(\Rdec[n] \le  \Pdec[n] \mbox{ for all }  n) = 1.  \qquad \mbox{(SD7)}
\end{equation}
It is easily shown that the refinement ordering SD1 $\implies$ SD7, but not conversely, due to the counterexample with $P_n$ and
$P_m$ mentioned above. It is also the case that the two variants of the stochastic ordering condition, SD4 with different couplings for different $n$, and 
SD7 with a single coupling for all $n$, are not equivalent.  This can be seen from the following simple example:
\begin{itemize}
\item Let $P = \Pdec$ be equally likely to be $(3,3,0)/6$ or  $(4,1,1)/6$.
\item Let $R = \Rdec$ be equally likely to be $(3,2,1)/6$ or  $(4,2,0)/6$.
\end{itemize}
Then $R[n] \ed P[n]$, hence $R[n] \led P[n]$, for each $n = 1,2,3$.
But it is impossible to couple $P$ and $R$ so that 
$\P(R[n] \le P[n] \mbox{ for } n = 1,2) = 1$.
For $\P( R[n] \le P[n] ) = 1$ and $R[n] \ed P[n]$ would imply $\P(R[n] = P[n]) = 1$ for $n = 1,2$, hence $\P(R[n] =  P[n] \mbox{ for } n = 1,2) = 1$,
hence $R \ed P$, which is obviously not the case.

It is easily shown that 
\begin{equation}
\label{rdecsq}
\mbox{ if $\Rdec[n] \le \Pdec[n]$ for all $n$, then $\sum_i (\Rdec_i)^2 \le \sum_i(\Pdec_i)^2$.}
\end{equation}
This is really a fact about
arbitrary fixed ranked distributions,  which applies also to random ranked distributions. To see \eqref{rdecsq}, for $0 \le \lambda \le 1$ consider
the convex combination 
$\Pdec(\lambda):= (1- \lambda) \Rdec + \lambda \Pdec$, which is evidently another ranked discrete distribution, and differentiate
$\sum_i \Pdec_i(\lambda)^2$ with respect to $\lambda$. This derivative is a linear function of $\lambda$, which is of the requisite positive sign for all $0 \le \lambda \le 1$
iff it is positive for $\lambda = 0$ and $\lambda = 1$. But that is easily checked using the condition that both $\Pdec$ and $\Rdec$ are ranked.
A connection with the convex order of means is that if $X$ has mean $0$ and finite mean square, then, as discussed further in Section \ref{sec:moms},
it is easily seen that
\begin{equation}
\label{secondmom}
\E ( M_P(X) ^2 ) = \E(X^2) \E \sum_i P_i^2 
\end{equation}
So a necessary condition for $M_R(X) \lecx M_P(X)$ for all $X$ with a finite mean is that
\begin{equation}
\E \sum_i R_i^2 \le  \E \sum_i P_i ^2 .
\end{equation}
This is obviously implied by the existence of a coupling  of $\Rdec$ and $\Pdec$ with $\sum_i(\Rdec_i)^2 \le  \sum_i(\Pdec_i)^2$,
as implied by \eqref{rdecsq}, but is clearly a lot weaker than that condition. 
Other necessary conditions for $M_R(X) \lecx M_P(X)$ are implied by the generalization of \eqref{secondmom} to higher powers
presented later in Corollary \ref{mainthm}.
So much remains to be clarified regarding these various orderings with respect to stochastic diversity.

\subsection{Reversed martingales in the Chinese Restaurant}
This section, which can be skipped at a first reading, explains how in the canonical construction  of $(\E X, M_P(X) , X )$ as a three term martingale,  
as in Proposition \ref{prp:canon}, the $X$ and $M_P(X)$ are the first term and the almost sure limit of the reversed martingale constructed in the following proposition.

\begin{proposition}
Let $(J_1, J_2, \ldots)$ be a random sample from a random discrete distribution $P$, with 
$(J_1, J_2, \ldots)$ and $P$, independent of the i.i.d. sequence $(X_1, X_2, \ldots)$.
Let $J_k^*$ the $k$th distinct value observed in the sequence $(J_1, J_2, \ldots)$, with $J_k = \infty$ if there is no such value.
Let
$$
P_{n,k}:= \frac{1}{n}\sum_{i = 1}^n (J_i = J_k^*) , 
$$
so $P_n  = (P_{n,k}, k = 1,2, \ldots )$ is the random empirical distribution of sample values $J_1, \ldots, J_n$ reindexed by
their order of appearance. For a measurable function $g$, let
$$
M_{P_n}(g(X)):= \sum_{k = 1}^\infty g(X_{J_k^*}) P_{n,k}  = \frac{ 1}{n} \sum_{i=1}^n g( X_{J_i} ) 
$$
so in particular $M_{P_1}(g(X)):= g(X)$ for $X:= X_{J_1} = X_{J_1^*}$. Then for each $g$ with $\E | g(X) | < \infty$ the sequence of 
$P_n$-means $M_{P_n}(g(X))$ is a  reversed martingale, which converges both almost surely and in $L^1$ to
$$
M_P ( g(X) ) :=  \sum_{k=1}^\infty g(   X_{J_k^*} ) P_k \ed \sum_{j=1}^\infty g(   X_j ) P_j .
$$
\end{proposition}
\begin{proof}
The equality of the two expressions for $M_{P_n} ( g(X) )$ follows easily from the definitions. 
The rest of the argument is a variation of the proof of Kingman's representation of partition structures by 
\citet{MR883646}. 
It is easily checked that the sequence $(X_{J_i}, i = 1,2, \ldots)$ is exchangeable, so $M_{P_n} ( g(X) )$ 
is a reversed martingale by standard theory of exchangeable sequences. The remaining conclusions follow easily.
\end{proof}

The Chinese Restaurant Process provides a visualization
of successive random partitions generated by the cycles of random permutations $\pi_n$ of $[n]$,
where $\pi_{n+1}$ is obtained from $\pi_n$ by inserting element $n+1$ into one of $n+1$ possible places relative to the
cycles of $\pi_n$.  Various aspects of this metaphor are developed in 
\citet[\S 3.1]{MR2245368} 
In terms of Chinese Restaurant, the random distribution $P_{n}$ with support $\{1, \ldots, K_n \}$ is the empirical distribution
of how the first $n$ customers are assigned to tables $j$ for $1 \le j \le K_n$, where $K_n$ the number of distinct values in the sample $J_1, \ldots, J_n$
from $P$.
In this picture, table $k$ is brought into service when the $k$th distinct value $J_k^*$ appears, and that $k$th table is labeled by the positive integer $J_k^*$.
The $(n+1)$th customer is given the random value $J_{n+1}$ picked from $(P_1, P_2, \ldots)$, and assigned to whichever table 
has label equal to $J_{n+1}$, if that label has appeared before, and otherwise, if there are $K_n = k$ tables in use, with $k$ different labels,
customer $n+1$ is assigned to a new table $k+1$ with value $J_{k+1}^* = J_{n+1}$.
Suppose that in addition to its index $k$ in order of appearance and its label $J_k^*$, the $k$th table is assigned value $X_{J_k^*}$ for $(X_1, X_2, \ldots)$
an i.i.d. sequence with values in an arbitrary measurable space, independent of $P$ and the sample $(J_1, J_2, \ldots)$ from $P$ which drives the Chinese Restaurant Process.
Say $X_{J_k^*}$ is the {\em table color} of the $k$th table brought into service in the restaurant.
Then the sequence of table colors  encountered by customers as they enter the restaurant, that is $(X_{J_1}, X_{J_2}, \ldots)$, is an exchangeable
sequence of random variables which generates a partition structure which may be coarser than the partition of customers by tables, if there are ties among the
$X$-values, but which will be identical to the partition of customers by tables if the distribution of $X$ is continuous so the $X$-values are almost surely distinct.
Note that the sequence $(P_j^*, j = 1,2, \ldots)$  is a size-biased random permutation of the
original random discrete distribution $(P_j)$ driving the Chinese Restaurant Process, by a mechanism that is independent of the $X$-sequence. 

\subsection{Fragmentation operators and composition of $P$-means}
\label{sec:frag}

\citet*[\S 6]{MR1386296} 
introduced the {\em composition operation} on two random discrete distributions $P$ and $Q$ which creates a new random discrete distribution $R:= P \ootimes Q$ as follows.
Let $P:= (P_i)$ be independent of $(Q_{i,j}, j = 1,2, \ldots)$, a sequence of i.i.d. copies of $Q$, and let $P \ootimes Q$ denote the ranked ordering
of the  collection of products $(P_i Q_{i,j}, i = 1,2, \ldots, j = 1,2, \ldots)$. Intuitively, each atom of $P$ is fragmented by its own copy of $Q$, and these fragments are reassembled in non-increasing order to form $R:= P \ootimes Q$. Clearly, $R$ is a very special kind of refinement  of $P$, as discussed in Section \ref{sec:refine}.
The composition operation $\ootimes$ may be regarded either as an operation on ranked discrete distributions, as in \citet*[\S 6]{MR1386296}, 
or on their corresponding partition structures, as detailed in 
\citet[Lemma 35]{MR1742892}. 

Independent of $(P_i)$ and $(Q_{i,j})$ as above, let $(X_{i,j})$ be an array of i.i.d. copies of $X$, assumed to be either bounded or non-negative.
Then
$$
M_Q^{(i)}(X):= \sum_{j=1}^\infty X_{i,j} Q_{i,j}
$$
is a sequence of i.i.d. copies of $M_Q(X)$. So a $P$-mean of $M_Q(X)$ is naturally constructed as
\begin{align}
M_P( M_Q(X) ) &= \sum_{i = 1}^ \infty M_Q^{(i)} P_i \\
&= \sum_{i = 1}^ \infty \sum_{j=1}^ \infty X_{i,j}P_i Q_{i,j} \, \ed \, M_{P \ootimes Q}(X) .
\end{align}

Hence the following proposition:

\begin{proposition}
The operation $P \ootimes Q$ of composition of random discrete distributions $P$ and $Q$ corresponds to composition of their mean operators $M_P$ and $M_Q$:
\begin{equation}
M_{P \ootimes Q} (X) \ed M_P( M_Q (X) )
\end{equation}
for all bounded or non-negative $X$. Consequently, for three random discrete distributions  $P$, $Q$ and $R$, the following two conditions are equivalent:
\begin{itemize}
\item  $M_R (X) \ed M_P (M_Q(X) )$ for every $X$ with a finite number of values;
\item  $R^\downarrow \ed P \ootimes Q$.
\end{itemize}
\end{proposition}
\begin{proof}
The first sentence summarizes the preceding discussion. 
The second sentence follows from the characterization of partition structures by their $P$-means (Corollary \ref{crl:means:kingman}).
\end{proof}

Typically, the operation of composition of random discrete distributions is quite difficult to describe explicitly.
A remarkable exception is the result of 
\citet*[Proposition 22]{MR1434129} 
that for the $P_{\alpha,\theta}$ governing the $(\alpha,\theta)$ model, 
there is the simple composition rule
\begin{equation}
\label{althcomp}
P_{\alpha, \theta} = P_{0,\theta} \ootimes P_{\alpha,0}
\qquad( 0 < \alpha < 1, \theta >0 )
\end{equation}
corresponding to the identity in distribution of corresponding $P$-means
\begin{equation}
\label{althcompmeans}
M_{\alpha, \theta} (X) \ed  M_{0,\theta} ( M_{\alpha,0} (X) )
\qquad( 0 < \alpha < 1, \theta >0 )
\end{equation}
for all bounded or non-negative random variables $X$.
See \citet[\S 3.4]{MR2245368} 
for an account of how the identity \eqref{althcomp} was first discovered by a representation of the $(\alpha,\theta)$ model
for $0 < \alpha < 1$ and $\theta >0$ as the limiting proportions of various classes of individuals in a continuous time branching process.
See also
\citet[Theorem 12]{MR1742892} 
for a proof of the more general result that 
\begin{equation}
\label{althcompmeans1}
P_{\alpha, \theta} = P_{\alpha \beta ,\theta} \ootimes P_{\alpha, - \alpha \beta} \qquad (0 < \alpha < 1, 0 \le \beta < 1, \alpha \beta < \theta ),
\end{equation}
which has a similar interpretation in terms of $P$-means. See also
\citet[\S 5.5]{MR2245368} 
for further discussion and combinatorial interpretations of \eqref{althcomp} and \eqref{althcompmeans1}.
As indicated in 
Section \ref{sec:twomeans}
these composition rules for $(\alpha,\theta)$ means are closely related to 
Tsilevich's formula \eqref{composmomalth0} for the generalized Stieltjes transform of an $(\alpha,\theta)$ mean.
See also 
\citet[Theorem 2.1]{MR2398765} 
where a presentation of \eqref{althcompmeans} 
was derived from Tsilevich's formula \eqref{composmomalth0}.
But the equivalence of \eqref{althcomp} and \eqref{althcompmeans} is only hinted at there,
by a reference to \cite{MR2183216}, 
which contains related results for interval partitions and random discrete distributions derived from self-similar random sets.

A result of \citet[Theorem 12]{MR1742892}. 
establishes a close connection between the operation of fragmentation of one random discrete distribution by another,
and a kind of dual coagulation operation. Curiously, while this coagulation operation has a simple description in terms of composition 
of associated  processes with exchangeable increments, it does not seem to have any simple description in terms of $P$-means.
See \citet[\S 5]{MR2245368} 
and \cite{MR2253162} 
for further discussion of fragmentation and coagulation operations and associated Markov processes whose state space is
the set of ranked discrete distributions.

\subsection{Moment formulas}
\label{sec:moms}

Let $(\tilX , \tilY): = M_P ( X,Y)$ be the pair of $P$-means of two random variables $X$ and $Y$ with some joint distribution.
It is a basic problem to calculate the expectation $\E \tilX \tilY$, in particular $\E \tilX^2 $ in the case $\tilX = \tilY$.
This problem was first considered by \citet{MR0350949} 
for the $(0,\theta)$ model of $P$. Following Ferguson's approach in that particular case,
expand the product as
$$
\tilX \tilY = \left( \sum_j X_j P_j  \right) \left( \sum_k X_k P_k  \right)  = \sum_{j} X_j Y_j  P_j^2  + \sum_{j \ne k } X_j Y_j P_j P_k
$$
and take expectations to conclude 
that
\begin{equation}
\label{fergprod}
\E \tilX \tilY = p(2) \E (X Y)  + p(1,1) \E(X ) \E(Y) 
\end{equation}
where
$$
p(2):=  \E \sum_j P_j^2 \mbox{ and } p(1,1) := \E \sum_{j \ne k } P_j P_k 
$$
are the two most basic partition probability formulas encoded in the EPPF $p$ derived from the random discrete distribution by \eqref{eppf}, that is
$$
\mbox{ $p(2) = \P(J_1 = J_2)$ and $p(1,1) = \P(J_1 \ne J_2)$  }
$$
for $(J_1, J_2)$ a sample of size 2 from $P$. 
In the Dirichlet case considered by \citet[Theorem 4]{MR0350949} 
$P$ is governed by the $(0,\theta)$ model, which makes $p(2) = 1/(1 + \theta)$ and $p(1,2) = \theta/(1 + \theta)$.

This method extends easily to a product of three $P$-means, say $\tilX \tilY \tilZ$, with a different sum appearing for each of the $5$
partitions of the index set $[3]$, according to ties between indices of summation:
$$
\E \tilX \tilY \tilZ =  \sum_{i = j = k} + \sum_{i,j,k {\rm distinct} } + \sum_{i=j \ne k } + \sum_{i=k \ne j } + \sum_{j=k \ne i }
$$
where for instance
$$
\sum_{i=j \ne k }  ~= ~\E ( X Y ) \E ( Z ) \E \sum_{i \ne k } P_i ^2 P_k ~= ~\E ( X Y ) \E ( Z ) p(2,1) 
$$
by \eqref{eppf}.
Continuing to a product of $n$ factors, the corresponding moment formula is given by the following proposition. This is 
a variant of product moment formulas due to \citet*[Proposition (10.1)]{MR1879065}, 
for the two-parameter model, and \citet*{MR2026070} 
for a general random discrete distribution $P$, possibly even defective, as in Section \ref{sec:defect}.

\begin{proposition}
{\em [Product moment formula for $P$-means]}
Let $(\tilY_i, 1 \le i \le n) = M_P(Y_1, \ldots,  Y_n)$ be the random vector of $P$-means derived from some joint distribution of $(Y_1, \ldots, Y_n)$.
For instance if $Y_i = g_i(X)$ for some sequence of measurable functions $g_i$ and some basic random variable $X$, then $\tilY_i:= \sum_{j} g_i(X_j)P_j$ for $(X_1, X_2, \ldots)$ a sequence of i.i.d. copies of $X$, 
independent of $P$ with EPPF $p$.
Then, assuming either the $Y_i$ are either all bounded, or all non-negative,
\begin{equation}
\label{prodmom}
\E \prod_{i=1}^n \tilY_i = \sum_{k=1}^n \sum_{\{B_1, \ldots , B_k \}} p(\#B_1, \ldots, \# B_k) \prod_{j = 1}^ k \mu(B_j )
\end{equation}
where $\# B$ is the size of block $B$ and $\mu(B):= \E \prod_{i \in B} Y_i$, and
where for each $k$ the inner sum is over the set of 
all partitions of $[n]$ into $k$ blocks $\{B_1, \ldots, B_k \}$.
\end{proposition}
\begin{proof}
Expand the product according to the partition generated by ties between indices. For each particular partition $\{B_1, \ldots, B_k \}$, the corresponding expectation
is evaluated using the basic formula \eqref{eppf} for the EPPF.
\end{proof}

Observe that  no matter what the joint distribution of the $Y_i$, 
if $\Pi_n$ is the random partition generated by a sample of size $n$ from $P$, and the definition of the product moment function $\mu(B)$
on subsets $B$ of $[n]$ is extended to a partition $\Pi = \{B_1, \ldots, B_k \}$ of $[n]$ by $\mu(\Pi):= \prod_{j=1}^k \mu(B_j)$, then the
product moment formula \eqref{prodmom} becomes simply:
\begin{equation}
\label{prodmom1}
\E \prod_{i=1}^n \tilY_i = \E \mu ( \Pi_n) .
\end{equation}
It is tempting to think this formula somehow evaluates $\E \prod_{i=1}^n \tilY_i $  by conditioning on $\Pi_n$ in a suitable construction of 
the product jointly with $\Pi_n$ to make $\E (\prod_{i=1}^n \tilY_i \giv \Pi_n) = \mu(\Pi_n)$, which would obviously imply \eqref{prodmom1}.
However this thought is completely wrong. Just consider the simplest case \eqref{fergprod} for $n=2$ for $X = Y$
with $\E(X) = \E(Y) = 0$. We know from examples that the distribution of $\tilX^2$ can be continuous, with $p(1,1) >0$.
But then there is no event $E$ with probability $p(1,1)$ such that $\E ( \tilX^2 \giv E) = \E(X) \E(Y) = 0$.

Be that as it may, the probabilistic form \eqref{prodmom1} 
of the product moment formula for $P$-means explains why this formula reduces easily in special cases, by manipulation of $\E \mu(\Pi_n)$.
For instance, if the joint distribution of $(Y_1, \ldots, Y_n)$ is exchangeable, then $\mu(B)$ depends only on $\#B$, say $\mu(B) = \mu(\# B)$ where
the definition of the moment function $\mu$ is extended to  positive integers $m$ by
$\mu(m):= \E \prod_{i= 1}^m Y_i$. That is, the mean product of any collection of $m$ of the variables. 
In this case, $\mu$ as a function of partitions of $[n]$ simplifies to $\mu ( \{B_1, \ldots, B_k \} ) = \prod_{j = 1}^k \mu(\# B_j)$.
This is a symmetric function of the sizes of the blocks of $\Pi_n$, which can be evaluated by listing the sizes of these blocks in any 
order, say $(N_{1:n}, N_{2:n} , \ldots, N_{K_n:n})$.
So for exchangeable $(Y_1, \ldots, Y_n)$ formula \eqref{prodmom1} becomes
\begin{equation}
\label{prodmom2}
\E \prod_{i=1}^n \tilY_i = \E \prod_{j=1}^{K_n} \mu( N_{j:K_n}) = \E \prod_{i=1}^n  \mu(i)^{c_i (\Pi_n) }
\end{equation}
where $\mu(m)$ is the expected product of any $m$ of the $Y_i$, and 
$$c_i(\Pi_n):= \sum_{j= 1}^{K_n} 1 ( N_{j:K_n} = i)$$ 
is the number of blocks of $\Pi_n$ of size $i$.
In the important special case when $Y_i \equiv  X$ for every $1 \le i \le n$, $\mu(m) = \E X^m$, and \eqref{prodmom2}
may be recognized in \citet[Theorem (4.2.2)]{MR1618739} 
in the equivalent form
\begin{equation}
\label{permsmoms}
\E  \tilX^n = \sum_{\pi } \P( \pi_n = \pi ) \prod_{i = 1}^n (\E X^i ) ^{c(i,\pi)}  
\end{equation}
where $\pi_n$ is a random permutation of $n$ which conditionally given $\Pi_n$ is uniformly distributed over all permutations of $[n]$
whose cycle partition is $\Pi_n$, as generated by the Chinese Restaurant Construction of $\Pi_n$, and $c(i,\pi)$  is the number of cycles of size $i$ in $\pi$.
See also \citet[\S 2]{MR1425401} 
where the formula \eqref{permsmoms} was first derived for the $(0,\theta)$ model of $P$ which generates the Ewens $(\theta)$ distribution on random permutations with
\begin{equation}
\label{ewensperm}
\P( \pi_n = \pi )  = \frac{ n! \theta^{K_n(\pi)}  } {(\theta)_n }
\end{equation}
for $K_n(\pi)$ the number of cycles of $\pi$.
Here is a version of Kerov's moment formula  \eqref{permsmoms}
in terms of the ECPF of $P$, as introduced in \eqref{nexnform}:

\begin{corollary}
{\em[Moment formula for $P$-means]}
\label{mainthm}
Let  $P$ be a random discrete distribution with ECPF $\pex$.
For every distribution of $X$ with $\E |X|^n < \infty$, the $n$th moment of $\tilX_P$, the $\Pbul$-mean of a sequence of i.i.d. copies of
$X$, is finite and given by the formula
\begin{equation}
\label{moments}
\E \tilX _P^n  = \sum_{k=1}^n \sum_{(n_1, \ldots, n_k) } \pex(n_1, \ldots, n_k) \prod_{i=1}^k \E X^{n_i} 
\end{equation}
where the inner sum is over all $\binom{n-1}{k-1}$ compositions of $n$ into $k$ parts.
In particular, if $\E \exp(t X) < \infty$ for $t$ in some open interval $I$ containing $0$, as for a bounded random variable $X$,  then
for every random discrete distribution $P$, 
\begin{itemize}
\item $\E \exp(t \tilX_P) \le  \E \exp(t X) < \infty$ for $t \in I$; 
\item the distribution of $\tilX_P$ is uniquely determined by its moment sequence \eqref{moments}.
\end{itemize}
\end{corollary}

\begin{proof}
For non-negative $X$, this is read from 
\eqref{prodmom2}  for $Y_i \equiv X$ and
the particular choice of the exchangeable random presentation
$\Nexn$ of sizes of blocks of $\Pi_n$.
Then take the usual difference $X = X_+ - X_-$ for signed $X$.
The rest is read from Corollary \ref{crl:convex} and standard theory of moment generating functions.
\end{proof}

A good check on this general moment formula for $P$-means is provided by taking $X$ to be the constant random variable $X = 1$ in \eqref{moments}. 
Then $\tilX = 1$ too, and the moment formula confirms that $\pex(n_1, \ldots, n_k)$ is a probability function on compositions of $n$ for each $n$,
as in \eqref{pexsum1}.
Another check is provided by the {\em classical case}, when $P = P_m$ say is constant, and equal to the uniform distribution on $[m]$.
The exchangeable composition probability function of $\Nexn$ is  then
\begin{equation}
\label{pexmform}
\pex_m(n_1, \ldots, n_k) = \left( \frac{1}{m} \right)^n  \binom{m}{ k } \binom{ n }{n_1, \ldots, n_k }  .
\end{equation}
The above moment formulas for $P$-means then reduce to classical formulas for moments of the arithmetic
mean of a sequence of i.i.d. random variables, discussed further in Section \ref{sec:arithmetic}.
The ECPF \eqref{pexmform}
can be derived quickly as follows
.
Each of $n$ balls indexed by $1 \le i \le n$ is equally likely to be painted 
any one of $m$ colors $j \in [m]$, and given there are $k$ different colors used, the clusters of balls by color are put in any one of $k!$ different orders by 
a uniform random permutation of $[k]$. Then $\pex_m(n_1, \ldots, n_k)$ is the probability that the sequence of cluster sizes $(n_1, \ldots, n_k)$ is achieved by this
random ordering.
But there are $k! \binom{m}{k}$ different ways to choose the sequence of $k$ different colors $(j_1,\ldots, j_k)$ generated by this ordering,
and for each of these choices of $k$ colors, the probability of the achieving the counts $(n_1, \ldots, n_k)$ by this sequence of colors, is the probability $1/k!$ that the particular
$k$ colors  are put in the desired order, times the multinomial probability of achieving counts $(n_1, \ldots, n_k)$ for these colors $(j_1, \ldots, j_k)$, and count $0$ for all other colors, 
in a simple random sample with replacement of $n$ colors from $[m]$.


\begin{problem}
Suppose that $\pex$ is a symmetric function of compositions $(n_1, \ldots, n_k)$ such that for some random discrete distribution $P$ 
the moment formula \eqref{moments} holds for all simple random variables $X$. If $\pex$ is known to be an ECPF, then $\pex = \pex_P$ the ECPF of $P$, by 
Corollary \ref{crl:means:kingman}. 
But this is not very obvious algebraically. 
What if
$\pex$ is not known to be an ECPF? Can it still be concluded that $\pex = \pex_P$? If not, what further side conditions (e.g. non-negativity) might be imposed to
obtain this conclusion?
\end{problem}

As a simple case in point, for each $m = 1,2, \ldots$, the classical moment formula for arithmetic means shows that the moment formula \eqref{moments} holds for all simple 
random variables $X$ and the function $\pex$ displayed in \eqref{pexmform}. 
Does formula \eqref{moments} alone imply that $\pex = \pex_m$ is in fact the ECPF for sampling from the uniform distribution
on $[m]$?
For small $n_1 + \cdots + n_k = 1,2,3,4$ it seems easy enough to conclude that by varying the distribution of $X$ over two values that there are enough independent linear equations to force 
$\pex(n_1, \ldots, n_k) = \pex_P(n_1, \ldots, n_k)$. But as $n$ increases, it seems necessary to involve three or more values of $X$, in which case the 
necessary linear independence of these equations does not seem to be obvious.

\subsection{Arithmetic means}
\label{sec:arithmetic}
The study of averages of i.i.d. random variables has a long history. Borel and Kolmogorov established almost sure convergence of 
$\tilX_m:= \sum_{j=1}^m X_j/m$ to $\E(X)$ as $m \to \infty$. In this instance,
$\tilX_m$ is the $P$ mean of $X$ for the non-random weights $P_j:= 1(j \le m)/m$ that are uniform on the set $[m]:= \{1,\ldots, m \}$, and it is assumed that $\E|X| < \infty$.
Characterizations of the exact distribution of $\tilX_m$ in terms of the distribution of $X$ are provided by the theory of moments, 
moment generating functions and characteristic functions, developed specifically for this purpose, as described in every textbook of probability theory.
For $X$ with a moment generating function (m.g.f.) $\E \exp(t X)$ that is finite for $t$ in some neighborhood of $0$, the m.g.f. of $m \tilX_m$  is 
\begin{equation}
\label{mgfs}
\E \exp( t m \tilX_m ) =  \E \exp \left( t \sum_{i=1}^m X_i \right) = \left(  \E \exp (t X ) \right)^m
\end{equation}
from which the $n$th moment of $m \tilX_m$ can be extracted by equating coefficients of $t^n$:
\begin{equation}
\label{mgfs1}
m^n \E \tilX_m ^n = n! \, [t^n]  \left( \sum_{j=0}^\infty \frac{ \E X^j }{j!} t^j \right)^m
\end{equation}
where $[t^n] g(t)$ is the coefficient of $t^n$ in the expansion of $g(t)$ in powers of $t$.
In expanding the product of $m$ factors on the right side of \eqref{mgfs1}, each product of terms contributing to the coefficient of $t^n$ involves 
some subset $I \subseteq [m]$ with say $\#I = k$ factors involving some $t^{n_i}$ with $n_i >0$ for $i \in I$ and $n_i = 0$ otherwise.
Hence, for all positive integers $m$ and $n$, the {\em classical moment formula} for the arithmetic mean of $m$ i.i.d. copies of
some basic variable $X$:
\begin{equation}
\label{classicalmoms}
\E  \tilX_m ^n = 
\left( \frac{1}{m} \right)^n 
\sum_{k=1}^n 
\sum_{(n_1, \ldots, n_k)  }
\binom{m}{k} 
\binom{n}{n_1, \ldots, n_k} 
\prod_{i=1}^k \E( X^{n_i} )
\end{equation}
where $(n_1, \cdots, n_k)$ ranges over the  set of $\binom{n-1}{k-1}$ {\em compositions of $n$ into $k$ parts}, that is sequences of $k$ positive integers with sum $n$.
The term indexed by $(n_1, \cdots, n_k)$ is a symmetric function of $(n_1, \cdots, n_k)$, which remains unchanged if $(n_1, \cdots, n_k)$ is replaced by its {\em non-increasing rearrangement}
$(\ndec_1, \cdots, \ndec_k)$, called a {\em partition of $n$}. This partition of $n$ is often encoded by the sequence of counts 
$$c_j:= \sum_{i=1}^n 1(n_i = j) = \sum_{i=1}^n 1(\ndec_i = j)$$
for
$1 \le j \le n$, in terms of which $k = \sum_j c_j$ and $\sum_j j c_j$, and the right side  of \eqref{classicalmoms} involves
$$
\binom{n}{n_1, \ldots, n_k} \prod_{i=1}^k \E( X^{n_i} ) = n!  \prod_{j=1}^n \left( \frac{ \E X^j }{j! }\right) ^{c_j} .
$$
So the classical moment formula may be rewritten as a sum over partitions of $n$ with
a multiplicity factor counting the number of compositions for each partition, or as a similar sum over permutations of $[n]$, with a different multiplicity factor, using the cycle structure of the permutations 
to index partitions of $n$.

The classical moment formula shows explicitly how the moments of $\tilX_m$ are determined by those of $X$, in the first instance for  $X$ with a m.g.f. that converges in a neighborhood of $0$.
But then, by standard arguments involving formal power series, the formula holds also for every $X$ with $\E |X|^n < \infty$.
Instances and applications of this formula are well known. For instance, the case $n = 2$ of \eqref{classicalmoms} gives
\begin{equation}
\label{classicalmoms2}
\E  \tilX_m ^2 = \frac{ \E (X^2) }  {m}  \mbox{ if } \E(X^2 ) < \infty \mbox{ and } \E(X) = 0,
\end{equation}
hence the weak law of large numbers for such $X$, by Chebychev's inequality. And the case $n = 4$ of \eqref{classicalmoms} gives
\begin{equation}
\label{classicalmoms4}
\E  \tilX_m ^4 = \frac{1}{m^4} \left( m \E(X^4) + 3! \binom{m}{2} ( \E(X^2) ^2 \right) \mbox{ if } \E X^4 < \infty \mbox{ and } \E(X) = 0,
\end{equation}
hence the strong law of large numbers for such $X$, by Chebychev's inequality and the Borel-Cantelli Lemma  
\citep[Theorem 2.3.5]{MR2722836}. 
The classical moment formula \eqref{classicalmoms} and its variant with summation over partitions have been known for a long time. 
It was used already by Markov in one of the first proofs of the central limit theorem. See e.g. 
\citet[Appendix II]{uspensky1937introduction}. 
It was also used by 
\citet{MR0214150} 
to establish the Gaussian nature of increments in 
his proof of L\'evy's martingale characterization of Brownian motion.  See also \citet{ferger2014moment} for a recent discussion without acknowledgement
of the classical literature.


The above derivation of moments of the arithmetic mean $\tilX_m$ of a sequence of i.i.d. copies of $X$ can be adapted to $P$-means by first conditioning on $P$.
This gives
$$
\E ( \tilX_P ^n ) =  \E  \left[ \E ( \tilX_P ^n |  P ) \right] =  n! \, [t]^n \, \E  \prod_{j=1}^\infty \left( 1 + \frac{P_j  \E(X) t }{1!} + \frac{P_j^2 \E(X^2) t^2 }{2! } + \cdots \right)
$$
Now the coefficient of $t^n$ involves expanding the infinite product, picking out some finite number $k$ of the factors, say those indexed by $j_i$ 
factors of $t^{n_i}$ with $n_i >0$, for $1 \le i \le k$,
and then summing over all choices of $(j_1, \ldots, j_k)$ and all compositions $(n_1, \ldots, n_k)$ of $n$. This provides another proof of
the moment formula for $P$-means \eqref{moments}.

\subsection{Improper discrete distributions}
\label{sec:defect}
\citet{MR509954} 
showed that to provide a general representation of sampling consistent families of random partitions of positive integers $n$, 
it is necessary to treat not just sampling from random discrete distributions $(P_i)$ with $P_i \ge 0 $ and $\sum_{i} P_i = 1$,
but also to consider sampling from $(P_i)$ with $P_i \ge 0 $ and $\sum_{i} P_i \le 1$. This more general model may be interpreted to mean that the 
$P_i$ with $P_i >0$ are the jumps of some random distribution function $F$, but that $F$ may also have a continuous component whose total mass is the {\em defect}
\begin{equation}
\label{eq:defect}
P_\infty:= 1 - \sum_i P_i \ge 0.
\end{equation}
Call $P$ {\em proper} iff $P_\infty = 0$, and {\em defective} or {\em improper} if $P_\infty >0$.
It was shown in
\citet[Proposition 26]{MR1742892} 
how improper random discrete distributions arise naturally in the study of random coalescent processes. See
\cite[\S 3]{mohle2010asymptotic} 
and work cited there for more recent developments in this vein.

\citet{MR1618739} 
indicated the right generalization of the definition of the $P$-mean $M_P(X)$ to defective random discrete distributions $P$. Restrict
discussion to $X$ with $\E|X| < \infty$, and set
\begin{equation}
\label{mpxdefect}
M_P(X):= \sum_{j} X_j P_j \, \,  + P_\infty \E X 
\end{equation}
for $(X_j)$ as usual a sequence of i.i.d. copies of $X$.
This definition is justified by the way that defective distributions of $P$ arise as weak limits of proper discrete distributions.
For instance, if $P_m$ is the uniform distribution on $[m]$ as in the previous section, then $P_m \convd P:= (0,0, \ldots)$ as $m \to \infty$, 
in the sense of convergence of finite dimensional distributions. In this case the limit $P$ has $P_\infty = 1$,
and Kolmogorov's law of large numbers gives $M_{P_m}(X) := m^{-1}\sum_{i=1}^m X_i \to \E(X)$ almost surely. This justifies the definition \eqref{mpxdefect} in the extreme
case $P_j \equiv 0$ and $P_\infty = 1$.
More generally, it is known \citep{MR0195135} 
that if $(a_{n,k})$ is a Toeplitz summation matrix (i.e., $\lim _n  a_{n,k} = 0$ for each $k$, $\lim_n \sum_{k} a_{n,k} =1$, and $\sum_{k}|a_{n,k}|$ is bounded in $n$),
and $\tilX_n := \sum_k a_{n,k} X_k$, then  for any non-degenerate distribution of $X$ with $\E|X| < \infty$, there is convergence $\tilX_n \to \E(X)$ in probability iff
$\max_k |a_{n,k}| \to 0$ as $n \to \infty$.
As an easy consequence of this fact, there is the following proposition, whose proof is left to the reader:

\begin{proposition}
\label{prp:lim}
Assume $E |X| < \infty$. Let $P_n$ be a sequence of proper discrete distributions, with $\Pdec_n \convd \Pdec$, meaning that
the finite-dimensional distributions of $\Pdec_n$ converge in distribution to those of $\Pdec$, for $\Pdec$ some possibly improper random discrete distribution.
Then $M_{P_n}(X) \convd \tilX:= M_{\Pdec}(X)$ defined by \eqref{mpxdefect}.
Moroever, this conclusion continues to hold for a sequence of possibly defective discrete distribution $P_n$, provided \eqref{mpxdefect} is taken as the 
definition of $M_{P_n}(X)$.
\end{proposition}

In other words, for $X$ with $\E|X|<\infty$, the definition \eqref{mpxdefect} is the only definition of $M_P(X)$ which agrees 
with the definition in the proper case, and which makes $\Pdec \mapsto M_{\Pdec}(X)$  weakly continuous as a mapping from laws of possibly
defective random ranked discrete distributions $\Pdec$ to laws of $M_{\Pdec}(X)$.
Beware that the above proposition is false if the assumption $\Pdec_n \convd \Pdec$ is replaced by
$P_n \convd P$: just take $P_n$ to be certain to be a unit mass at $n$. Then  $P_n \convd (0,0,\ldots)$, but $M_{P_n}(X) \ed X$ for every $n$, which does 
not converge to $\E X$ unless $X$ is constant.

For more about improper discrete distributions, and the tricky issue of extending the notion of a size-biased permutation to
this case, see \citet{MR1659532}. 

\section{Models for random discrete distributions}
\label{sec:models}

This section recalls some of the basic models for random discrete distributions. 
These models all arose from applications of random discrete distributions, and
spurred the development of a general theory of distributions of $P$-means and its relation to partition structures.

\subsection{Residual allocation models.}
\label{sec:residual}
Consideration of $P$-means by splitting off the first term, suggests that their study should be simplest for those $P$ which can be presented
in some order by a {\em residual allocation model}, or {\em stick-breaking scheme}, involving a recursive splitting like \eqref{splitser}.
That is, assuming the terms of $P$ have already been put in the right order for such a recursion,
there is the {\em stick-breaking representation}
\begin{equation}
\label{stickbreak}
P_j = H_j \prod_{i=1}^{i-1} ( 1 - H_i) \qquad (j = 1,2, \ldots )
\end{equation}
for a sequence of independent {\em stick-breaking factors} $H_i$ with $H_i \in [0,1]$.
\citet{MR0158483} 
studied Bayesian estimation for such $P$ given a sample $J_1, \ldots, J_n$ from $P$,
assuming the stick-breaking representation \eqref{stickbreak}
for $H_i$ such that 
$$(H_1, \ldots, H_N), H_{N+1}, H_{N+2}, \ldots$$ 
are independent for some fixed $N \ge 0$. Freedman called such distributions of $P$ {\em tail-free}.
\citet{MR2735350} 
provide an extensive account of the distribution theory
of a sample $(J_1, \ldots, J_n)$ from a residual allocation model with i.i.d. factors, calling this model for $(J_1, \ldots, J_n)$ the {\em Bernoulli sieve}.

Assuming the stick-breaking form \eqref{stickbreak} for $P:= (P_1, P_2, \ldots)$ derived from $(H_1, H_2, \ldots)$,
let $R:= (R_1, R_2, \ldots)$ be the {\em residual random discrete distribution} defined derived correspondingly from $(H_2, H_3, \ldots)$.
Then, assuming only that $H_1$ is independent of $(H_2, H_3 \ldots)$,
for $M_P(X)$ the $P$-mean of a sequence of i.i.d. copies of $X$, there is the decomposition
\begin{equation}
\label{mpxsplit}
M_P(X) \ed P_1 X_1 + (1- P_1) M_R (X) 
\end{equation}
where on the right side, $P_1$, $X_1$ and $M_R(X)$ are independent, with $X_1 \ed X$.
The case of independent stick-breaking when $P_1 \ed \beta_{r,s}$ for some $r,s >0$ is of particular
interest, due to the ease of computation of moments of $M_P(X)$ in this case.
Multiply  \eqref{mpxsplit} by  an independent $\gamma_{r+s}$ variable, and appeal to the beta-gamma algebra 
\eqref{betagamindpt}
to see that \eqref{mpxsplit} for $P_1 \ed \beta_{r,s}$
implies
$$
\gamma_{r+s} M_P(X) \ed \gamma_r X_1 + \gamma'_s  M_R(X)
$$
where on the right side, $X_1$ and $M_R(X)$ are independent, independent also of $\gamma_r$ and $\gamma'_s$, which are independent gamma variables with the indicated parameters.
In terms of moment generating functions, this becomes
$$
\E \exp [ \lambda \gamma_{r+s} M_P(X) ] = \E \exp [ \lambda \gamma_r X_1  ] \E \exp[ \lambda \gamma_s M_R(X) ] .
$$
That is, by conditioning on all except the gamma variables,,
\begin{equation}
\label{betars}
\E ( 1 - \lambda M_P(X) )^{-(r+s)}  = \E (  1 - \lambda X_1  ) ^{-r} \E ( 1 - \lambda M_R(X)  )^{-s}  .
\end{equation}
For instance, if $X_p:= 1(U \le p)$ is an indicator variable of an event with probabilty $p$, 
and $P_1 \ed \beta_{r,s}$ is independent of the residual fractions $(R_2. R_3, \ldots)$, then
\begin{equation}
\label{betars1}
\E ( 1 - \lambda M_P(X_p) )^{-(r+s)}  = (1 - p + p(1- \lambda)^{-r} ) \E ( 1 - \lambda M_R(X_p)  )^{-s}  .
\end{equation}
Formula \eqref{betars} is a generalization of Proposition 3 of \citet{MR2177313}, 
which is the particular case with $r = 1$ and $s = \theta >0$
of greatest interest in Bayesian non-parametric inference.
See also Proposition 4 of \citet{MR2177313} which gives the corresponding expression in terms of moments.

For an i.i.d. stick-breaking scheme, with factors $H_i \ed P_1$ for all $i$, formula \eqref{mpxsplit} holds with $R \ed P$,  implying that the distribution of 
$\tilX:= M_P(X)$ solves the stochastic equation
\begin{equation}
\label{mpxsplit1}
\tilX \ed P_1 X + (1- P_1) \tilX .
\end{equation}
where on the right side  $P_1$, $X$ and $\tilX$ are independent.
As shown by 
\citet{feigin1989linear} 
and \citet{MR1669737}, 
this stochastic equation uniquely determines the  distribution of $\tilX$ under mild regularity conditions.
See \citet[Proposition 9]{MR2177313} regarding the important case of the $(0,\theta)$ model with $P_1 \ed \beta_{1,\theta}$ for some $\theta >0$.  

\subsection{Normalized increments of a subordinator} 
\label{sec:subord}

A well known method of construction of random discrete distributions $P = (P_1,P_2, \ldots)$ is to  start from a sequence of 
non-negative random variables $(A_1, A_2, \ldots )$, and then normalize these variables by their sum $A_\Sigma$:
\begin{equation}
\label{relabunds}
(P_1, P_2, \ldots ) := \frac{1 }{A_\Sigma} (A_1, A_2, \ldots ) \mbox{ where } A_\Sigma= \sum_{i=1}^\infty A_i  .
\end{equation}
Here it is assumed that $\P(A_\Sigma >0 ) = 1$, which provided $\P(A_i >0) >0$ for some $i$ can 
always be arranged by conditioning on the event $(A_\Sigma >0)$.
Say $(P_1, P_2, \ldots )$ is {\em derived from increments of a subordinator} $(A(r), 0 \le r \le \theta)$, where $\theta >0$,
if $A(\bullet)$ is an increasing process with stationary independent increments,
and the $A_i$ are the independent increments of $A(\bullet)$ over consecutive intervals of lengths $\theta_i$ with $\sum_i \theta_i = \theta$.
The normalizing factor $A_\Sigma$ in \eqref{relabunds} is then $A_\Sigma = A(\theta)$.  

A closely related, but more important construction, with the same normalizing factor $A(\theta)$, is obtained by
supposing that $A_i = A_i(\theta)$ in \eqref{relabunds} are some exhaustive list of the jumps $\Delta A(r):= A(r) - A(r-)$ with 
$\Delta A(r) >0$ and $0 \le r \le \theta$, for a subordinator with no drift component, meaning that almost surely
\begin{equation}
\label{subjumps}
A(\theta) = \sum_{0 < r \le \theta}  \Delta A (r) = \sum_{i = 1}^\infty A_i (\theta) .
\end{equation}
Precise definition of the $A_i(\theta)$ and the corresponding $P_i(\theta)$ in \eqref{relabunds} requires an ordering for these jumps. However, 
according to Corollary \ref{crl:means:kingman},
the distribution  of $P$-means $M_P(X)$, and all other aspects of the partition structure derived from $P$, do not depend on what ordering of jumps is chosen.
As shown by L\'evy's analysis of occupation times of Brownian motion, it may be possible to identify the distributions of various $P$-means by suitable decompositions
like \eqref{splitser1}, even without fully specifying the ordering in a construction of $P$ from a countable collection of interval lengths. Historically, this was done by
L\'evy and Lamperti, decades before analysis of the size-biased orderings of jumps of a subordinator by McCloskey, and the ranked jumps by
\citet*{MR0373022} 
and \citet{MR0368264}. 

According to the L\'evy-It\^o theory of subordinators, the jumps $A_i(\theta)$ in \eqref{subjumps} are the points of a  Poisson point process on $(0,\infty)$
\begin{equation}
\label{levypois}
N_\theta(\bullet):= \sum_{0 < r \le \theta} 1 ( \Delta A_r \in \bullet) = \sum_{i = 1} ^\infty 1( A_i (\theta) \in \bullet)
\end{equation}
with intensity measure $\Lambda(\bullet)$, for some L\'evy measure $\Lambda$ on $(0,\infty)$, which is uniquely determined by the L\'evy-Khintchine 
representation of the {\em Laplace exponent} of the subordinator
\begin{equation}
\label{levexp}
\Phi(\lambda ) := \int_0^\infty ( 1 - e^{- \lambda x } ) \Lambda (dx) \qquad (\lambda \ge 0 )
\end{equation}
with
\begin{equation}
\label{lkrep}
\E \exp[ - \lambda A( t ) ] = \exp [ - t \Phi(\lambda ) ]  \qquad (t \ge 0, \lambda \ge 0 ).
\end{equation}
The joint law of ranked jumps $\Adec(\theta)$ is then easily read from the Poisson description of the associated counting process  \eqref{levypois},
as detailed in 
\citet*{MR0373022}. 
More or less explicit descriptions of the finite dimensional distributions of $(\Pdec_j(\theta), j = 1,2, \ldots)$ are known.
See 
\citet*[Proposition 22]{MR1434129} 
which reviews earlier work on ranked discrete distributions.
But to derive partition probabilities or distributions of $P$-means,  ranked discrete distributions are impossible to work with. 
For such purposes, a much better ordering is the size-biased ordering $\Pst$ introduced in this setting  by \citet{mccloskey}.
McCloskey imagined each $A_i(\theta)$ to be a Poisson intensity rate of trapping, called the {\em abundance} of some species labeled by $i$, in a species sampling 
model driven by a collection of independent Poisson point processes of random rates $A_i(\theta)$, for some fixed parameter value $\theta >0$.
McCloskey showed that for $A_i(\theta)$ the jumps of a standard gamma process $(\gamma(r), 0 \le r \le \theta)$, in the size-biased order of their
discovery in the Poisson species sampling model, the resulting random discrete distribution $P^*$ has i.i.d. beta$(1,\theta)$ distributed residual fractions,
and that beta$(1,\theta)$ is the only possible distribution of i.i.d. residual fractions which generates a random discrete distribution with its components
in size-biased random order. Later work showed that this GEM$(0,\theta)$ model for $P^*$ introduced by McCloskey is the size-biased presentation of limit frequencies
associated with the limit model proposed earlier by \citet{fisher1943solo}, 
with partition probabilities governed by the Ewens sampling formula.
Before discussing the GEM$(0,\theta)$ this model in more detail, the following proposition presents a fundamental connection between the more elementary model
\eqref{relabunds} with $(P_1, P_2, \ldots)$ the normalized increments of some subordinator $A(\bullet)$ over some fixed sequence of intervals
of lengths $\theta_i$ with $\sum_i \theta_i = \theta$, and the model obtained from the same subordinator by some ordering of its relative jump sizes.

\begin{proposition}
\label{propthetabull}
Let $P_\theta(\bullet):= \sum_{j} 1(Y_j \in \bullet) P_j(\theta)$ be the random probability measure on an abstract space $(S,\SS)$ defined as in \eqref{sprinkle} by assigning i.i.d. 
random locations $Y_i$ to each normalized jump $P_i(\theta)$ of a subordinator up to time $\theta$.
Then for every ordered partition $(S_1, S_2, \ldots )$ of $S$ into disjoint measurable subsets with $\theta \P(Y_j \in  S_i) = \theta_i$, there is the
equality in distribution of discrete random distributions on the positive integers
\begin{equation}
\label{identnorm}
(P_\theta( S_i), i = 1,2, \ldots ) \ed ( A_i(\theta_i) /A(\theta), i = 1,2, \ldots )
\end{equation}
where on the right side the $A_i(\theta_ii)$ are the independent increments of the subordinator $A$ over a partition of $[0,\theta]$ into a succession  of disjoint intervals
of lengths $\theta_i$ with $\sum_{i}  \theta_i = \theta $, that is $A_i(\theta_i):= A( \Sigma_{h=1}^{i}  \theta_i  ) - A( \Sigma_{h=1}^{i-1}  \theta_i  )$. 
\end{proposition}
\begin{proof}
This is a straightforward consequence  of standard marking and thinning properties of Poisson 
point processes, which make the $(T_i(\theta), A_i(\theta), Y_i)$ the points of a Poisson process on $[0,\theta] \times (0,\infty) \times S$ 
with intensity $dt\, \Lambda(da) \,\P(Y \in ds)$, where $T_i(\theta)$ is the arrival time in $[0,\theta]$ of the jump of the subordinator with
magnitude $A(T_i(\theta)) - A(T_i(\theta)-) = A_i(\theta)$.
\end{proof}

This proposition yields a fairly explicit description of the finite dimensional distributions of 
the random measure $P_\theta(\bullet)$ on $S$, as well as the distribution of various $P$-means:

\begin{corollary}
\label{corthetabull}
Let $P(\theta):= (P_j(\theta), j = 1,2, \ldots)$ be the sequence of normalized jumps of a subordinator $(A(r), 0 \le r \le \theta)$ governed by a L\'evy measure $\Lambda$ with
infinite total mass. Then every discrete random variable $X:= \sum_{i} a_i X_{p_i}$, with distinct possible values $x_i$,
and $X_{p_i}$ the Bernoulli$(p_i)$ indicators of disjoint events $(X = x_i)$ with $p_i:= \P(X = x_i)$ subject to $\sum_{i} p_i= 1$,
the distribution of $M_{P(\theta)}(X)$, the $P(\theta)$-mean of a sequence of i.i.d. copies of $X$ independent of $P(\theta)$,
is determined by the equality in distribution
\begin{equation}
\label{mptheta}
M_{P(\theta)} \left( \Sigma_{i} x_i X_{p_i} \right) \ed \frac{1}{A(\theta)} \sum_{i} x_i  A_i(\theta\, p_i)
\end{equation}
where the right side is a corresponding normalized linear combination of independent increments $A_i(\theta p_i)$ of the subordinator $A$ over a partition of $[0,\theta]$ 
into disjoint intervals, as in \eqref{identnorm}. If $X$ has an infinite number of possible values, \eqref{mptheta} means that
if either side is well defined by almost sure absolute convergence, then so is the other, and the distributions of both sides are equal.
\end{corollary}
\begin{proof}
The case of a finite sum is read immediately from the previous proposition. The case of infinite sums then follows by an obvious
approximation argument.
\end{proof}

These distributions of $P$-means can be described much more explicitly in the particular cases of gamma and stable subordinators,
as discussed further below. 
See also 
\citet*{MR1983542}, 
regarding more general subordinators.

\subsection{Dirichlet distributions and processes.}
\label{sec:dirichlet}

The model for a random discrete distribution derived from normalized increments of a subordinator is 
of special interest for the {\em standard gamma subordinator} $A(r) = \gamma(r)$ for $r >0$,
defined by the standard gamma density \eqref{gamdens}. 
The convolution property of gamma distributions, that
$$
\gamma (r) + \gamma'(s) \ed \gamma(r+s)
$$
for independent gamma variables of the indicated parameters $r,s >0$, is part of the basic {\em beta-gamma algebra}
\eqref{betafromgam}-\eqref{betagamindpt} which underlies all the following calculations with the gamma process.
First of all, this property allows the construction of the standard gamma subordinator with stationary independent increments.
For any subordinator $A$, it is known \cite[Corollary 8.9]{MR1739520} 
that for each continuity point $\epsilon >0$ of its L\'evy measure $\Lambda(\bullet)$, the
restriction of $\Lambda(\bullet)$ to $(\epsilon, \infty)$ is the weak limit as $r \downarrow 0$ of the 
same restriction of the measure $r^{-1} \P( A(r)  \in \bullet )$. 
For the gamma density \eqref{gamdens}, in this limit there
is the pointwise convergence of densities 
at each $x >0$
$$
\frac{ \P( \gamma(r)  \in dx )}{r} = 
\frac{ x^{r-1} e^{-x} }
{ r \Gamma(r) }
\to  x^{-1} e^{-x}  \mbox{ as } r \downarrow 0
$$
because $r \Gamma(r) = \Gamma(r+1) \to \Gamma(1) = 1$. This identifies the L\'evy measure of the gamma process
\begin{equation}
\Lambda_{\gamma} (dx) = x^{-1} \,e^{-x } \, 1(x >0 ) \,dx 
\end{equation}
hence the L\'evy-Khintchine exponent
\begin{equation}
\label{levexpgam}
\Phi\gamma(\lambda ) = \int_0^\infty ( 1 - e^{- \lambda x } ) x^{-1} e^{-x} dx  = \log ( 1 + \lambda )
\qquad (\lambda \ge 0 )
\end{equation}
which is a Frullani integral. The corresponding Laplace transform is obtained more easily by integration with respect to
the gamma$(r)$ density \eqref{gamdens}:
\begin{equation}
\label{gamlt}
\E \exp[ - \lambda \gamma( \theta ) ] = \exp [ - \theta \Phi_\gamma(\lambda ) ]   = (1 +  \lambda)^{-\theta} \qquad (\theta  \ge 0, \lambda \ge 0 ).
\end{equation}
The negative binomial expansion of this Laplace transform in powers of $-\lambda$ encodes the moments of $\gamma(\theta)$:
\begin{equation}
\label{gamnegbin}
\sum _{n=0}^\infty \E  \gamma(\theta )^n  \frac{\lambda^n}{n!}  =  \sum_{n=0}^\infty \frac{(\theta)_n}{n!} \lambda^n = (1 - \lambda )^{-\theta} \qquad ( |\lambda| < 1, \theta >0 ).
\end{equation}
Hence, by equating coefficients of $\lambda^n$, the list of integer moments of a gamma$(\theta)$ variable:
\begin{equation}
\label{gammoms1}
\E  \gamma(\theta)^n  = (\theta)_n  := \frac{ \Gamma(\theta + n )} {\Gamma(\theta) } = \prod_{i=1}^{n} (\theta + i -1 ) \qquad (n = 0,1,2, \ldots).
\end{equation}
Apart from the last equality, this moment evaluation holds also for all real $n > - \theta$, by direct integration and the definition of the gamma function.
Easily from \eqref{gammoms1} by beta-gamma algebra, or by direct integration, there is the corresponding beta moment formula:
\begin{equation}
\label{betamoms}
\E \beta_{r,s}^n ( 1  - \beta_{r,s} ) ^m = \frac{ (r)_n (s)_m } {(r + s )_{n+m} }
\end{equation}
where for non-negative integers $r$ and $s$, 
the right side involves just factorial powers of $r$, $s$ and $r+s$,
and  the formula extends to all real $n > - r$ and $m > - s$ with the general definition  \eqref{gammoms1} of the {\em Pochhammer symbol} $(\theta)_n$.
This Pochhammer symbol, appearing in most formulas involving Dirichlet distributions with total weight $\theta$, is often best understood through beta-gamma algebra as 
the $n$th monent of a gamma$(\theta)$ variable, that is the magic multiplier which makes the Dirichlet components independent.

The {\em Dirichlet distribution of $P$ with weights $(\theta_1, \theta_2, \ldots)$} is the
distribution obtained as $P_i:= A_i/A(\theta)$
from the normalized subordinator increments construction \eqref{relabunds}, with independent $A_i \ed \gamma(\theta_i)$ 
for some $\theta_i \ge 0$ with $\theta:= \sum_i \theta_i >0$, so $A(\theta) \ed \gamma(\theta)$.
The {\em finite Dirichlet $(\theta_1, \ldots, \theta_m)$} distribution of $P$,
is the distribution of $(P_1, \ldots, P_m)$ on the $m$-simplex $\sum_{i=1}^m P_i = 1$ so obtained by taking $\theta_i = 0 $ for $i > m$.
This distribution can be characterized in a number of different ways. For instance, 
by the joint density  of $(P_1, \ldots , P_{m-1})$ at $(u_1, \ldots, u_{m-1})$ relative to  Lebesgue  measure in $\reals^{m-1}$, which is 
$$
\P(P_1 \in du_i, 1 \le i \le m-1) = \frac{1}{\Gamma(\theta) } \prod_{i=1}^m \frac{u^{\theta_i  - 1} }{\Gamma(\theta_i)} 1 \left( 0 \le u _i \le 1, \sum_{i=1}^n u_i = 1  \right)
$$
or by its product moments
$$
\E \prod_{i=1}^m P_i^{n_i} = \frac{ \prod_{i = 1}^m (\theta_i)_{n_i} }{ ( \theta )_n } \mbox{ for } n_i \ge -\theta_i \mbox{ with } \sum_{i=1}^m n_i = n 
$$
which are easily obtained by beta-gamma algebra, like the case \eqref{betamoms} for $m = 2$.

The {\em symmetric Dirichlet distribution with total weight $\theta$}, denoted here by \\
Dirichlet$(m||\theta)$, is the particular case with $\theta_i \equiv \theta/m$ for $1 \le i \le m$.
As examples:
\begin{itemize}
\item the distribution of the $m$ consecutive spacings between order statistics of $m-1$ independent uniform 
$[0,1]$ variables is the Dirichlet$(m||m)$ distribution with $m$ weights equal to $1$. 
\item
For any integer composition $(m_1, \ldots, m_k)$ of $m$, 
a finite Dirichlet $(m_1, \ldots, m_k)$ random vector can then be constructed from suitable disjoint
sums of terms in a  Dirichlet$(m||m)$ random vector, by property (ii) in the following proposition.
\end{itemize}

This proposition summarizes some well known properties of the Dirichlet model for $P$.

\begin{proposition}
Let $P:= (P_j, j \ge 1)$ have the Dirichlet distribution with weights $(\theta_1, \theta_2, \ldots)$ defined by
the normalization $P_j:= A_j/\gamma(\theta)$ as in  \eqref{relabunds} for a sequence of independent gamma$(\theta_j)$ variables $A_j$ with total $\sum_j A_j = \gamma(\theta)$.
For a set of positive integers $B$, let $P(B):= \sum_{j\in B} P_j$.
Then
\begin{itemize}
\item [(i)] the sequence of ratios $(P_1, P_2, \ldots )$ is independent of the total $\gamma(\theta)$.
\item [(ii)] For each partition of positive integers into a finite number of disjoint subsets $B_1, \ldots, B_m$, 
the distribution of $(P(B_i), 1 \le i \le m)$ is the finite Dirichlet $(\theta P(B_i), 1 \le i \le m)$ distribution on the $m$-simplex.
\item [(iii)] In particular, the distribution of $P(B)$ is beta$(\theta P(B), \theta - \theta P(B) )$.
\item [(iv)] This model is identical to the residual allocation model \eqref{stickbreak} with independent beta distributed factors
\begin{equation}
\label{betafordir}
H_j \ed \beta_{\theta_j, \sigma_j} \mbox{ with } \sigma_{j}:= 
\theta - \sum_{i=1}^j \theta_i 
= 
\theta_{j+1} + \theta_{i+2} + \cdots 
\end{equation}
\end{itemize}
\end{proposition}
\begin{proof}
Straightforward applications of the basic beta-gamma algebra \eqref{betafromgam}-\eqref{betagamindpt}.
\end{proof}

These definitions and properties of Dirichlet distributions allow 
Proposition \ref{propthetabull} and its corollary to be combined and restated as follows, 
for the Dirichlet random discrete distributions on abstract spaces introduced by 
\citet{MR0350949}. 

\begin{proposition}
\label{fergdir}
Let $P_\theta(\bullet):= \sum_{j} 1(Y_j \in \bullet) P_j(\theta)$ be the random probability measure on an abstract space $(S,\SS)$ defined as in \eqref{sprinkle} by assigning i.i.d. 
random locations $Y_j$ to each normalized jump $P_j(\theta)$ of a standard gamma subordinator up to time $\theta$.
Then for every ordered partition $(S_1, S_2, \ldots )$ of $S$ into disjoint measurable subsets with $\theta \P(Y_j \in  S_i) = \theta_i$, the 
sequence $(P_\theta( S_i), i \ge 1)$ has the Dirichlet distribution with parameters $(\theta_i, i \ge 1)$. That is 
\begin{equation}
\label{identnorm2}
(P_\theta( S_1), P_\theta(S_2), \ldots ) \ed \frac{1}{\gamma(\theta)} (\gamma_1(\theta_1) , \gamma_2(\theta_2) , \ldots )
\end{equation}
where the 
$\gamma_i(\theta_i)$ are the independent gamma$(\theta_i)$ distributed increments  of the gamma subordinator over a partition of $[0,\theta]$
into disjoint intervals of lengths $\theta_i$.
Moreover, for each discrete distribution of $X := \sum_{i} a_i X_{p_i}$ as in  \eqref{mptheta}, 
there is the particular case of \eqref{mptheta}
\begin{equation}
\label{mptheta2}
M_{P(\theta)} \left( \sum_{i} a_i X_{p_i} \right) \ed \frac{1}{\gamma(\theta)} \sum_{i} a_i  \gamma_i(\theta\, p_i)
\end{equation}
where $P(\theta)$ is a random discrete distribution defined by any exhaustive listing of the normalized jumps $P_j(\theta)$ of a 
standard gamma subordinator up to time $\theta$.
\end{proposition}


\subsection{Finite Dirichlet means}
\label{sec:findir}
As a general remark, if the $X_i$ in a random average $\tilX:= \sum_i X_i P_i$ are either constants, or made so 
by conditioning, say $X_i = x_i$ for some bounded sequence of numbers $x_i$, then 
as $(x_i)$ ranges over bounded sequences, the collection of distributions of $\tilX$, or a suitable collection of moments or  transforms of those distributions,
provides an encoding of the joint distribution of random weights $P_i$. 
This approach works very nicely for the Dirichlet model:

\begin{proposition}
\label{vonwatson}
{\em [\citet{von1941distribution}, 
\citet{watson1956joint} 
]}
For each fixed sequence of non-negative coefficients $(x_1, \ldots, x_m)$ 
and $(P_1, \ldots, P_m)$ with Dirichlet $(\theta_1, \ldots, \theta_m)$ distribution with $\sum_{i=1}^m \theta_i = \theta$,
the distribution of the finite Dirichlet mean $\sum_{i=1}^m x_i P_i$ is uniquely determined by the following 
Laplace transform of $\gamma(\theta) \, \sum_{i=1}^m x_i P_i$, for $\gamma(\theta)$ with gamma$(\theta)$ distribution independent of $(P_1,\ldots, P_m)$:
\begin{equation}
\label{gammaeqlt}
\E  \exp \left( - \lambda \gamma(\theta) \sum_{i} x_i P_i   \right) = \E \left( 1 + \lambda \sum_{i} x_i P_i   \right)^{-\theta} = \prod_{i}  ( 1 +  \lambda x_i )^{-\theta_i}  .
\end{equation}
For $\lambda = 1$, with the left side regarded as the multivariate Laplace transform of the random vector 
$\gamma(\theta) (P_1, \ldots, P_m)$ with arguments $x_1, \ldots, x_m$, this formula uniquely characterizes the
Dirichlet $(\theta_1, \ldots, \theta_m)$ distribution of $(P_1, \ldots, P_m)$.
\end{proposition}
\begin{proof}
After multiplying both sides
of \eqref{gammaeqlt} by an independent $\gamma(\theta)$ variable,  the beta-gamma algebra
makes  the $P_i \gamma(\theta)$ a collection of independent gamma$(\theta_i)$ variables, hence
\begin{equation}
\label{gammath}
\gamma(\theta) \, \sum_{i} x_i P_i   =  \sum_{i} x_i \gamma_{i}(\theta_i) 
\end{equation}
for independent $\gamma_{i}(\theta_i)$ with sum $\gamma(\theta)$, as above.
Hence by taking Laplace transforms:
\begin{equation}
\label{gammalt1}
\E  \exp \left( - \lambda \sum_{i} x_i P_i \gamma(\theta)   \right) = \prod_{i} \E  \exp \left( - \lambda x_i \gamma _i(\theta_i)  \right).
\end{equation}
Condition on all the $P_i$, and integrate out the gamma variables using the Laplace transform \eqref{gamlt}, to obtain the two further expressions in \eqref{gammaeqlt}.
For each fixed choice of coefficients $x_i$, this formula determines the Laplace transform of $\gamma(\theta) \sum_{i} x_i P_i$,
hence the distribution of $\gamma(\theta) \sum_{i} x_i P_i$,  hence also the distribution of the finite Dirichlet mean $\sum_{i} x_i P_i$,
by Lemma \ref{lem:gammacancel}.
\end{proof}

The basic {\em Dirichlet mean transform} \eqref{gammaeqlt} has a long history, dating back to 
\citet{von1941distribution}, 
who gave a more complicated derivation in the case  of particular interest in mathematical statistics, with parameters $\theta_i = k_i/2$ for some positive integers
$k_i$ with $\sum_{i=1}^m k_i = k$ when 
$$
(P_i , 1 \le i \le m) \ed  (A_i, 1 \le i \le m)/A
$$
for a sequence of independent random variables $A_i \ed \chi^2_{k_i} \ed 2 \gamma(k_i/2)$ and $A:= \sum_{i=1}^m A_i \ed \chi^2_k \ed 2 \gamma(k/2)$,
where $\chi_k^2 \ed \sum_{i=1}^k Z_i^2$ for a sequence of i.i.d. standard Gaussian variables $Z_i$. So in this instance, 
which provided the original motivation for study of the finite Dirichlet  distribution in mathematical statistics
$\sum_{i} x_i P_i$ is the ratio of two dependent quadratic forms in a sequence of $k$  i.i.d. standard Gaussian variables.
As observed by Von Neumann, for half integer $\theta_i$, the basic beta-gamma algebra
behind the above formulas, especially the key independence \eqref{betagamindpt} of the Dirichlet distributed ratios and their gamma distributed denominator, 
follows from the symmetry of the joint distribution of the underlying Gaussian variables in $\reals^k$ with respect to orthonormal transformations.

\citet{watson1956joint} 
gave the simple general argument indicated above using beta-gamma algebra.
Watson also supposed each $\theta_j$ to be a multiple of $1/2$, but his argument generalizes immediately to general $\theta_i$
as above.
Watson indicated how the same method yields a transform of the joint law of any finite number of linear combinations of Dirichlet variables. 
Simply take $\lambda  = 1$ and $x_j = \sum_{i} t_i \sum_{j} x_{i,j} D_j$ 
in 
\eqref{gammaeqlt} 
to obtain a joint Laplace transform of $\sum_{i} \sum_{j} x_{i,j} D_j, 1 \le i \le m$ for any matrix of real coefficients  $x_{i,j}$ , $1 \le i \le m, 1 \le j \le k$.
This trick, of turning what looks at first like a univariate transform into a multivariate transform, has been rediscovered many times, 
often without recognizing that it can done so simply by a change of variables.
See also \citet{mauldon1959generalization}, 
\cite{MR0286230} 
\cite{diniz2002calculating}
for detailed studies of the distributions and joint distributions  of linear combinations of Dirichlet variables, motivated by applications 
to linear combinations of order statistics and their spacings.

The above proposition was formulated for a fixed sequence of coefficients $x_1, \ldots, x_m$.
But a corresponding result for random coefficients $(X_1, \ldots, X_m)$ follows immediately by conditioning:

\begin{corollary}
\label{crl:findir}
Let $(X_1, \ldots, X_m)$ be a sequence of random variables independent of $(P_1, \ldots, P_m)$ with Dirichlet $(\theta_1, \ldots, \theta_m)$ distribution with $\sum_{i=1}^m \theta_i = \theta$.
Then:
\begin{itemize}
\item  the distribution of the random Dirichlet mean $\sum_i X_i P_i$ 
is uniquely determined by the following 
Laplace transform: 
for $\gamma(\theta)$ independent of $(P_1,\ldots, P_m)$, and $\lambda \ge 0$
\begin{equation}
\label{gammaeqltgen}
\E  \exp \left( - \lambda \gamma(\theta) \Sigma_{i} X_i P_i   \right) = \E \left( 1 + \lambda \Sigma_{i} X_i P_i   \right)^{-\theta} = \E \prod_{i}  ( 1 +  \lambda X_i )^{-\theta_i}  .
\end{equation}
\item 
If the $X_i$ are independent,  this holds with $\E \prod_{i}$ replaced by $\prod_{i}\E$ in the rightmost expression. In particular, if the $X_i$ are i.i.d. copies of
$X$,  so $M_P(X):=  \sum_{i} X_i P_i$ is the $P$-mean of $X$ for this Dirichlet distribution of $P$, then
\begin{equation}
\label{gammaeqltind}
\E  \exp \left( - \lambda \gamma(\theta) M_P(X) \right) = \E \left( 1 + \lambda M_P(X)  \right)^{-\theta} = \prod_{i}  \E ( 1 +  \lambda X)^{-\theta_i}  .
\end{equation}
\item
As a special case, for $\tilX_{m||\theta}$ the $P$-mean of $X$ for $P = (P_1, \ldots, P_m)$ with the symmetric Dirichlet$(m||\theta)$ distribution with total weight $\theta$,
\begin{equation}
\label{gammaeqltsym}
\E  \exp \left( - \lambda \gamma(\theta) \tilX_{m||\theta} \right) = \E \left( 1 + \lambda \tilX_{m||\theta} \right)^{-\theta} = \left( \E ( 1 +  \lambda X)^{-\theta/m} \right)^m.
\end{equation}
\end{itemize}
\end{corollary}

To illustrate the basic transform \eqref{gammaeqltsym} of the distribution of a symmetric Dirichlet mean, observe that for $a,b>0$ the beta$(a,b)$ distribution is characterized by
\begin{equation}
X \ed \beta_{a,b} \qquad \iff \qquad \E ( 1 - \lambda X)^{-(a+b)} = (1-\lambda)^{-a} .
\end{equation}
Hence easily from \eqref{gammaeqltsym},
\begin{equation}
\label{symmdirab}
X \ed \beta_{a,b} \qquad \iff \qquad \tilX_{m||m(a+b)} \ed \beta_{ma, mb}.
\end{equation}
In the particular case $a = b = \hf$, for the symmetric Dirichlet$(m||m)$  mean of i.i.d. copies of $X$ with the arcsine distribution of $\beta_{1/2,1/2}$,
the implication $\Rightarrow$ in \eqref{symmdirab} was established 
in \cite{MR3250962} 
by a more difficult argument involving Stieltjes transforms. See
also  \citet{MR3557549}  
where the same case is derived by moment calculations, involving the instance for Dirichlet$(m||m)$ of the general moment formula \eqref{moments} for $P$-means.

To illustrate \eqref{symmdirab} for $0 < p < 1$ and $q:= 1-p$, if a unit interval is cut into $m$ segments by $m-1$ independent uniform cut points,
and a beta$(p,q)$-distributed fraction of each segment is painted red, independently from one segment to the next, then the total
length of red segments has beta$(mp,mq)$ distribution. 

\subsection{Infinite Dirichlet means}

The extension of the basic transforms of Corollary \ref{crl:findir} from finite to infinite Dirichlet means is surprisingly easy:

\begin{corollary}
\label{crl:inftheta}
{\em [Infinite Dirichlet mean transform: \citet*{MR1041402}]} 
For every non-negative random variable $X$, and $P_{0,\theta}$ the random discrete distribution derived from the normalized jumps of standard gamma process on $[0,\theta]$, the distribution of 
the distribution of the $P_{0,\theta}$-mean $\tilX_{0,\theta}$ of $X$ is uniquely determined by the Laplace transform 
of $\gamma(\theta) \tilX_{0,\theta}$, for $\gamma(\theta)$ independent of $\tilX_{0,\theta}$, according to the formula for $\lambda > 0$
\begin{equation}
\label{dirlog}
\E  \exp \left( - \lambda \gamma(\theta) \tilX_{0,\theta} \right) = \E ( 1 + \lambda \tilX_{0,\theta} )^{-\theta } = \exp [ - \theta \E \log ( 1 + \lambda X ) ]   .
\end{equation}
For unbounded $X \ge 0$, this formula should be read with the convention  $(1 + \lambda \infty)^{-\theta} = e^{- \infty}  = 0$, implying 
\begin{equation}
\label{feigintweedie}
\qquad \P(\tilX_{0,\theta} < \infty ) = 1 \mbox{ or } 0   \mbox{ according as } \E \log ( 1 + X ) < \infty \mbox{ or } = \infty.
\end{equation}
\end{corollary}
\begin{proof}
Suppose first that $X$ is a simple random variable $X = \sum_{i=1}^m x_i X_{p_i}$ for Bernoulli$(p_i)$ indicators $X_{p_i}$ of $m$ disjoint events with probabilities $p_i = \theta_i/\theta$.
Proposition \ref{fergdir} gives $\tilX_{0,\theta} \ed \sum_{i} x_i P_i$ for $(P_1, \ldots, P_m)$ with the finite Dirichlet distribution with parameters 
$(\theta p_i, 1 \le i \le m)$. So Proposition \ref{vonwatson} gives
\begin{align}
\nonumber \E  \exp \left( - \lambda \gamma(\theta) \tilX_{0,\theta} \right) &= \E \left( 1 + \lambda \sum_{i} x_i P_i \right)^{- \theta} \\
\nonumber &= \prod_{i} ( 1 + \lambda x_i )^{- p_i \theta} \\
\nonumber &= \exp \left( - \theta \sum_{i} p_i \log ( 1 + \lambda x_i ) \right) \\
\nonumber &= \exp \left( - \theta \E \log ( 1 + \lambda X ) \right)  .
\end{align}
This is \eqref{dirlog} for simple non-negative $X$. The case of general $X \ge 0$ follows by taking simple $X_n$ with $0 \le X_n \uparrow X$
and appealing to  the monotone convergence theorem for $P$-means \eqref{monconv}.
\end{proof}

\begin{corollary}
\label{crl:ft}
{\em 
\citep{feigin1989linear} 
}
For  a general distribution of $X$, for each fixed $\theta >0$ the $(0,\theta)$ mean 
$\tilX_{0,\theta}$ of $X$
is well defined by almost sure absolute convergence iff $\E \log (1 + |X| ) < \infty$.
\end{corollary}

See also \citet{MR2932413} 
for a nice proof of this result without use of transforms.
The problem of inverting the transform \eqref{dirlog} to obtain more explicit formulas for the distribution of a $(0,\theta)$ mean $\tilX_{0,\theta}$
has attracted a great deal of attention. 
One of the first appearances of the right side of formula \eqref{dirlog} in connection with the distribution of a $(0,\theta)$ mean $\tilX_{0,\theta}$ is
in \citet[Theorem 2.5]{MR630318}, 
where for $X$ with $\E|X| < \infty$ it is shown that for each real $x$ the formula 
\begin{equation}
\label{hannumetal}
\phi_{T^x} (t):= \exp ( - \theta \E  \log [ 1 - i t (X -x) ] ) \qquad (t \in \reals)
\end{equation}
with
\begin{equation}
\mbox{ $\log[ 1 + i v] := \log \sqrt{1 + v^2} + i \xi $ for $\xi = \arctan v \in (- \pi, \pi )$, }
\end{equation}
defines the characteristic function of a random variable $T_x$, which is a limit in distribution of a linear combination of independent gamma
variables with suitable Dirichlet distributed weights. Provided $\P(X = x) < 1$ the distribution of $T^x$ is continuous, and such that
\begin{equation}
\label{hannumid}
\P( \tilX_{0,\theta}  \le x ) = \P( T_x  \le 0 ) .
\end{equation}
The c.d.f. of $\tilX_{0,\theta}$ is therefore determined by inversion of the characteristic function \eqref{hannumetal}. Something missing in 
this discussion of \citet{MR630318} 
identification 
\begin{equation}
\label{txdef}
T^x =  \gamma(\theta ) ( \tilX_{0,\theta} - x ) \mbox{  for $\gamma(\theta)$ independent of $\tilX_{0,\theta}$} 
\end{equation}
which is evident by inspection of formula \eqref{dirlog} for $\lambda = - i t $.
This observation makes both the identity \eqref{hannumid} and the continuity of the distribution of $T^x$ completely obvious. It is also clear from 
Corollary \ref{crl:ft} that this description  of the distribution of $\tilX_{0,\theta}$ is valid for any $X$ with $\E \log ( 1  + |X| ) < \infty$.
Closely related generalized Stieltjes transforms of the distribution of $\tilX_{0,\theta}$ appear also in
\citet{MR1041402}, 
with references to earlier work by those authors. For a later treatment with further references, and explicit inversion formulas for the
density of $\tilX_{0,\theta}$, see
\cite[Proposition 2]{MR1936323} 
which is a Fourier variant of Corollary \ref{crl:inftheta}, with subsequent analysis involving \eqref{hannumid} and inversion of the Fourier transform
\eqref{hannumetal}. Surprisingly, none of the above references mention the simple interpretation \eqref{txdef} of $T^x$.

\subsection{The two-parameter model}
\label{sec:twoparam}
As recalled in Section \ref{sec:twoparamintro}, following the initial development of the basic infinite Dirichlet model 
with a single parameter $\theta$ by Fisher (who used $\alpha$ instead of $\theta$ for the parameter),
subsequent work of McCloskey,  Ewens, Ferguson and Engen,
and the work of L\'evy, Lamperti, Dynkin and others on last exit times and occupation times of various stochastic
processes related to the stable subordinator of index $\alpha \in (0,1)$,
\citet*{MR1156448} 
developed the two-parameter extension of these basic models for random discrete distributions. The partition structure of this
$(\alpha,\theta)$ model was described by \citet{MR1337249}, 
following which 
\citet*{MR1434129} 
gave an account of the corresponding ranked discrete distributions, and 
 \citet{MR1691650} 
characterized the distributions of $P_{\alpha,\theta}$-means for the complete range of parameters $(\alpha,\theta)$.
The $(\alpha,\theta)$ model is most easily described by a residual allocation model  
\eqref{stickbreak} 
for generating
its  size-biased permutation $P^*$, commonly known as the GEM$(\alpha,\theta)$ distribution.
This is obtained by the particular choice of distributions for independent factors $H_i$ with
\begin{equation}
\label{gembreaks}
H_i \ed \beta_{1-\alpha, \theta + \alpha i }  \qquad (i = 1, 2, \ldots) .
\end{equation}
The corresponding EPPF is known to be
\begin{equation}
\label{altheppf}
p_{\alpha,\theta}(n_1, \ldots, n_k) := \frac { \left( \prod_{i = 1}^{k-1} (\theta + i \alpha ) \right) \prod_{i=1}^k (1 - \alpha)_{n_i - 1} }{ (\theta + 1)_{n-1} } .
\end{equation}
It is easily shown that this EPPF corresponds to the above choice of beta distributed factors in the residual allocation model, and
that this choice leads to a well defined random discrete distribution $P$ iff one of following three cases obtains.
See \citet[\S 3.1]{MR2245368} 
for details and references to original sources.

\noindent 
\bitem GEM$(-\theta/m,\theta)= $ {\em size-biased Dirichlet$(m||\theta)$.}
This is the case $\alpha = -\theta/m < 0$ for some positive integer $m$ and $\theta > 0$, with the convention $P_j = H_j = 0$ for $j >m$.
This distribution of $(P_1, \ldots, P_m)$ is the size-biased random permutation of the symmetric Dirichlet$(m||\theta)$ model.

\noindent 
\bitem GEM$(0,\theta) = $ {\em size-biased Dirichlet$(\infty||\theta)$.}
This is the case $\alpha = 0$ and $\theta \ge 0$, which is the weak limit of the Dirichlet$(m||\theta)$ model as $m \to \infty$.
In this model, $P_j >0$ a.s. for all $j$ if $\theta >0$.
Statistical aspects of this limit process 
were first  considered by
\citet{fisher1943solo}. 
As first shown by McCloskey, the GEM$(0,\theta)$ model is the size-biased ordering of relative sizes of jumps of the standard gamma process on $[0,\theta]$,
relative to their gamma$(\theta)$ distributed total.
This is also the size-biased distribution of atom sizes of any Dirichlet random measure governed by a continuous measure with total weight $\theta$.
The corresponding partition structure is governed by the Ewens sampling formula.

\noindent 
\bitem GEM$(\alpha, \theta) = $ {\em size-biased stable $(\alpha, \theta)$ model derived from a stable$(\alpha)$ subordinator}. This is the case $0 < \alpha < 1$ and $\theta > - \alpha$, with $P_j >0$ a.s. for all $j$.
This case has special subcases as follows.

\begin{itemize} 
\item $(\alpha,0)$. This model with $\theta = 0$ is the size-biased ordering of relative sizes of jumps of a stable process of index $\alpha$ on $[0,s]$,
for any fixed time $s$. Equivalently in distribution, an interval partition of $[0,1]$ may be created by the collection of maximal open intervals in the complement of
the range of the stable subordinator, relative to $[0,1]$. Then the GEM$(\alpha,0)$ distributed $(P_j)$ may be obtained either as a size-biased ordering of the lengths of these intervals,
or by letting $P_1$ be the last (meander) interval with right end $1$, and size-biasing the order of the rest of the intervals.

\item $(\alpha,\alpha$). This case with $\theta = \alpha \in (0,1)$, is derived from the previous construction by conditioning the stable subordinator to hit the point $1$.
So there is no last interval, rather an exchangeable interval partition, whose lengths in size-biased order are GEM$(\alpha,\alpha)$.
Equivalently, this is the sequence of lengths of  excursions, in size-biased random order, for the excursions of a Bessel bridge of dimension $(2-2 \alpha)$ from
$(0,0)$ to $(1,0)$.

\item $(\alpha, m \alpha)$ for $m = 1,2, \ldots$. This model is obtained from the $(\alpha, 0)$ model by deleting the first $m$ values $P_j, 1 \le j \le m$, and renormalizing the
residual values $(P_{m+1}, P_{m+2}, \ldots)$ by their sum $1 - \sum_{i=1}^m P_i$. Or, by the same scheme, starting from the $(\alpha, \alpha)$ model associated with
the excursions of a Bessel bridge of dimension $(2-2 \alpha)$ after deleting the first $m-1$ values $P_j, 1 \le j \le m-1$, and renormalizing the residual values.

\item $(\alpha, \theta)$ for $\theta >0$.  This model model can be obtained by first splitting $[0,1]$ into subintervals by GEM$(0,\theta)$, that is 
by i.i.d. beta$(1,\theta)$ stick-breaking, then splitting each of these subintervals independently according to GEM$(\alpha, 0)$.  The result is an $(\alpha,\theta)$ interval partition of $[0,1]$, meaning that the interval lengths in size-biased order form a GEM$(\alpha,\theta)$.

\item $(\alpha, \theta)$ with $-\alpha < \theta < 0$ there is no known construction of GEM$(\alpha,\theta)$ of a comparable kind.

\item $(\alpha, \theta)$ for general $0 < \alpha < 1$ and $\theta > - \alpha$. The GEM$(\alpha,\theta)$ model for generating $P$, and a random sample from $P$ from which the partition structure is created, is absolutely continuous relative to the GEM$(\alpha,0)$ model, with density factor $c_{\alpha, \theta} S_\alpha ^{\theta/\alpha}$, where $S_\alpha$, the {\em $\alpha$-diversity of $P$}, is the almost sure limit of $K_n/n^\alpha$ as $n \to \infty$ for $K_n$ the number of distinct elements in a sample of size $n$ from $P$, and $c_{\alpha,\theta} := \Gamma(1 + \theta )/\Gamma ( 1 + \theta/\alpha)$ is a normalization constant.  So if $\E_{\alpha,\theta}$ is the expectation operator governing $P$ as a GEM$(\alpha,\theta)$, and a sample $(J_1, J_2, \ldots)$ from $P$, then
for every non-negative random variable $Y$ which is  a measurable function of $P$ and the sample $(J_1, J_2, \ldots )$ from $P$:
\begin{equation}
\label{rnd}
\E_{\alpha, \theta} Y = c_{\alpha,\theta} \E_{\alpha, 0} Y   S_\alpha ^{\theta/\alpha}
\end{equation}

\end{itemize}

In the 1990's, this $(\alpha,\theta)$ model for a random discrete distribution $P$,  and its associated partition structures and  $P$-means, were extensively studied in a series of articles cited in Section 
\ref{sec:twoparam}.
Since around 2000, the merits of this $(\alpha,\theta)$ model for a random discrete distribution $\Pbul$ have been widely acknowledged, and there is by now a substantial literature of developments and applications of this 
model in various contexts, as mentioned in the introduction.

\subsection{Two-parameter means}
\label{sec:twomeans}

Looking at the general moment formula for $P$-means \eqref{moments}, it is evident that this formula will simplify greatly if the
EPPF factors as
\begin{equation}
\label{prodform}
p(n_1, \ldots, n_k) = \frac{v(k)}{c(n)} \prod_{i=1}^k w( n_i)
\end{equation}
for some pair of weight sequences $v(k), k = 1,2, \ldots$ and $w(m), m = 1,2, \ldots $.
For then by \eqref{nexnform} the corresponding ECPF factors as
\begin{equation}
\label{prodformex}
\pex(n_1, \ldots, n_k) = \frac{v(k)/k!}{c(n)/n!} \prod_{i=1}^k w( n_i)/n_i!
\end{equation}
It was shown by \citet{MR2160323} 
that apart from some degenerate limit cases, the only EPPFs of the form 
\eqref{prodform}, defined for all positive integer compositions and subject to the consistency constraint
\eqref{consistnnm}
for all $n$, are those in displayed in \eqref{altheppf},
corresponding to a random discrete distribution $P$ whose size-biased presentation follows the GEM$(\alpha,\theta)$ residual allocation model \eqref{gembreaks}. 
Assuming that \eqref{prodformex} is an EPPF, which we know is possible
for suitable choices of weights $v(k), w(n) $ and $c(n)$, 
the general moment formula \eqref{moments} reduces easily to the identity 
\begin{equation}
\label{coeffsn}
\frac{c(n)}{n!} \E(\tilX ^n) =  [\lambda^n]  \sum_{k=1}^\infty \frac{ v(k) }{k!} \left( \sum_{m = 1}^\infty \frac{ w(m) }{m!} \E (\lambda Y)^m  \right)^k .
\end{equation}
Introducing the generating functions 
$$
C(t):= 1 + \sum_{n=1}^\infty \frac{c(n)}{n!} t^n ; \qquad
V(s):= 1 + \sum_{k=1}^\infty \frac{v(k)}{k!} s^k;  \qquad
W(t):= \sum_{m=1}^\infty \frac{w(m)}{m!} t^m ,
$$
formula \eqref{coeffsn} is the identity of coefficients of $\lambda^n$ in
\begin{equation}
\label{composmom}
\E C( \lambda \tilX )  = V ( \E  W ( \lambda X ) ) 
\end{equation}
which for $\tilX = X = 1$ gives
\begin{equation}
\label{composc}
C( \lambda )  = V\circ W(\lambda):= V ( W ( \lambda ) ).
\end{equation}
Thus the general formula \eqref{moments} for moments of $P$-means has the following corollary.
\begin{corollary}
\label{crl:compos}
{\em [Composite moment formula for $(\alpha,\theta)$-means;  \citet{MR1691650}]}. 
For any presentation of an $(\alpha,\theta)$ EPPF in the product form \eqref{prodform} for some sequences of weights $v(k)$ and $w(n)$ with exponential generating
functions $V$ and $W$ as above, these generating functions are convergent in some neighborhood of the origin, and for each bounded random variable $X$ the distribution of the
$(\alpha,\theta)$-mean $\tilX$ is the 
unique distribution whose positive integer moments are determined by the identity of formal power series in $\lambda$
\begin{equation}
\label{composv}
\E [ V \circ W (\lambda \tilX )]   = V ( \E W ( \lambda X ) ).
\end{equation}
\end{corollary}
To check the claim of convergence of the generating functions, it seems necessary to check case by case as below. But this composite moment
formula for $(\alpha,\theta)$-means provides a remarkable unification of a number of different formulas that were first discovered in the
special cases listed below.
This composite moment formula for $P$-means
is a variation of the {\em compositional} or {\em Fa\`a di Bruno} formula,  
which shows how the coefficients $c(n)$ of the composite function $C(\lambda) = V\circ W(\lambda)$ are determined the two weight sequences $v(k)$ and $w(m)$.
See \citet[\S 1.2]{MR2245368}. 
Consider the product $\pi(n_1, \ldots, n_k):= v(k) \prod_{i=1}^k w(n_i)$ appearing in \eqref{prodform}, without the factor of $c(n)$ in the denominator.
Starting from any two sequences of weights $v(k)$ and $w(m)$ such that this product is non-negative for all $(n_1, \ldots, n_k)$, the
compositional formula \eqref{composc} determines the sequence of non-negative coefficients $c(n)$ that is necessary to make $p(n_1, \ldots, n_k):= \pi(n_1, \ldots, n_k)/c(n)$ the EPPF of some 
exchangeable random partition $\Pi_n$ of $[n]$ for each $n$. However, for these $\Pi_n$ to be derived by sampling from some random discrete distribution $P$,
 it is necessary that they be {\em consistent} as $n$ varies in the sense of \eqref{consistnnm}, and it is this consistency requirement that limits the scope of application
of the composite moment formula to the $(\alpha,\theta)$ model.

The simplest algebraic form of the $(\alpha,\theta)$ EPPF 
\eqref{altheppf} is obtained for $\alpha \ne 0$ and $\theta \ne 0$ by writing it as
\begin{equation}
\label{altheppfgeneric}
p_{\alpha,\theta}(n_1, \ldots, n_k) := \frac{(-1)^k (\theta/\alpha)_k  }{ ( \theta )_n } \prod_{i=1}^k (-\alpha)_{n_i} \qquad (\alpha \ne 0, \theta \ne 0 ) .
\end{equation}
which allows the product form 
\eqref{prodform} to be achieved by what appears to be the simplest possible choice of weights, that is 
\begin{align}
w(m) &= (- \alpha)_m := \prod_{i=0}^{n-1} (i - \alpha) = (-1)^m m! \binom{\alpha}{m} \\
v(k) &= (-1)^k (\theta/\alpha)_k  = (-1)^k k! \binom{- \theta/\alpha}{k} \\
c(n) &= (\theta)_n
\end{align}
The corresponding exponential generating functions then all simplify by negative binomial expansions:
\begin{align}
W(t) &= \sum_{m=1}^\infty \frac{ (- \alpha)_m }{m!} t^m = (1 - t)^\alpha  -1 \\
V(s) &= \sum_{k=1}^\infty \frac{ (\theta/\alpha)_k  }{k!}  s^k = ( 1 + s)^{- \theta/\alpha} \\
C(t) &= \sum_{n=0}^\infty \frac{( \theta)_n }{n! } t^n  = (1-t)^{-\theta}
\end{align}
which magically combine as they must according to the composite formmula
\eqref{composc}:
$$
V(W(t)) = ( 1 + ( 1 - t)^{\alpha} - 1 )^{- \theta/\alpha} = (1-t)^{-\theta} = C(t).
$$
This argument simplifies a similar argument due to \citet{MR1691650}, 
by working consistently with compositions rather than partitions of $n$.
A puzzling feature of the argument is that for $0 < \alpha < 1$, there is  no obvious interpretation of the weight sequence $w(m) = (- \alpha)_m $ in probabilisitic or combinatorial
terms, due to negativity of the weight for $m = 1$. This is compensated by the alternating sign in the definition of $v(k)$, which ensures that the product 
\eqref{prodform} is positive, as it must be for all compositions of positive integers $(n_1, \ldots, n_k)$. Still, the result of this algebraically simple calculation is
a remarkable unified formula for what appear at first to be extremely different cases of the $(\alth)$ model, that is the elementary 
symmetric Dirichlet $(m || \theta)$ case with  only a finite number $m$ of positive $P_i$, and the fat tailed $(\alpha, \theta)$ models for $0 < \alpha < 1$.

\begin{corollary}
\label{crl:compos2}
{\em [generic $(\alpha,\theta)$ Cauchy-Stieltjes transform; \citet{MR1691650}]}. 
Suppose that either $\alpha = - \theta/m$ for some $m = 1,2, \ldots$, or $0 < \alpha < 1$ and $\theta > - \alpha$ with $\theta \ne 0$. Then
for any distribution of $X \ge 0$, the distribution of $\tilX_\alth$, the $(\alpha,\theta)$-mean of $X$, is uniquely determined by the formula
\begin{equation}
\label{composmomalth}
\E ( 1 + \lambda \tilX_\alth  )^{-\theta } = \left(  \E ( 1 + \lambda X )^\alpha \right)^{- \thoval}  \qquad (\alpha \ne 0, \theta \ne 0, \lambda \ge 0).
\end{equation}
Also, for $\alpha \ne 0, \theta \ne 0$ and all $X$ with $\E|X|^n < \infty$ for some $n = 1,2, \ldots$ the $n$th moment of $\tilX_\alth$ is well defined, and given by
the equality of coefficients of $\lambda^n$ in the formal power series
\begin{equation}
\label{composmomalth1}
\frac{ ( \theta )_n }{n!} \E \tilX_\alth ^n  = [\lambda^n]  \sum_{j=1}^\infty \frac{ (\theta/\alpha)_j \alpha ^j }{j!} \left( \sum_{\ell=1}^ \infty \frac{ (1-\alpha)_{\ell-1} \lambda^\ell \E(X^\ell) } {\ell! } \right)^j .
\end{equation}
And for $0 < \alpha < 1$ and arbitrary $\theta > - \alpha$
\begin{itemize}
\item $\tilX_\alth$ is finite with probability one for all $\theta > - \alpha$ if $\E X^\alpha < \infty$; 
\item $\tilX_\alth$ is infinite with probability one for all $\theta > - \alpha$ if $\E X^\alpha = \infty$.
\end{itemize} 
\end{corollary}
\begin{proof}
Formula \eqref{composmomalth} is read from Corollary \ref{crl:compos}, in the first instance for bounded $X$, when the convergence of 
all power series is easily justified. The formula then extends to unbounded $X \ge 0$ by monotone convergence,
using the consequence of Proposition \ref{prp:canon} that $P$-means $\tilX$ and $\tilY$ of $X$ and $Y$ with $0 \le X \le Y$
can always be constructed as $\tilX = X_J \le \tilY = Y_J$ for $(X_i,Y_i)$ a sequence of i.i.d. copies of $(X,Y)$.
It follows easily that if $\E|X|^n < \infty$ for some $n = 1,2, \ldots$ then the  $n$th moment of $\tilX_\alth$ is well defined, and can be evaluated as indicated
by equating coefficients in the formal power series. The conclusions regarding finiteness of $\tilX_\alth$ follow similarly by monotone approximation, in the first instance for
And for $0 < \alpha < 1$ and $\theta > - \alpha$ with $\theta \ne 0$, then also for $\theta = 0$ by the result of 
\citet*{MR1434129} 
that for each fixed $0 < \alpha < 1$ the laws of GEM$(\alth)$ distributions are mutually absolutely continuous as $\theta$ varies.
\end{proof}

Two checks on formula \eqref{composmomalth} are provided as follows.
One check is the finite symmetric Dirichlet $(\theta)$ case with $\theta >0$ and $\alpha = -\theta/m$ for some $m = 1, 2, \ldots$,
when \eqref{composmomalth} reduces to the symmetric Dirichlet mean transform \eqref{gammaeqltsym}.
Another check 
is provided by the case $\alpha = \theta$, when for simple $X$ it reduces to a formula of 
\citet{MR1022918}. 
The infinite  Dirichlet mean transform \eqref{dirlog} is the limit case  for fixed $\theta$ and $\alpha = -\theta/m \uparrow 0$ as $m \to \infty$,
as already indicated around \eqref{dirlog}.
Next, the limit case for $0 < \alpha < 1 ,  \theta = 0$:

\begin{corollary}
For $0 < \alpha <1$ and $X \ge 0$, if $\E X^\alpha < \infty$ then the distribution of $\tilX_{\alpha,0}$ is determined by the transform
\begin{equation}
\label{composlogal}
\E \log ( 1 + \lambda \tilX_{\alpha,0}  )  = \frac{1}{\alpha} \log \left(   \E( 1 + \lambda X)^ \alpha  \right) \qquad ( 0 < \alpha < 1, \lambda \ge 0)
\end{equation}
which admits the alternative form
\begin{equation}
\label{composmomalz}
\E ( 1 + \lambda \tilX_{\alpha,0}  )^{-1} =  \frac{ \E ( 1 + \lambda X)^{\alpha - 1} } { \E ( 1 + \lambda X)^{\alpha} } \qquad (0 < \alpha < 1, \lambda \ge 0)
\end{equation}
\end{corollary}
Observe that \eqref{composmomalz} for $X = X_p$ the indicator of an event of probability $p$ reduces to Lamperti's Stieltjes transform \eqref{lampertistieltjes}
for the generalized arcsine law with probability density \eqref{darling-lamperti}.
The case of \eqref{composmomalz} for simple $X$ is due to \citet{MR1022918}, while 
while \eqref{composlogal} was first indicated by \citet{MR1691650}.  
For simple $X$, each of \eqref{composlogal} and \eqref{composmomalz} follows easily from the other, by differentiation or integration of the power series.
These formulas for general $X \ge 0$ are obtained by increasing approximation with simple $X$, as in the proof of Corollary \ref{crl:compos2}.



\section*{Acknowledgement}
Thanks to Lancelot James,  Wenpin Tang, Zhiyi You and Teddy Zhu for careful readings of earlier versions of this article, and pointers to the literature.


\def\cprime{$'$} \def\cprime{$'$} \def\cprime{$'$} \def\cprime{$'$}

\end{document}